\documentclass[11pt]{aart}

\usepackage{titlesec}
\titleformat{\subsection}[runin]{\normalfont\bfseries}{\thesubsection.}{.5em}{}[.]\titlespacing{\subsection}{0pt}{2ex plus .1ex minus .2ex}{.8em}
\titleformat{\subsubsection}[runin]{\normalfont\itshape}{\thesubsubsection.}{.3em}{}[.]\titlespacing{\subsubsection}{0pt}{1ex plus .1ex minus .2ex}{.5em}
\titleformat{\paragraph}[runin]{\normalfont\itshape}{\theparagraph.}{.3em}{}[.]\titlespacing{\paragraph}{0pt}{1ex plus .1ex minus .2ex}{.5em}

\usepackage[labelfont=sc,font=small,labelsep=period]{caption}
\makeatletter\def\SetFigFont#1#2#3#4#5{\small}

\usepackage[letterpaper, hmargin=1.1in, top=1in, bottom=1.2in, footskip=0.6in]{geometry}

\usepackage{amsmath} 
\usepackage{amssymb}
\usepackage{amsfonts}
\usepackage{latexsym}
\usepackage{amsthm}
\usepackage{amsxtra}
\usepackage{amscd}
\usepackage{mathrsfs}
\usepackage{bm}

\usepackage{graphicx, color}
\usepackage[nottoc,notlof,notlot]{tocbibind} 
\usepackage{cite} 

\numberwithin{equation}{section}
\numberwithin{figure}{section}

\theoremstyle{plain} 
\newtheorem{theorem}{Theorem}[section]
\newtheorem*{theorem*}{Theorem}
\newtheorem{lemma}[theorem]{Lemma}
\newtheorem*{lemma*}{Lemma}
\newtheorem{corollary}[theorem]{Corollary}
\newtheorem*{corollary*}{Corollary}
\newtheorem{proposition}[theorem]{Proposition}
\newtheorem*{proposition*}{Proposition}
\newtheorem{definition}[theorem]{Definition}
\newtheorem*{definition*}{Definition}

\newtheorem*{conjecture*}{Conjecture}

\theoremstyle{definition} 

\newtheorem*{assumption*}{Assumption}

\newtheorem*{example*}{Example}
\newtheorem{remark}[theorem]{Remark}
\newtheorem*{remark*}{Remark}

\newcommand{\f}[1]{\boldsymbol{\mathrm{#1}}} 
\renewcommand{\cal}{\mathcal} 
 
\newcommand{\fra}{\mathfrak} 
\newcommand{\ul}[1]{\underline{#1} \!\,} 

\newcommand{\wt}{\widetilde}

\definecolor{darkred}{rgb}{0.9,0,0.3}
\definecolor{darkblue}{rgb}{0,0.3,0.9}
\definecolor{darkgreen}{rgb}{0.1,0.65,0.1}

\usepackage{ifthen}
\def\comment#1{\ifthenelse{\isodd{\value{page}}}{\marginpar{\raggedright\scriptsize{\textcolor{darkred}{#1}}}}{\marginpar{\raggedleft\scriptsize{\textcolor{darkred}{#1}}}}}

\renewcommand{\P}{\mathbb{P}}
\newcommand{\E}{\mathbb{E}}
\newcommand{\R}{\mathbb{R}}
\newcommand{\C}{\mathbb{C}}
\newcommand{\N}{\mathbb{N}}
\newcommand{\Z}{\mathbb{Z}}

\newcommand{\ee}{\mathrm{e}}
\newcommand{\ii}{\mathrm{i}}
\newcommand{\dd}{\mathrm{d}}
\newcommand{\col}{\mathrel{\vcenter{\baselineskip0.75ex \lineskiplimit0pt \hbox{.}\hbox{.}}}}
\newcommand*{\deq}{\mathrel{\vcenter{\baselineskip0.65ex \lineskiplimit0pt \hbox{.}\hbox{.}}}=}

\newcommand{\eqdist}{\overset{\text{d}}{=}}

\renewcommand{\leq}{\leqslant}
\renewcommand{\geq}{\geqslant}
\renewcommand{\epsilon}{\varepsilon}

\newcommand{\qq}[1]{[\![{#1}]\!]}

\newcommand{\ind}[1]{\f 1 (#1)}
\newcommand{\indb}[1]{\f 1 \pb{#1}}

\newcommand{\p}[1]{({#1})}
\newcommand{\pb}[1]{\bigl({#1}\bigr)}
\newcommand{\pB}[1]{\Bigl({#1}\Bigr)}
\newcommand{\pbb}[1]{\biggl({#1}\biggr)}
\newcommand{\pBB}[1]{\Biggl({#1}\Biggr)}
\newcommand{\pa}[1]{\left({#1}\right)}

\newcommand{\qb}[1]{\bigl[{#1}\bigr]}
\newcommand{\qB}[1]{\Bigl[{#1}\Bigr]}
\newcommand{\qbb}[1]{\biggl[{#1}\biggr]}
\newcommand{\qBB}[1]{\Biggl[{#1}\Biggr]}

\newcommand{\h}[1]{\{{#1}\}}
\newcommand{\hb}[1]{\bigl\{{#1}\bigr\}}
\newcommand{\hB}[1]{\Bigl\{{#1}\Bigr\}}
\newcommand{\hbb}[1]{\biggl\{{#1}\biggr\}}

\newcommand{\abs}[1]{\lvert #1 \rvert}
\newcommand{\absb}[1]{\bigl\lvert #1 \bigr\rvert}

\newcommand{\absbb}[1]{\biggl\lvert #1 \biggr\rvert}

\newcommand{\norm}[1]{\lVert #1 \rVert}

\newcommand{\normBB}[1]{\Biggl\lVert #1 \Biggr\rVert}

\newcommand{\scalar}[2]{\langle{#1} \mspace{2mu}, {#2}\rangle}

\DeclareMathOperator{\re}{Re}
\DeclareMathOperator{\im}{Im}

\DeclareMathOperator{\dist}{dist}

\DeclareMathOperator{\arcsinh}{arcsinh}

\newcommand{\ddp}[2]{\frac{\partial #1}{\partial #2}}

\newcommand{\sigG}{{\mathcal G}}
\newcommand{\sigF}{{\mathcal F}}

\newcommand{\Ei}{\E^{[i]}}
\newcommand{\Eu}{\E_{\sigF_0}}
\newcommand{\Pmu}{\P_{\sigG_\mu}}
\newcommand{\Emu}{\E_{\sigG_\mu}}
\newcommand{\Emup}{\E_{\sigG_{\mu+1}}}
\newcommand{\Emuc}{}
\newcommand{\Emumuc}{\Emu}
\newcommand{\EFmu}{\E_{\sigF_\mu}}
\newcommand{\EFmup}{\E_{\sigF_{\mu+1}}}

\renewcommand{\u}{u}
\newcommand{\umu}{u_\mu}
\newcommand{\amu}{a_\mu}
\newcommand{\bmu}{b_\mu}
\newcommand{\umuc}{\tilde u_\mu}
\newcommand{\amuc}{\tilde a_\mu}

\newcommand{\uc}{\tilde u}
\newcommand{\ac}{\tilde a}
\newcommand{\bc}{\tilde b}
\newcommand{\ua}{\underline{a}}
\newcommand{\uac}{\underline{\ac}}
\newcommand{\ub}{\underline{b}}
\newcommand{\ubc}{\underline{\bc}}
\newcommand{\ur}{\underline{r}}

\newcommand{\Gc}{\tilde G}
\newcommand{\Ac}{\tilde A}

\newcommand{\Acmu}{\tilde A^\mu}
\newcommand{\Gcmu}{\tilde G^\mu}
\newcommand{\Xcmu}{\tilde X^\mu}

\newcommand{\Amu}{A}
\newcommand{\Gmu}{G}
\newcommand{\Xmu}{X}

\newcommand{\Gammamu}{\Gamma_\mu}

\newcommand{\goodev}{\Xi}
\newcommand{\badev}{\Xi^c}
\newcommand{\indgoodmu}{\phi_\mu}

\newcommand{\indgood}{\phi}
\newcommand{\indbad}{\bar\phi}
\newcommand{\chimu}{\chi_\mu}

\begin{document}
\title{Local semicircle law for random regular graphs}
\author{Roland Bauerschmidt\footnote{Harvard University, Department of Mathematics. E-mail: {\tt brt@math.harvard.edu}.} \and
Antti Knowles\footnote{ETH Z\"urich, Departement Mathematik. E-mail: {\tt knowles@math.ethz.ch}.} \and
Horng-Tzer Yau\footnote{Harvard University, Department of Mathematics. E-mail: {\tt htyau@math.harvard.edu}.}}
\date{February 12, 2017}
\maketitle

\begin{abstract}
  We consider random $d$-regular graphs on $N$ vertices, with degree $d$ at least $(\log N)^4$.
  We prove that the Green's function of the adjacency matrix and the
  Stieltjes transform of its empirical spectral measure are well approximated by Wigner's semicircle
  law, down to the optimal scale given by the typical eigenvalue spacing (up to a logarithmic correction).
  Aside from well-known consequences for the local eigenvalue distribution,
  this result implies the complete (isotropic) delocalization of all eigenvectors and a
  probabilistic version of quantum unique ergodicity.
\end{abstract}


\section{Introduction and results}

\subsection{Introduction}
\label{sec:intro}

Let $A$ be the adjacency matrix of a random $d$-regular graph on $N$ vertices.
For fixed $d \geq 3$, it is well known that as $N \to \infty$ the empirical
spectral measure of $A$ converges weakly to the \emph{Kesten-McKay law} \cite{MR0109367,MR629617},
with density
\begin{equation} \label{e:KM-unrescaled}
\frac{d}{d^2 - x^2} \frac{1}{2 \pi} \sqrt{[4 (d - 1) - x^2]_+}\,.
\end{equation}
Thus, the rescaled adjacency matrix $(d - 1)^{-1/2} A$ has asymptotic spectral density
\begin{equation} \label{e:KM-rescaled}
\varrho_d(x) \;\deq\; \pbb{1 + \frac{1}{d - 1} - \frac{x^2}{d}}^{-1} \frac{\sqrt{[4 - x^2]_+}}{2 \pi}\,.
\end{equation}
Clearly, $\varrho_d(x) \to \varrho(x)$ as $d\to\infty$, where $\varrho(x) \deq \frac{1}{2\pi} \sqrt{[4-x^2]_+}$
is the density of Wigner's \emph{semicircle law}.
The semicircle law is the asymptotic eigenvalue distribution of a random Hermitian matrix with independent (upper-triangular) entries
(correctly normalized and subject to mild tail assumptions).
From \eqref{e:KM-rescaled} it is natural to expect that,
for sequences of random $d$-regular graphs such that $d \to \infty$ as $N\to\infty$ simultaneously,
the spectral density of $(d-1)^{-1/2} A$ converges to the semicircle law.
This was only proved recently \cite{MR2999215} (in \cite{MR3025715} it was also shown with the
restriction that $d$ is only permitted to grow logarithmically in $N$).

In the study of universality of random matrix statistics, local
versions of the semicircle law and its generalizations have played a
crucial role; see for instance the survey \cite{MR2917064}.  The local
semicircle law is a far-reaching generalization of the weak
convergence to the semicircle law mentioned above. First, the local
law admits test functions whose support decreases with $N$ so that far
fewer than $N$ eigenvalues are counted, ideally only slightly more
than order $1$.
(In contrast, weak convergence of probability measures applies only to
macroscopic test functions counting an order $N$ eigenvalues). Second, the
local law controls \emph{individual matrix entries} of the Green's
function. Both of these improvements have proved of fundamental
importance for applications.
In particular, the local law established in this paper is a crucial input in \cite{1505.06700-aop},
where, with J.\  Huang, we prove that the local eigenvalue statistics of $A$ coincide with 
those of the Gaussian Orthogonal Ensemble; see also Section~\ref{sec:intro-discuss} below.
For Wigner matrices, i.e.\ Hermitian random matrices with independent
identically distributed upper-triangular entries, the semicircle law
is known to hold down to the optimal spectral scale $1/N$,
corresponding to the typical eigenvalue spacing, up to a logarithmic
correction.  In \cite{MR2999215,MR3025715,1304.4343,1305.1039}, it was
shown that the semicircle law (for $d \to \infty$) or the
Kesten-McKay law (for fixed $d$) holds for random $d$-regular graphs
on spectral scales that are slightly smaller than the macroscopic scale $1$
(typically by a logarithmic factor;
see Section \ref{sec:intro-discuss} below for more details).

In this paper we show that
$d$-regular graphs with degree $d$ at least $(\log N)^4$ obey
the semicircle law down to spectral scales $(\log N)^4/N$.
This scale is optimal up to the power of the logarithm.

From the perspective of random matrix theory, the adjacency matrix of
a random $d$-regular graph is a symmetric random matrix with
nonnegative integer entries constrained so that all row and column
sums are equal to $d$. These constraints impose nontrivial
dependencies among the entries.  For example, if the sum of the first
$k$ entries of a given row is $d$, the remaining entries of that row
must be zero.
Previous approaches to bypass this difficulty include local approximation of the 
random regular graph by a regular tree (for small degrees) and coupling to an Erd\H{o}s-R\'enyi graph (for large degrees).
These approaches have been shown to be effective for the study of several combinatorial properties,
as well as global spectral properties of random regular graphs.
However, they encounter serious difficulties when applied to the eigenvalue distribution on small scales
(see Section~\ref{sec:intro-discuss} below for more details).
Our strategy instead relies on a multiscale iteration of a self-consistent equation, in part
inspired by the approach for random matrices with independent entries initiated in \cite{MR2481753}
and significantly improved in a sequence of subsequent papers
(again see Section~\ref{sec:intro-discuss} for details).
In previous works on local laws for random matrices, independence of
the matrix entries plays a crucial role in deriving the self-consistent equation
(see e.g.\ \cite{MR3068390} for a detailed account).
While the independence of the matrix entries can presumably be
replaced by weak or short-range dependence, the dependence structure
of the entries of random regular graphs is global.
Thus, instead of independence, our approach uses the well known invariance of the
random regular graph under a dynamics of local switchings, via a \emph{local resampling} of vertex neighbourhoods.
We believe that our strategy of local resampling, using invariance under a local dynamics
combined with a multiscale iteration,
is generally applicable to the study of the local eigenvalue distribution of
random matrix models with constraints.

\paragraph{Notation}

We use $a=O(b)$ to mean that there exists an absolute constant $C>0$ such that $\abs{a} \leq C b$,
and $a \gg b$ to mean that $a \geq C b$ for some sufficiently large absolute constant $C>0$.
Moreover, we abbreviate $\qq{a,b} \deq [a,b] \cap \Z$.
We use the standard notations $a \wedge b \deq \min \{a,b\}$ and $a \vee b \deq \max \{a,b\}$.
Every quantity that is not explicitly a constant may depend on $N$, which we almost always omit from our notation.
Throughout the paper, we tacitly assume $N\gg 1$.

\subsection{Random regular graphs}
\label{sec:intro-models}

We establish the local law for the following three standard models of random $d$-regular graphs.

\paragraph{Uniform model}

Let $N$ and $d$ be positive integers such that $Nd$ is even. The uniform model is the uniform probability measure on 
the set of all simple $d$-regular graphs on $\qq{1,N}$. (Here, \emph{simple} means that the graph has no loops or multiple edges.)
Equivalently, its adjacency matrix $A$ is uniformly distributed over the
symmetric matrices with entries in $\h{0,1}$ such that all rows have sum $d$ and the diagonal entries are zero.

\paragraph{Permutation model}

Let $N$ be a positive integer and $d$ an even positive integer.
Let $\sigma_1, \dots, \sigma_{d/2}$ be independent uniformly
distributed permutations on $S_N$, the symmetric group of order $N$.  The
permutation model is the random graph on $N$ vertices obtained by
adding an edge $\{i, \sigma_\mu(i)\}$ for each $i \in \qq{1,N}$ and
$\mu \in \qq{1,d/2}$.
Its adjacency matrix $A$ is given by
\begin{equation} \label{e:A-PM}
A_{ij} \;\deq\; \sum_{\mu = 1}^{d/2} \pb{\ind{j = \sigma_\mu(i)} + \ind{i = \sigma_\mu(j)}}
\;=\; \sum_{\mu=1}^d \ind{j = \sigma_\mu(i)}
\,,
\end{equation}
with the convention that $\sigma_{d-\mu} =\sigma_\mu^{-1}$ for $d/2+1 \leq \mu \leq d$ in the second equality.
All vertices have even degree, and in general the graph may have loops as well as multiple edges.
Each loop contributes two to the degree of its incident vertex.

\paragraph{Matching model}
Let $N$ be an even positive integer and $d$ a positive integer.
Let $\sigma_1, \dots, \sigma_d$ be independent
uniformly distributed perfect matchings on $\qq{1,N}$. A perfect matching can be identified with a
permutation of $S_N$ whose cycles all have length two.
As in the permutation model,
a graph on $\qq{1,N}$ is obtained by adding an edge $\{i,\sigma_\mu(i)\}$ for all $i\in\qq{1,N}$ and $\mu \in \qq{1,d}$.
Thus, the corresponding adjacency matrix is again
\begin{equation} \label{e:A-MM}
  A_{ij} \;\deq\; \sum_{\mu=1}^d \ind{j = \sigma_\mu(i)} \,.
\end{equation}
Graphs of this model can have multiple edges but no loops.
Their degree $d$ is arbitrary, but their number of vertices must be even.

The models introduced above include simple graphs (uniform model),
graphs with loops and multiple edges (permutation model),
and graphs with multiple edges but no loops (matching model).
Throughout this paper, all statements apply to any of the above three models,
unless explicitly stated otherwise.
As discussed in Section~\ref{sec:intro-discuss} below,
our approach is quite general,
and applies to other models of random regular graphs as well.
For brevity, however, we give the details for the three representative models introduced above.

We shall give error bounds depending on the parameter
\begin{alignat}{2}
  \label{e:D-UM}
  D &\;\deq\; d \wedge \frac{N^2}{d^3}
  \qquad&&\text{(uniform model)}
  \,,\\
  \label{e:D-PMMM}
  D &\;\deq\; d \wedge \frac{N^2}{d}
  \qquad&&\text{(permutation and matching models)}
  \,.
\end{alignat}
In particular, for the uniform model, $D=d$ if $d \leq \sqrt{N}$,
and for the permutation and matching models, $D=d$ if $d \leq N$.
Throughout the paper, we make the tacit assumption $D \geq 1$,
which leads to the conditions $d\leq N^{2/3}$ for the uniform model
and $d \leq N^2$ for the permutation and matching models.

\subsection{Main result}
\label{sec:intro-results}
To state our main result, we first observe that the adjacency matrix $A$ of any $d$-regular graph
on $N$ vertices has the eigenvector $\f e \deq \smash{N^{-1/2}}(1,\dots,1)^*$
with eigenvalue $d$, and that (by the Perron-Frobenius theorem) all other eigenvalues are at most $d$ in absolute value.
The largest eigenvalue $d$ of the eigenvector $\f e$ is typically far from the other eigenvalues,
and it is therefore convenient to set it to be zero.
In addition, we rescale the adjacency matrix so that its eigenvalues are typically of order one.
Hence, instead of $A$ we consider
\begin{equation} \label{e:H}
  H \;\deq\; (d-1)^{-1/2} \pb{A - d \, \f e \f e^* }\,.
\end{equation}
Clearly, $A$ and $H$ have the same eigenvectors, and the spectra of $(d-1)^{-1/2} A$ and $H$ coincide on the subspace orthogonal to $\f e$.

Our main result is stated in terms of the Green's function (or the resolvent) of $H$, defined by
\begin{equation} \label{e:Gdef}
G(z) \;\deq\; (H - z)^{-1}
\end{equation}
for $z\in\C_+$. Here $\C_+ \deq \{ E+\ii\eta\col E \in \R, \eta > 0\}$ denotes the upper half-plane.
We always use the notation $z=E+\ii\eta$ for the real and imaginary parts of $z \in \C_+$,
and regard $E \equiv E(z)$ and $\eta \equiv \eta(z)$ as functions of $z$.

For $z \in \C_+$, let
\begin{equation} \label{e:mdef}
m(z) \;\deq\; \int \frac{\varrho(x)}{x - z} \, \dd x
\;=\; \frac{-z + \sqrt{z^2 - 4}}{2}
\end{equation}
be the Stieltjes transform of the semicircle law. Here the square root is chosen
so that $m(z) \in \C_+$ for $z \in \C_+$, or, equivalently,
to have a branch cut $[-2,2]$ and to satisfy $\sqrt{z^2-4} \sim z$ as $|z|\to\infty$.
We shall control the errors using the parameter
\begin{equation} \label{e:Phidef}
  \Phi(z)
  \;\deq\; \frac{1}{\sqrt{N \eta}} + \frac{1}{\sqrt{D}}
  \,,
\end{equation}
and the function
\begin{equation} \label{e:Fdef}
  F_z(r)
  \;\deq\;
  \qbb{\pbb{1+\frac{1}{\sqrt{|z^2-4|}} }r} \wedge \sqrt{r} 
  \,,
\end{equation}
where $r \in [0,1]$.
Away from the two edges $z=\pm 2$ of the support of the semicircle law, i.e.\
for $|z\pm 2| \geq \varepsilon$ for some $\varepsilon>0$,
the function $F$ is linearly bounded: $F_z(r) = O_\varepsilon(r)$.
Near the edges, $F_z(r) \leq \sqrt{r}$ provides a weaker bound.

We now state our main result.

\begin{theorem}[Local semicircle law] \label{thm:semicircle}
Let $G(z)$ be the Green's function \eqref{e:Gdef} of any of the models of
random $d$-regular graphs introduced in Section~\ref{sec:intro-models}.
Let $\xi \log \xi \gg (\log N)^2$ and $D \gg \xi^2$.
Then, with probability at least $1-\ee^{-\xi \log \xi}$, 
\begin{equation} \label{e:semicircle}
  \max_i \abs{G_{ii}(z) - m(z)} \;=\; O(F_z(\xi\Phi(z))) \,,
  \qquad
  \max_{i\neq j} \abs{G_{ij}(z)} \;=\; O(\xi\Phi(z))
  \,,
\end{equation}
simultaneously for all $z\in \C_+$ such that $\eta \gg \xi^2/N$.
\end{theorem}

The condition $D \gg \xi^2$ in the statement of Theorem~\ref{thm:semicircle}
implies the following restrictions on the degree of the graphs:
\begin{alignat}{2}
  \label{e:intro-dUM}
  \xi^2 &\;\ll\; d \;\ll\;
  \pbb{\frac{N}{\xi}}^{2/3} \qquad&&\text{(uniform model)}
  \,,\\
  \label{e:intro-dPMMM}
  \xi^2 &\;\ll\; d \;\ll\;
  \pbb{\frac{N}{\xi}}^2 \qquad&&\text{(permutation and matching models)}
  \,.
\end{alignat}
Thus, for the smallest possible degree $d$ and the smallest spectral scale $\eta$
for which Theorem~\ref{thm:semicircle} applies,
the parameter $\xi$ needs to be chosen as small as permitted, which is slightly smaller than $(\log N)^2$.
In particular, the local semicircle law holds  for all
$\eta \geq (\log N)^4/N$ and all $d \geq (\log N)^4$ satisfying $d \leq N^{2/3} (\log N)^{-4/3}$
for the uniform model and $d \leq N^2 (\log N)^{-4}$ for the permutation and matching models.

The estimates \eqref{e:semicircle} have
a number of well-known consequences for the eigenvalues and eigenvectors of $H$, and hence also for those of $A$.
Some of these are discussed below.
In fact, by the exchangeability of random regular graphs,
Theorem~\ref{thm:semicircle} actually implies an \emph{isotropic} version of the local semicircle law,
as well as corresponding isotropic versions of its consequences for the eigenvectors,
such as isotropic delocalization and a probabilistic version of local quantum unique ergodicity. 
We discuss the isotropic modifications in Section~\ref{sec:isotropic}, and restrict ourselves here to the standard basis of $\R^N$.

For instance, Theorem~\ref{thm:semicircle} implies that all eigenvectors are completely delocalized.

\begin{corollary}[Eigenvector delocalization] \label{cor:eigenvectors}
  Under the assumptions of Theorem~\ref{thm:semicircle}, with probability at least  $1-\ee^{-\xi\log\xi}$,
  all $\ell^2$-norm\-al\-ized eigenvectors of $A$ or $H$ have $\ell^\infty$-norm of size $O(\xi/\sqrt{N})$.
\end{corollary}

\begin{proof}
Since $A$ and $H$ have the same eigenvectors, it suffices to consider $H$.
Let $\f v_\alpha = (v_{\alpha,i})_{i = 1}^N$, $\alpha = 1, \dots, N$
denote an orthonormal basis of eigenvectors with $H \f v_\alpha = \lambda_\alpha \f v_\alpha$.
Let $\xi$ be as in Theorem \ref{thm:semicircle}, and set $\eta \deq C \xi^2 / N$ for some large enough constant $C$. 
Note that
\begin{equation*}
  v_{\alpha,i}^2 
  \;\leq\;
  \sum_\beta \frac{\eta^2v_{\beta,i}^2}{(\lambda_\beta-\lambda_\alpha)^2+\eta^2}
  \;=\;
  \eta \im G_{ii}(\lambda_\alpha + \ii\eta)\,.
\end{equation*}
By Theorem~\ref{thm:semicircle}, there exists an event of probability at least $1-\ee^{-\xi \log \xi}$
such that for all $i$ and $\alpha$ the right-hand side above is bounded by
\begin{equation*}
\eta \im G_{ii}(\lambda_\alpha + \ii \eta) \;\leq\; \eta |m(\lambda_\alpha + \ii \eta)| + O(\eta\sqrt{\xi\Phi (\lambda_\alpha + \ii \eta)})
  \;\leq\; 2\eta\,,
\end{equation*}
where we used the bound
\begin{equation} \label{e:mbd}
  |m(z)| \;\leq\; 1\,,
\end{equation}
which follows easily from \eqref{e:mdef}.
Thus $v_{\alpha,i}^2 \leq 2\eta = O(\xi^2/N)$ as claimed, concluding the proof.
\end{proof}

Next, Theorem \ref{thm:semicircle} yields a semicircle law on small scales for the empirical spectral measure of $H$.
The Stieltjes transform of the empirical spectral measure of $H$ is defined by
\begin{equation} \label{e:sdef}
  s(z) \;\deq\; \frac{1}{N} \sum_{\alpha=1}^N \frac{1}{\lambda_\alpha-z}
  \;=\; \frac{1}{N} \sum_{i=1}^N G_{ii}(z)\,,
\end{equation}
where $\lambda_1, \dots, \lambda_N$ are the eigenvalues of $H$.
Theorem~\ref{thm:semicircle} implies that
\begin{equation} \label{e:s-m}
s(z) \;=\; m(z) + O(F_z(\xi\Phi(z)))
\end{equation}
with probability at least $1-\ee^{-\xi \log \xi}$.
Following a standard application of the Helffer-Sj\"ostrand functional calculus along the lines of \cite[Section 8.1]{MR3098073},
the following result may be deduced from \eqref{e:s-m}.

\begin{corollary}[Semicircle law on small scales] \label{cor:sc}
Let
\begin{equation*}
\varrho(I) \;\deq\; \int_I \varrho(x) \, \dd x \,, \qquad \nu(I) \;\deq\; \frac{1}{N} \sum_{\alpha = 1}^N \ind{\lambda_\alpha \in I}
\end{equation*}
denote the semicircle and empirical spectral measures, respectively, applied to an interval $I$.
Fix a constant $K > 0$. Then, under the assumptions of Theorem \ref{thm:semicircle}, for any interval $I \subset [-K,K]$ we have
\begin{equation} \label{sc_small_scales}
\nu(I) - \varrho(I) \;=\; O\qBB{\xi \, \frac{\abs{I}}{\sqrt{\kappa(I) + \abs{I}}} \pbb{\frac{1}{\sqrt{D}} + \frac{1}{\sqrt{N \abs{I}}}} + \frac{\xi^2}{N}}
\end{equation}
with probability at least  $1-\ee^{-\xi \log \xi}$,
where $\abs{I}$ denotes the length of $I$ and $\kappa(I) \deq \dist(I, \{-2,2\})$ the distance from $I$ to the spectral edges $\pm 2$.
\end{corollary}

Corollary \ref{cor:sc} says in particular that, in the bulk
spectrum, the empirical spectral density of $H$ is well approximated
by the semicircle law down to spectral scales $\xi^2 / N$. Indeed, fix
$\epsilon > 0$ and suppose that $I \subset [-2+\epsilon, 2 -
\epsilon]$, so that $\kappa(I) \geq \epsilon$. Then the right-hand
side of \eqref{sc_small_scales} is much smaller than $\varrho(I)$
provided that $\abs{I} \gg \xi^2 / N$. We deduce that the distribution
of the eigenvalues of $H$ is very regular all the way down to the
microscopic scale. Moreover, clumps of eigenvalues containing more
than $(\log N)^4$ eigenvalues are ruled out with high probability: any
interval of length at most $(\log N)^4 / N$ contains with high
probability at most $O((\log N)^4)$ eigenvalues.

\begin{remark}
The estimate \eqref{sc_small_scales} deteriorates near the edges, when
$\kappa(I)$ is small. Here we do not aim for an optimal edge
behaviour, and \eqref{sc_small_scales} can in fact be
improved near the edges by a more refined application of \eqref{e:s-m}.
For example, from \eqref{e:s-m} we also obtain the estimate
\begin{equation} \label{sc_small_scales_edge}
\nu(I) - \varrho(I) \;=\; O\qBB{\sqrt{\xi} \abs{I} \pbb{\frac{1}{D^{1/4}} + \frac{1}{(N \abs{I})^{1/4}}} + \frac{\xi^2}{N}}
\end{equation}
with probability at least  $1-\ee^{-\xi \log \xi}$, which is stronger than \eqref{sc_small_scales} when $\abs{I}$ and $\kappa(I)$ are small.
Moreover, as explained in Remark \ref{rem:edge} below,
\eqref{sc_small_scales} itself, and hence estimates of the form \eqref{sc_small_scales_edge}, can be improved near the edges.
We do not pursue these improvements here.
\end{remark}

\begin{remark}
Theorem \ref{thm:semicircle} has a simple extension in which the condition $\eta \gg \xi^2 / N$ is dropped.
Indeed, using Lemma \ref{lem:Gbd} below, it is easy to conclude that,
under the assumptions of Theorem \ref{thm:semicircle},
for any $z \in \C_+$ with $\eta = O(\xi^2/N)$ we have the estimate
$\abs{G_{ij}(z) - \delta_{ij} m(z)} = O\pb{\frac{\xi^2}{N \eta}}$ with probability at least $1-\ee^{-\xi \log \xi}$.
\end{remark}

\begin{remark} \label{rem:edge}
Up to the logarithmic correction $\xi$, we expect that the estimates \eqref{e:semicircle}
cannot be improved in the bulk of the support of the semicircle law, i.e.\ for $|E| \leq 2-\varepsilon$.
On the other hand, \eqref{e:semicircle} is not optimal for $|E| \geq 2-\varepsilon$.
For example, a simple extension of our proof allows one to show that the term $\xi\Phi(z)$ on the right-hand sides of \eqref{e:semicircle} can be
replaced by the smaller bound
\begin{equation} \label{e:Psibd}
\xi \sqrt{\frac{\im m(z)}{N\eta}} + \frac{\xi}{\sqrt{D}} + \pbb{\frac{\xi^2}{N\eta}}^{2/3} \,.
\end{equation}
In order to focus on the main ideas of this paper, we give the proof of the simpler estimate \eqref{e:semicircle}.
In Appendix \ref{app:Psi}, we sketch the required changes to obtain the improved error bound \eqref{e:Psibd}.
The bound \eqref{e:semicircle} is sufficient for most applications, including Corollaries~\ref{cor:eigenvectors}--\ref{cor:sc}.
Finally, we remark that all of our error bounds are designed with the regime of bounded $z$ in mind;
as $z \to \infty$, much better bounds can be easily obtained. We do not pursue this direction here.
\end{remark}

\subsection{Related results}
\label{sec:intro-discuss}

We conclude this section with a discussion of some related results. The convergence of the empirical
spectral measure of a random $d$-regular graph has been previously established on spectral scales
slightly smaller than the macroscopic scale $1$. More precisely, in \cite[Theorem 1.6]{MR2999215},
the semicircle law is established down to the spectral scale $d^{-1/10}$ for $d \to \infty$.
In \cite[Theorem 2 and Remark 1]{MR3025715}, the semicircle law is established down to the
spectral scale $(\log N)^{-1}$ for $d =(\log N)^\gamma$ with $\gamma > 1$, and the spectral scale
$1/d$ for $d = (\log N)^\gamma$ with $\gamma < 1$.
In \cite[Theorem 5.1]{1304.4343}, it is shown that for fixed $d$ the Kesten-McKay law holds down
to the spectral scale $(\log N)^{-c}$ for some $c>0$.
Finally, in \cite[Theorem 2.1]{1305.1039}, it is shown that for fixed $d$ the Kesten-McKay law
holds down to the spectral scale $(\log N)^{-1}$.

The results of \cite{MR2999215} were proved by coupling to an Erd\H{o}s-R\'enyi graph.
The probability that an Erd\H{o}s-R\'enyi graph in which each edge is chosen independently with
probability $p$ is $d$-regular, with $d=pN$, is at least $\exp(-c N \log d)$.
Hence, any statement that holds for the Erd\H{o}s-R\'enyi graph with probability greater
than $1 - \exp(-c N \log d)$ also holds for the random $d$-regular graph.
While global spectral properties can be established with such high probabilities,
super-exponential error probabilities are not expected to hold for local spectral properties.

In a related direction,
contiguity results imply that almost sure asymptotic properties of various models
of random regular graphs can be related to each other (see e.g.\ \cite{MR1725006} for details).
Such results are difficult to extend to the case where $d$ grows with $N$,
for example because the probability that a graph of the permutation model is simple
tends to zero roughly like $\exp(-c d^2)$. This probability is smaller than the error probabilities
that we establish in this paper.
Our proof does not rely on a comparison between different models, but works directly with each
model. It is rather general, and may in particular be adapted to
other models of random regular graphs.
For instance, by an argument similar to (but somewhat simpler than) the one given in Section \ref{sec:UM},
we may prove Theorem~\ref{thm:semicircle} for the configuration model of random regular graphs.
Moreover, by a straightforward extension of our method, our results remain valid for arbitrary
superpositions of the models from Section~\ref{sec:intro-models}. For example, we can consider
a regular graph defined as the union of several independent uniform regular graphs of lower degree.
(In fact, the matching model is the union of $d$ independent copies of a uniform $1$-regular graph).

The results of \cite{MR3025715,1304.4343,1305.1039} were obtained by local approximation by a tree.
It is well known that, locally around almost all vertices, a random $d$-regular graph is 
well approximated by the $d$-regular tree, at least for fixed $d \geq 3$.
The Kesten-McKay law is the spectral measure of the infinite $d$-regular tree,
and many previous results on the spectral properties of $d$-regular graphs use some form of
local approximation by the $d$-regular tree.
In particular, it is known that the spectral measure of any sequence of graphs converging
locally to the $d$-regular tree converges to the Kesten-McKay law; see for instance \cite{MR2724665}.
Moreover, in \cite{MR3038543}, under an assumption on the number of small cycles (corresponding approximately to a
locally tree-like structure and satisfied with high probability by random regular graphs), it is proved
that eigenvectors cannot be localized in the following sense: if for some $\ell^2$-normalized eigenvector $\f v = (v_i)_{i = 1}^N$
a set $B \subset \qq{1,N}$ satisfies $\sum_{i\in B} |v_i|^2 \geq \varepsilon >0$,
then $|B| \geq N^{\delta}$ with high probability for some small $\delta \propto \varepsilon^2$.
In comparision, for a random $d$-regular graph with $d \geq (\log N)^4$, Corollary~\ref{cor:eigenvectors} implies
that if a set $B$ has $\ell^2$-mass $\varepsilon>0$ then
$|B| \geq \varepsilon N(\log N)^{-4}$ with high probability, which is optimal up to the power of the logarithmic correction.
Furthermore, in \cite{1304.4343}, for $d$-regular expander graphs with local tree structure, for fixed $d \geq 3$,
a graph version of the quantum ergodicity theorem is proved:
it is shown that averages over eigenvectors whose eigenvalues lie in an interval containing at least $N (\log N)^{-c}$
eigenvalues converge to the uniform distribution, along with a version of the Kesten-McKay
law at spectral scales slightly smaller (by a related logarithmic factor) than the macroscopic scale $1$.
For random regular graphs, also using the local tree approximation,
similar estimates for eigenvalues on scales of order roughly $(\log N)^{-c}$ were also established in \cite{MR3025715,1305.1039}.
In all of these works, the logarithmic factor arises as the radius of the largest neighbourhood where the tree approximation holds,
which is of order $\log_dN$.

Our proof does not use the tree approximation.
Instead, we use that a local resampling using appropriately chosen \emph{switchings} leaves the 
random regular graphs from Section~\ref{sec:intro-models} invariant.
Switchings of random regular graphs were introduced to prove enumeration
results in \cite{MR790916}; see also \cite{MR1725006} for a survey of subsequent developments.
Switchings are also commonly used for simulating random regular graphs using Monte Carlo methods;
see e.g.\ \cite{MR2334585} and references therein.
Recently, switchings were employed to
bound the singularity probability of directed random regular graphs \cite{1411.0243}.%

For $d$-regular graphs, the value of second largest eigenvalue $\lambda_2$ is of particular interest.
At least for fixed $d \geq 3$, it was conjectured that for almost all random $d$-regular graphs
we have $\lambda_2 = 2\sqrt{d-1} + o(1)$ with high probability \cite{MR875835}.
For fixed $d$, this conjecture was proved in \cite{MR2437174}, following several larger bounds (for which references are given in \cite{MR2437174}).
Very recently, the results of\cite{MR2437174} were generalized and their proofs simplified in \cite{MR3385636,Bord15}.
For the permutation model with $d\to\infty$ as $N \to\infty$,
the best known bound is $\lambda_2 = O(\sqrt{d})$ \cite{FriedmanKahnSzemeredi}
(see \cite[Theorem~2.4]{MR3078290} for a more detailed proof).

Finally, it is believed that the eigenvalues of random $d$-regular graphs obey random matrix statistics as soon as $d \geq 3$.
There is numerical evidence that
the local spectral statistics in the bulk of the spectrum are governed by those of the Gaussian Orthogonal Ensemble (GOE) \cite{MR2647344,MR1691538},
and further that the distribution of the appropriately rescaled second largest eigenvalue $\lambda_2$
converges to the Tracy-Widom distribution of the GOE \cite{MR2433888}.

In \cite{1505.06700-aop}, with J.\ Huang, we prove that GOE eigenvalue statistics hold in the bulk
for the uniform random $d$-regular graph with degree
$d \in [N^\alpha, N^{2/3-\alpha}]$ for arbitrary $\alpha>0$.
Here, the lower bound on the degree is of purely technical
nature, and we believe that the results of \cite{1505.06700-aop} can be
established with the same method under the weaker assumption $d \geq (\log N)^{O(1)}$.
The local law proved in this paper, in addition to the results of \cite{1504.03605,1504.05170},
is an essential input for the proof in \cite{1505.06700-aop}.

For Erd\H{o}s-R\'enyi graphs, in which each edge is chosen independently with
probability $p$, the local semicircle was established under the condition $pN \geq (\log N)^{O(1)}$ in \cite{MR3098073}.
Moreover, random matrix statistics for both the bulk eigenvalues
and the second largest eigenvalue were established in \cite{MR2964770} under the condition
$pN \geq N^{2/3+\alpha}$ for arbitrary $\alpha>0$.
For random matrix statistics of the bulk eigenvalues,
the lower bound on $pN$ was recently extended to $pN \geq N^{\alpha}$ for any $\alpha>0$
in \cite{1504.05170},
and GOE statistics for the eigenvalue gaps was also established.
Previous results on the spectral statistics of Erd\H{o}s-R\'enyi graphs are discussed in \cite{MR2964770,MR3098073,MR2999215}.

\section{Preliminaries and the self-improving estimate}
\label{sec:outline}

In this section we introduce some basic tools and definitions on which our proof relies, and state a self-improving estimate, Proposition \ref{prop:Lambda}, from which Theorem \ref{thm:semicircle} will easily follow. The rest of this paper will be devoted to the proof of Proposition \ref{prop:Lambda}.

From now on we frequently omit the spectral parameter $z$ from our notation, and write $G \equiv G(z)$ and so on.
The spectral representation of $G$ implies the trivial bound
\begin{equation} \label{e:Getabd}
  |G_{ij}| \;\leq\; \frac{1}{\eta}\,.
\end{equation}
We shall also use the \emph{resolvent identity}: for invertible matrices $A,B$,
\begin{equation} \label{e:resolv}
  A^{-1} - B^{-1} \;=\;  A^{-1} (B-A) B^{-1} \,.
\end{equation}
In particular, applying \eqref{e:resolv} to $G - G^{*}$, we obtain
the \emph{Ward identity}
\begin{equation} \label{e:Ward}
  \sum_{k = 1}^{N} |G_{ik}|^2 \;=\; \frac{\im G_{ii}}{\eta} \,.
\end{equation}
Assuming $\eta \gg \frac{1}{N}$, \eqref{e:Ward} shows that the squared $\ell^2$-norm
$\frac{1}{N} \sum_{k=1}^N |G_{ik}|^2$ is smaller by the factor $\frac{1}{N\eta} \ll 1$
than the diagonal element $|G_{ii}|$.
This identity was first used systematically in the
proof of the local semicircle law for random matrices in \cite{MR2981427}.

The core of the proof is an induction on the spectral scale, where information about $G$ is passed on from the scale $\eta$ to the scale $\eta/2$.
(See Remark~\ref{rk:bootstrapping} below for a comparison of this induction
with the bootstrapping/continuity arguments used in the proofs of local laws in models with independent entries.)
The next lemma is a simple deterministic result that allows us to propagate bounds on the Green's function on a certain scale to weaker bounds on a smaller scale.
This result will play a crucial role in the induction step. In order to state it, we introduce the random error parameters
\begin{equation} \label{e:Gammadef}
  \Gamma \;\equiv\;
  \Gamma(z) \;\deq\; \max_{i,j} \abs{G_{ij}(z)}  \vee 1\,,
  \qquad
  \Gamma^* \;\equiv\;
  \Gamma^*(z) \;\deq\; \sup_{\eta' \geq \eta} \Gamma(E+\ii \eta')\,.
\end{equation}

\begin{lemma} \label{lem:Gbd}
  For any $M>1$ and $z \in \C_+$ we have
  $\Gamma(E+\ii\eta/M) \leq M \Gamma(E+\ii\eta)$.
\end{lemma}

\begin{proof}
Fix $E \in \R$ and write $\Gamma(\eta) = \Gamma(E+\ii \eta)$. 
For sufficiently small $h$, since $|x\vee 1-y\vee 1|\leq |x-y|$ for $x,y>0$,
using the resolvent identity,
the Cauchy-Schwarz inequality, and \eqref{e:Ward}, we get
\begin{align*}
  \abs{\Gamma(\eta+h)-\Gamma(\eta)}
  &\;\leq\; \max_{i,j} \abs{G_{ij}(E+\ii (\eta+h)) - G_{ij} (E+\ii \eta)}
  \nonumber\\
  &\;\leq\; |h| \max_{i,j} \sum_k \abs{G_{ik}(E+\ii (\eta+h))G_{kj} (E+\ii \eta)}
  \;\leq\; |h| \sqrt{\frac{\Gamma(\eta+h)\Gamma(\eta)}{(\eta+h)\eta}}
  \,.
\end{align*}
Thus, $\Gamma$ is locally Lipschitz continuous, and its almost everywhere defined derivative satisfies
\begin{equation*}
  \absbb{\ddp{\Gamma}{\eta}}
  \;\leq\; \frac{\Gamma}{\eta}\,.
\end{equation*}
This implies $\ddp{}{\eta} (\eta \Gamma(\eta)) \geq 0$ and therefore $\Gamma(\eta/M) \leq M\Gamma(\eta)$ as claimed.
\end{proof}

The main ingredient of the proof of Theorem~\ref{thm:semicircle} is the following result, whose proof
constitutes the remainder of the paper. To state it, we introduce the set
\begin{equation} \label{e:D}
\f D \;\equiv\; \f D(\xi) \;\deq\; \hbb{ E + \ii \eta \col \frac{\xi^2}{N} \ll \eta \leq N\,, \; -N \leq E \leq N } \,,
\end{equation}
where the implicit absolute constant in $\ll$ is chosen large enough in the proof of the following result.

\begin{proposition} \label{prop:Lambda}
Suppose that $\xi>0$, $\zeta > 0$, and that $D \gg \xi^2$.
If for a fixed $z \in \f D$ we have 
  \begin{equation*}
    \Gamma^*(z) \;=\; O(1)
  \end{equation*}
  with probability at least $1-\ee^{-\zeta}$, then for the same $z$ we have
  \begin{equation} \label{e:Lambdaind}
\max_i \abs{G_{ii}(z) - m(z)} \;=\; O(F_z(\xi \Phi(z)))\,, \qquad \max_{i \neq j} \abs{G_{ij}(z)} \;=\; O(\xi\Phi(z))
    \,,
  \end{equation}
  with probability at least $1-\ee^{-(\xi\log\xi)\wedge \zeta + O(\log N)}$.
\end{proposition}

Given Proposition~\ref{prop:Lambda}, Theorem~\ref{thm:semicircle} is a simple consequence.

\begin{proof}[Proof of Theorem~\ref{thm:semicircle}]
Let $\xi\log\xi \gg (\log N)^2$.
We first note that it suffices to prove \eqref{e:semicircle} for $z \in \f D$.
Indeed, since $d\ll N^2$ by assumption, the spectrum of $H$ is contained in the interval
$[-d/(2\sqrt{d-1}),d/(2\sqrt{d-1})] \subset [-\frac12 N, \frac12 N]$,
 and hence \eqref{e:semicircle} holds deterministically for $|E| \geq N$ by the spectral representation of $G$.
Similarly, the proof of \eqref{e:semicircle} is trivial for $\eta \geq N$.
Since $G$ is Lipschitz continuous in $z$ with Lipschitz constant bounded by $1/\eta^2 \leq N^2$,
it moreover suffices to prove \eqref{e:semicircle} for $z  \in \f D \cap (N^{-4} \Z^2)$. By a union bound, it suffices to
prove \eqref{e:semicircle} for each $E \in [-N,N] \cap N^{-4}\Z$.

Fix therefore $E \in [-N,N] \cap (N^{-4}\Z)$.
Let $K \deq \max \h{k \in \N \col N / 2^{k} \geq C \xi^2 / N}$, where $C > 0$ is the implicit absolute constant from the assumption $\eta \gg \xi^2 /N$ in the statement of the theorem. Clearly, $K \leq 4 \log N$.
For $k \in \qq{0, K}$, set $\eta_k \deq N/2^{k}$ and $z_k \deq E+\ii\eta_k$.
By induction on $k$, we shall prove that
\begin{equation} \label{induct_claim}
\Gamma^*(z_k) \leq 2 \text{ with probability at least } 1-\ee^{-\xi\log \xi + O(k\log N)}
\end{equation}
for $k \in \qq{0,K}$.
The claim \eqref{induct_claim} is trivial for $k=0$ since then $\eta_{k}= N$ and therefore \eqref{e:Getabd} implies
$\Gamma^*(z_k) \leq 1$ deterministically. 
Now assume that \eqref{induct_claim} holds for some $k \in \qq{0,K}$.
Then Lemma~\ref{lem:Gbd} applied with $\eta= \eta_{k}$ and $M = 2$ implies
\begin{equation} \label{e:bootstrap_assump}
\P\pb{
\Gamma^*(z_{k+1}) \geq 4
} \;\leq\; \ee^{-\xi\log\xi + O(k \log N)}\,.
\end{equation}
We may therefore apply Proposition~\ref{prop:Lambda} with $z = z_{k+1}$ and $\zeta = \xi \log \xi - O(k \log N)$.
Thus, we find that \eqref{e:Lambdaind} holds for $z = z_k$ with probability at least $1-\ee^{-\xi\log\xi + O((k+1)\log N)}$.
Since $\abs{m} \leq 1$ by \eqref{e:mbd}, we conclude that $\Gamma^* (z_{k+1}) \leq 2$ with with probability at least $1-\ee^{-\xi\log\xi + O((k+1)\log N)}$. This concludes the proof of the induction step, and hence of \eqref{e:Lambdaind} for all $z_k$ with $k \in \qq{0,K}$.

Finally, the argument may also be applied with $\xi$ replaced by $2\xi$,
concluding the proof since $\ee^{-2\xi\log 2\xi + O(\log N)^2} \leq \ee^{-\xi \log \xi}$ by assumption.
\end{proof}

\begin{remark} \label{rk:bootstrapping}
  The induction in the proof of Theorem~\ref{thm:semicircle}
  is \emph{not} a continuity (or
  bootstrapping) argument, as used e.g.\ in the works
  \cite{MR3068390,MR3098073,MR2481753} on local laws of models with independent
  entries. The \emph{multiplicative} steps $\eta \to \eta/2$ that we make
are far too large for a continuity argument to work, and we correspondingly
  obtain much weaker a priori estimates from the induction
  hypothesis.
  Thus, our proof relies on a priori control of $\Gamma$ instead of the error
  parameters $\Lambda_d$ and $\Lambda_o$ used in \cite{MR3068390,MR3098073,MR2481753}.
  The advantage, on the other hand, of the approach taken here
  is that we only have
  to perform an order $\log N$ steps, as opposed to the $N^C$ steps
  required in bootstrapping arguments. As evidenced by the proof of
  Theorem \ref{thm:semicircle},
  a logarithmic bound on the number of induction steps is crucial.
  An inductive approach was also taken in \cite{1311.0326},
where a local semicircle law without logarithmic corrections
  was proved for Wigner matrices with entries whose distributions
  are subgaussian.
\end{remark}

It therefore only remains to prove Proposition~\ref{prop:Lambda}. This
is the subject of the remainder of the paper, which we now briefly
outline.  We follow the concentration/expectation approach,
establishing concentration results on the entries of $G$ (Section
\ref{sec:concentration}) and computing the expectation of the diagonal
entries (Section \ref{sec:expectation}). All of this is performed with
respect to a conditional probability measure, which is constructed for
each fixed vertex. Roughly speaking, given a vertex, this conditional
probability measure randomizes the neighbours of the vertex in an
approximately uniform fashion. It is model-dependent and has to be
chosen with great care for all of the concentration/expectation
arguments of Sections \ref{sec:concentration}--\ref{sec:expectation}
to work. Its construction is easiest for the matching model, which we
explain in Section \ref{sec:parametrization}. The constructions for
the uniform and permutation models are given in Sections \ref{sec:UM}
and \ref{sec:PM} respectively.

\section{Local resampling}
\label{sec:parametrization}

All models of random regular graphs that we consider are invariant under permutation of vertices.
However, for our analysis, it is important to use a parametrization that distinguishes a fixed vertex.
Without loss of generality, we assume this vertex to be $1$.
This parametrization has to satisfy a series of properties, which are given in Proposition \ref{prop:switch} below.
Using these properties, in Sections~\ref{sec:concentration}--\ref{sec:expectation}, we complete the proof of Proposition \ref{prop:Lambda}. 
Loosely speaking, the parametrization allows us to resample the neighbours of $1$, independently, and only changing
a fixed number of edges in the remainder of the graph in a sufficiently random way.
In this section, we describe the parametrization and prove Proposition \ref{prop:switch} for the matching model.
The parametrizations for the uniform and permutation models are discussed in Section \ref{sec:UM} and \ref{sec:PM} respectively.

Random indices will play an important role throughout the paper.
We consistently use the letters $i,j,k,l,m,n$ to denote deterministic
indices, and $x,y$ to denote random indices.

\subsection{Local switchings} \label{sec:switchings}

Our basic strategy of \emph{local resampling} involves randomizing the neighbours of the fixed vertex $1$
by local changes of the graph, called \emph{switchings} in the graph theory literature \cite{MR1725006}.
We use \emph{double switchings} which involve three edges, as opposed to single switchings
which only involve two edges. Both are illustrated in
Figure~\ref{fig:switch}.

Throughout the following, we use the following conventions to describe
graphs. We consider general undirected graphs, which may have loops
and multiple edges. We consistently identify a graph with its
adjacency matrix $A$. The quantity $A_{ij} = A_{ji} \in \N$ is the
number of edges between $i$ and $j$, and $A_{ii} \in 2 \N$ is twice
the number of loops at $i$. The degree of $i$ is $\sum_{j} A_{ij}$,
which will always be equal to $d$ for all $i$. The graph $A$ is simple
if and only if it has no multiple edges or loops, i.e.\ $A_{ij} \in
\{0,1\}$ and $A_{ii} = 0$ for all $i,j$. Sometimes we endow edges with
a direction; we use the notation $ij$ for the edge $\{i,j\}$ directed
from $i$ to $j$.

\begin{figure}[t]
\begin{center}
\input{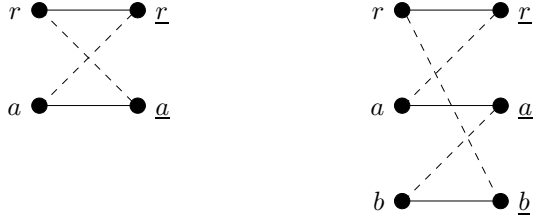}
\end{center}
\caption{Solid lines show edges of a graph before switching, dashed lines edges after a single switching (left)
  and after a double switching (right).
  In our application, $\ur$ is chosen to be $1$, so that the switching
  connects the vertex $1$ to a given vertex $a$.
  \label{fig:switch}}
\end{figure}

Let $\Delta_{ij}$ denote the adjacency matrix of a graph
containing only an edge between the vertices $i$ and $j$,
\begin{equation} \label{e:Eijdef}
  (\Delta_{ij})_{kl} \;\deq\; \delta_{ik}\delta_{jl} + \delta_{il}\delta_{jk}
  \,.
\end{equation}

To define switchings of a set of unoriented edges, it is convenient to assign directions to the edges
to be switched.
These directions determine which one of the possible switchings of the unoriented edges is chosen.
We define the \emph{single switching} of two edges $r\ur,a\ua$ of $A$ with the indicated directions
to be the graph
\begin{equation} \label{e:Asingleswitch}
  \tau_{r \ul r, a \ul a}(A) \;\deq\;
  A+\Delta_{\ur a} + \Delta_{r \ua} -\Delta_{r\ur} - \Delta_{a\ua}
\end{equation}
if $\abs{\h{r,\ur,a, \ua}} = 4$, and the graph $\tau_{r \ul r, a \ul a}(A) \deq A$ if $\abs{\h{r,\ur,a, \ua}} < 4$.
The \emph{double switching} of the three edges $r\ur, a\ua, b\ub$ of $A$ with the indicated directions is defined to be the graph
\begin{equation} \label{e:Adoubleswitch}
  \tau_{r \ul r, a \ul a, b \ul b}(A) \;\deq\; \tau_{b \ul r,a \ul a}(\tau_{r \ul r,b \ul b}(A)) \;=\;
  A+ \Delta_{\ur a} + \Delta_{\ua b} + \Delta_{\ub r} -\Delta_{r\ur}  - \Delta_{b\ub} - \Delta_{a\ua}
\end{equation}
if $\abs{\h{r,\ur,a, \ua, b, \ub}} = 6$, and the graph $\tau_{r \ul r, a \ul a, b \ul b}(A) \deq A$ if $\abs{\h{r,\ur,a, \ua, b, \ub}} < 6$.

Our goal is to use switchings to connect the distinguished vertex $1$ to essentially independent random vertices $a_1, \dots, a_d$
that are \emph{approximately uniform} in the sense of the next definition. 

\begin{definition}
A random variable $x$ with values in $\qq{1,N}$ is \emph{approximately uniform} if the total
variation distance of its distribution to the uniform distribution on $\qq{1,N}$ is of order $O\pb{\frac{1}{\sqrt{dD}}}$,
i.e.\ if $\sum_i \absb{\P(x=i)-\frac{1}{N}} = O\pb{\frac{1}{\sqrt{dD}}}$.
\end{definition}

To give an idea how approximately uniform random variables arise,
consider a switching with $\ur=1$ (to achieve our goal of connecting $1$ to a given vertex $a$ using a switching).
For simple graphs, a necessary condition to apply the switching \eqref{e:Adoubleswitch} is $a \neq 1$. Choosing
$a$ uniformly with this constraint means that it is uniform on $\qq{2,N}$. In particular,
the total variation distance of its distribution to that of the uniform distribution is
$O\pb{\frac{1}{N}} = O\pb{\frac{1}{\sqrt{dD}}}$.

Throughout this paper, $\frac{1}{\sqrt{dD}}$ appears frequently
as a bound on exceptional probabilities, and we tacitly use the estimates
\begin{equation}
  D \;\leq\; d \,, \qquad D \;\leq\; \sqrt{dD} \;\leq\; N \,,
\end{equation}
which follow directly from \eqref{e:D-UM}--\eqref{e:D-PMMM}, as well as
\begin{equation}
  \frac{1}{\sqrt{dD}} \;\leq\; \frac{1}{\sqrt{d}} \Phi \;\leq\; \Phi^2.
\end{equation}

We use the following conventions for conditional statements.

\begin{definition}
  Let $\sigG$ be a $\sigma$-algebra, $B$ an event, and $p \in
  [0,1]$. We say that, \emph{conditioned on $\cal G$, the event $B$
    holds with probability at least $p$} if $\P(B | \cal G) \geq p$
  almost surely. Moreover, we say that, \emph{conditioned on $\cal G$,
    the random variable $x$ is approximately uniform} if $\sum_i
  \absb{\P(x=i | \cal G)-\frac{1}{N}} = O\pb{\frac{1}{\sqrt{dD}}}$ almost
  surely.
\end{definition}

The use of double switchings opposed to single switchings ensures that either
condition (a) or (b) in the next lemma holds. These conditions will play an important role in
Section~\ref{sec:expectation}.
(That double switchings are in general more effective than single switchings is well known
in the combinatorial context; see for instance \cite{MR1725006} for a discussion.)

\begin{remark} \label{rk:switch3}
  Fix a $d$-regular graph $A$.
  For directed edges $r1,a\ua,b\ub,\ac\uac,\bc\ubc$ of $A$, we have
  \begin{equation} \label{e:switch3}
    \tau_{r1,\ac\uac,\bc\ubc}(A) - \tau_{r1,a\ua,b\ub}(A)
    \;=\;
    \Delta_{1\ac} - \Delta_{1a} + X
    \,,
  \end{equation}
  where $X$ is a sum of at most 8 terms $\pm \Delta_{x y}$. Explicitly, in the case
  \begin{equation} \label{disjoint_edges}
  \abs{\{1,r,a, \ul a, b, \ul b\}} \;=\;   \abs{\{1,r, \tilde a, \tilde {\ul a}, \tilde  b, \tilde {\ul b}\}} \;=\; 6
  \end{equation}
  we have
  \begin{equation} \label{X_def}
  X \;=\; \Delta_{\tilde {\ua} \tilde b} + \Delta_{\tilde {\ub} \tilde r}  - \Delta_{\tilde b\tilde {\ub}} - \Delta_{\tilde a\tilde {\ua}}
- \Delta_{\ua b} - \Delta_{\ub r} + \Delta_{b\ub} + \Delta_{a\ua}\,.
  \end{equation}
In particular, suppose that $A$ is deterministic and the directed edges $r1,a\ua,b\ub,\ac\uac,\bc\ubc$ are random such that \eqref{disjoint_edges} holds,
$a,\ua,\ac,\uac$ are approximately uniform, and, conditioned on $a,\ua,\ac,\uac$,
  the variables $b,\ub,\bc,\ubc$ are approximately uniform. Then for each term $\pm \Delta_{xy}$ we have
  (a) the random variables $x$ and $y$ are both approximately uniform, or
  (b) conditioned on $a,\ac$, at least one of $x$ and $y$ is approximately uniform.
\end{remark}

We emphasize that when we say that $x$ and $y$ are approximately uniform, this is a statement
about their individual distributions, and as such implies nothing about their joint distribution.

The introduction of switchings that connect $1$ to essentially independent random vertices $a_1, \dots, a_d$
is simplest in the matching model, in which the different neighbours of any given vertex are independent, so that
it suffices to consider a single neighbour of $1$ at a time.
In the next subsection, we explain in detail how this parametrization using switchings is defined for the matching model.

We state the conclusion, Proposition~\ref{prop:switch}, in great enough generality that it holds \emph{literally}
also for all of the other models, for which the more involved proofs are given in
Sections~\ref{sec:UM}--\ref{sec:PM}.
In the proof of Proposition~\ref{prop:Lambda}
(given in Sections~\ref{sec:concentration}--\ref{sec:expectation}),
and therefore in the proof of Theorem~\ref{thm:semicircle},
we only use the conclusion contained in Proposition~\ref{prop:switch}, and no other properties of the model.
Hence, Proposition \ref{prop:switch} summarizes everything about the random regular graphs that our proof requires.

\subsection{Matching model}
\label{sec:MM}

The matching model was defined in Section~\ref{sec:intro-models} in terms of $d$ independent
uniform perfect matchings of $\qq{1,N}$.
We first consider one such uniform perfect matching, i.e.\ a uniform $1$-regular graph.
We denote by $S_N$ the symmetric group of order $N$.
For $N$ even, denote by $M_N \subset S_N$ the set of perfect matchings of $\qq{1,N}$, which (as explained in Section \ref{sec:intro-models}) we identify with the subset of permutations whose cycles all have length $2$;
in particular $\pi=\pi^{-1}$ for $\pi \in M_N$.
For any perfect matching $\sigma \in M_N$, we denote the corresponding symmetric permutation matrix by
\begin{equation}  \label{e:Asigma}
  M(\sigma) \;\deq\; \frac12 \sum_{i=1}^N \Delta_{i\sigma(i)}
  \,.
\end{equation}
Note that $M(\cdot)$ is one-to-one.

Next, for $i,j,k \in \qq{1,N}$, we define the switching operation $T_{ijk} \col M_N \to M_N$ through
\begin{equation} \label{e:MM-Tijdef}
M(T_{ijk}(\pi)) \;\deq\; \tau_{\pi(i)i,j\pi(j),k\pi(k)}(M(\pi)) \,,
\end{equation}
where we recall that $\tau$ was defined in \eqref{e:Adoubleswitch}.
In particular, $T_{ijk}$ connects $i$ to $j$ (see Figure~\ref{fig:switch}) 
except in the exceptional case  $|\{i,j,k,\pi(i),\pi(j),\pi(k)\}| < 6$.

\begin{lemma} \label{lem:MM-tauij}
  Let $\pi$ be uniform over $M_N$, $i \in \qq{1,N}$ fixed,
  and $a,b$ independent and uniform over $\qq{1,N} \setminus \{i\}$.
  Then $T_{iab}(\pi)$ is uniform over $M_N$, and
  \begin{equation}\label{e:tauij-ij}
    (T_{iab}(\pi))(i) \;=\; a
  \end{equation}
  provided that $|\{i,a,b,\pi(i),\pi(a),\pi(b)\}| = 6$.
\end{lemma}

\begin{proof}
  To prove that $T_{iab}(\pi)$ is uniform over $M_N$, it suffices to check reversibility,
  i.e.\ that, for any fixed $\sigma,\sigma' \in M_N$,
  \begin{equation} \label{e:MM-rev}
    \P(T_{iab}(\pi) = \sigma' | \pi =\sigma) \;=\; \P(T_{iab}(\pi) =\sigma | \pi = \sigma') \,.
  \end{equation}
  Given $\sigma,\sigma' \in M_N$, $\sigma \neq \sigma'$, there is at most one pair $(a,b) \in (\qq{1,N}\setminus \{i\})^2$
  such that $T_{iab}(\sigma) = \sigma'$, and such a pair exists if and only if there exists a (different) pair $(a,b)$
  such that $T_{iab}(\sigma') = \sigma$ (see Figure~\ref{fig:switch} (right) for an illustration).
  If no such pairs exist, both sides of \eqref{e:tauij-ij} are zero.
  Otherwise, there exists precisely one pair $(a,b)$ such that $T_{iab}(\sigma) = \sigma'$,
  so that the left-hand side of \eqref{e:tauij-ij} is equal to $1/(N-1)^2$ because $(a,b)$
  is uniformly distributed over $(N - 1)^2$ elements;
  the same argument shows that the right-hand side of \eqref{e:tauij-ij} is also equal to $1/(N-1)^2$,
  which concludes the proof of \eqref{e:MM-rev}.
  Finally, \eqref{e:tauij-ij} is immediate from the definition of $T_{iab}$.
\end{proof}

The canonical realization of the probability space of the matching model is
the product of $d$ copies of the uniform measure on $M_N$.
For our analysis, we instead employ the larger probability space
$\Omega \deq \Omega_1 \times \cdots \times \Omega_d$ where
\begin{equation} \label{e:Omega_match}
  \Omega_\mu \;\deq\; M_N \times \qq{2,N} \times \qq{2,N} \,,
\end{equation}
also endowed with the uniform probability measure.
Elements of $\Omega_\mu$ are written as $(\pi_\mu,\amu,b_\mu)$.
We set  $\theta = (\pi_1, \dots, \pi_d)$,
$\umu = (a_{\mu},b_{\mu})$, and
\begin{equation} \label{e:sigmamu-MM}
  \sigma_\mu \;\deq\; T_{1 \amu \bmu}(\pi_\mu)
  \,.
\end{equation}
By Lemma~\ref{lem:MM-tauij},
$\sigma_1, \dots, \sigma_d$ are independent
uniform perfect matchings of $\qq{1,N}$, and therefore
the matching model is given by the adjacency matrix
\begin{equation} \label{e:A-M-MM}
  A \;=\; \sum_{\mu=1}^{d} M(\sigma_\mu) \,,
\end{equation}
which is a random variable on the probability space $\Omega$. 
To sum up, rather than working directly with the probability measure on matrices
that we are interested in, we use a measure-preserving lifting to a larger probability space, given by
$\Omega_\mu \to M_N \to \N^{N \times N}$
with $(\pi_\mu,a_\mu,b_\mu) \mapsto \sigma_\mu \deq T_{1a_\mu b_\mu}(\pi_\mu) \mapsto M(\sigma_\mu)$.

Throughout the following, we say that $(\alpha_1, \dots, \alpha_d) \in \qq{1,N}^d$ is an \emph{enumeration of the neighbours of $1$} if
\begin{equation} \label{e:neigh_enum}
A_{1 i} \;=\; \sum_{\mu = 1}^d \ind{i = \alpha_\mu}\,.
\end{equation}
(Recall that, as explained in the beginning of Section~\ref{sec:parametrization}, the vertex $1$ is distinguished.)
Defining $\alpha_\mu \deq \sigma_\mu(1)$, we find that $(\alpha_1, \dots, \alpha_d)$ is an enumeration of the neighbours of $1$.

\subsection{General parametrization}

Having described the probability space and the parametrization of the neighbours of $1$ for the matching model,
we now generalize this setup in order to admit other models of random regular graphs as well.

\begin{definition}[Parametrization of probability space] \label{def:parametrization}
We work on a finite probability space
\begin{equation} \label{e:Omega}
\Omega \;\deq\; \Theta \times U_1 \times \cdots \times U_d\,,
\end{equation}
whose points we denote by $(\theta, u_1, \dots, u_d)$.
Conditioned on $\theta \in \Theta$, the variables $u_1, \dots, u_d$ are independent.
For $\mu \in \qq{1,d}$ we define $\sigma$-algebras
\begin{align}
  \label{e:sigFdef}
  \sigF_\mu &\;\deq\; \sigma(\theta, \u_1,\dots, \u_\mu) \,,
  \\
  \label{e:sigGdef}
  \sigG_\mu &\;\deq\; \sigma(\theta, \u_1, \dots, \u_{\mu-1}, \u_{\mu+1}, \dots, \u_d) \,.
\end{align}
We also define $\sigF_0 \deq \sigma(\theta)$.
\end{definition}

In general, as in the case of the matching model in Section \ref{sec:MM},
the variable $u_\mu$ for $\mu \in \qq{1,d}$ determines (with high probability given $\theta \in \Theta$) the $\mu$-th neighbour of $1$.
Note that we have introduced an artificial ordering of the neighbours of $1$;
this ordering will prove convenient in Sections~\ref{sec:concentration}--\ref{sec:expectation}.
The interpretation of the $\sigma$-algebras  \eqref{e:sigFdef}--\eqref{e:sigGdef} is that
$\sigG_\mu$ determines all neighbours of $1$ except the $\mu$-th one,
and $\sigF_{\mu}$ determines the first $\mu$ neighbours of $1$.

Having constructed the probability space $\Omega$, we augment it with independent copies of the random variables $u_1, \dots, u_d$.

\begin{definition}[Augmented probability space] \label{def:parametrization2}
Let $\Omega$ be a probability space as in  Definition~\ref{def:parametrization}.
We augment $\Omega$ to a larger probability space $\tilde \Omega$ by adding independent copies of $u_\mu$
for each $\mu \in \qq{1,d}$. More precisely, we define
\begin{equation} \label{e:Omega_tilde}
\tilde \Omega \;\deq\; \Theta \times U_1 \times \cdots \times U_d \times U_1 \times \cdots \times U_d\,,
\end{equation}
whose points we denote by $(\theta, u_1, \dots, u_d, \tilde u_1, \dots, \tilde u_d)$.
We require that, conditioned on $\theta$, the variables $u_1, \dots, u_d, \tilde u_1, \dots, \tilde u_d$ are independent,
and that $u_\mu$ and $\tilde u_\mu$ have the same distribution.
On $\tilde \Omega$ we make use of the $\sigma$-algebras defined by \eqref{e:sigFdef}--\eqref{e:sigGdef}.

By definition, a random variable is a function 
$X \equiv X(\theta, u_1, \dots, u_d, \uc_1, \dots, \uc_d)$ on the augmented space $\tilde \Omega$.
Any function on $\Omega$ lifts trivially to an $\sigF_d$-measurable random variable.
Given a random variable $X \equiv X(\theta, u_1, \dots, u_d, \uc_1, \dots, \uc_d)$ and an index $\mu \in \qq{1, d}$,
we define the version $\tilde X^\mu$ of $X$ by exchanging the arguments $u_\mu$ and $\uc_\mu$ of $X$:
\begin{equation} \label{e:Xc}
  \Xcmu \;\deq\; X(\theta, u_1, \dots, u_{\mu-1}, \uc_\mu, u_{\mu + 1}, \dots, u_d, \uc_1, \dots, \uc_{\mu-1}, u_\mu, \uc_{\mu+1}, \dots, \uc_d) \,.
\end{equation}
\end{definition}

Throughout the following, the underlying probability space is always
the augmented space $\tilde \Omega$.
In particular, the vertex~$1$ is distinguished.
However, since our final conclusions are measurable with respect to $A = (A_{ij})$, and the law of $A$ is invariant under permutation of vertices,
they also hold for $1$ replaced with any other vertex; see in particular the proof of Lemma~\ref{lem:Gjj} below.

Remark~\ref{rk:switch3} and Lemma~\ref{lem:MM-tauij}
imply the following key result for the matching model, which is the main result of this section.
We state it in a sufficiently general form that holds for all graph models simultaneously;
the proof for the other models is given in Sections~\ref{sec:UM}--\ref{sec:PM}. 
For the matching model, the parametrization in its statement and the corresponding random
variables from \eqref{e:random_vars} were defined explicitly in Section~\ref{sec:MM}:
$a_i$ below \eqref{e:Omega_match}, $\alpha_i$ below \eqref{e:neigh_enum},
and $A$ in \eqref{e:A-M-MM}.

\begin{proposition}\label{prop:switch}
For any model of random $d$-regular graphs introduced in Section~\ref{sec:intro-models},
there exists a parametrization satisfying Definition \ref{def:parametrization},
augmented according to Definition \ref{def:parametrization2},
with $\sigF_d$-measur\-able random variables
\begin{equation} \label{e:random_vars}
a_1, \dots, a_d, \alpha_1, \dots, \alpha_d \in \qq{1,N} \,, \qquad A = (A_{ij})_{i,j \in \qq{1,N}}\,,
\end{equation}
such that the following holds.
\begin{enumerate}
\item
$A$ is the adjacency matrix of the $d$-regular random graph model under consideration,
and $(\alpha_1, \dots, \alpha_d)$ is an enumeration of the neighbours of $1$
in the sense of \eqref{e:neigh_enum}.
\item
(Neighbours of 1.)
Fix $\mu \in \qq{1,d}$. 
\begin{itemize}
\item[(1)]
Conditioned on $\sigG_\mu$, the random variable $\amu$ is approximately uniform.
\item[(2)]
Conditioned on $\sigF_0$, with probability $1-O\pb{\frac{1}{\sqrt{dD}}}$ we have $\alpha_\mu = a_\mu$.
\end{itemize}
\item
(Behaviour under resampling.)
Fix $\mu \in \qq{1,d}$.

\begin{itemize}
\item[(1)]
$\Acmu - \Amu$ is the sum of a bounded number of terms of the form $\pm \Delta_{xy}$ where
$x$ and $y$ are random variables in $\qq{1,N}$.
Conditioned on $\sigG_\mu$,
with probability $1-O\pb{\frac{1}{\sqrt{dD}}}$,
the number of such terms is constant. Conditioned on $\sigG_\mu$,
for each term $\pm \Delta_{xy}$ at least one of $x$ and $y$ is approximately uniform.
\item[(2)]
Conditioned on $\cal F_0$, with probability $1-O\pb{\frac{1}{\sqrt{dD}}}$ we have
\begin{equation}
  \label{e:HHEexp}
  \Acmu - \Amu
  \;=\; 
  \Delta_{1 \amuc} - \Delta_{1 \amu} + X\,,
\end{equation}
where $X$ is a sum of terms $\pm \Delta_{x y}$
such that one of the following two conditions holds: (a) conditioned on $\sigG_\mu$, the random variables $x$ and $y$ are both approximately uniform; or
(b) conditioned on $\sigG_\mu, \amu,\amuc$, at least one of $x$ and $y$ is approximately uniform.
(Here we abbreviated $\tilde a_\mu \equiv \tilde a_\mu^\mu$.)
\end{itemize}
\end{enumerate}
\end{proposition}

\begin{proof}[Proof of Proposition~\ref{prop:switch}: matching model]
The parametrization obeying Definition~\ref{def:parametrization}
and the random variables \eqref{e:random_vars} for the matching model
were defined in Section \ref{sec:MM}.
We augment the probability space according to Definition \ref{def:parametrization2}. 

The claim (i) follows immediately from Lemma~\ref{lem:MM-tauij}. To show (ii) and (iii), we fix $\mu \in \qq{1,d}$,
and drop the index $\mu$ from the notation and write for instance $\pi \equiv \pi_\mu$,
$a \equiv a_\mu$, and $\tilde A \equiv \tilde A^\mu$.

First, we prove (ii).
By definition, the random variable $a_\mu$ is uniform on $\qq{2,N}$ and
hence approximately uniform on $\qq{1,N}$, showing (ii)(1).
By \eqref{e:tauij-ij}, $\alpha_\mu=\sigma_\mu(1)=\amu$ holds on the event $|\{1,\pi(1),a,\pi(a),b,\pi(b)\}|=6$.
The latter event has probability $1-O(\frac{1}{N})  \geq 1-O(\frac{1}{\sqrt{dD}})$ conditioned on $\sigG_\mu$, and hence
in particular conditioned on $\sigF_0$, which proves (ii)(2).

Next, we prove (iii).
By the definitions \eqref{e:sigmamu-MM}--\eqref{e:A-M-MM},
\begin{equation*}
  \Ac - A
  \;=\; M(T_{1\ac\bc}(\pi))- M(T_{1ab}(\pi))
  \,.
\end{equation*}
By the definition of $T$ in \eqref{e:MM-Tijdef} and \eqref{e:Adoubleswitch},
any application of $T$ adds or removes at most $6$ terms $\Delta_{xy}$, and therefore
$\Ac-A$ is equal to a sum of at most 12 terms of the form $\pm \Delta_{xy}$,
which proves the first claim of (iii)(1).

To show the second claim of (iii)(1)
and to show (iii)(2), we may assume that
\begin{equation} \label{e:cond_MM_switch}
\abs{\h{1,\pi(1), a, \pi(a), b ,\pi (b)}} \;=\; 6 \,, \qquad
\abs{\h{1,\pi(1), \tilde a, \pi(\tilde a), \tilde b ,\pi (\tilde b)}} \;=\; 6\,,
\end{equation}
since this event occurs with probability at least $1 - O\pb{\frac1N} \geq 1 - O(\frac{1}{\sqrt{dD}})$
conditioned on $\cal G_\mu$ (and hence also conditioned on $\cal F_0$).
Under \eqref{e:cond_MM_switch}, we get
\begin{equation} \label{e:MM-AcA2}
  \Ac - A
  \;=\; M(T_{1\ac\bc}(\pi))- M(T_{1ab}(\pi))
  \;=\; \tau_{r1,\ac \uac, \bc\ubc}(M(\pi))- \tau_{r1,a\ua,b\ub}(M(\pi))
  \,,
\end{equation}
with $r=\pi(1)$, $\ua = \pi(a),\ub=\pi(b),\uac = \pi(\ac), \ubc = \pi(\bc)$.
As in Remark 3.3, we find that the right-hand side of \eqref{e:MM-AcA2} is
\begin{equation*}
  \Delta_{1\tilde a} + \Delta_{\pi(1) \pi(\tilde b)} + \Delta_{\tilde b \pi(\tilde a)} - \Delta_{\tilde b \pi(\tilde b)} - \Delta_{\tilde a \pi(\tilde a)}
  - \pb{\Delta_{1a} + \Delta_{\pi(1) \pi(b)} + \Delta_{b \pi(a)} - \Delta_{b \pi(b)} - \Delta_{a \pi(a)}}
  \,,
\end{equation*}
from which the claim is obvious.
\end{proof}

\subsection{Stability of the Green's function under resampling}

From now on we make use of the following notations for conditional expectations and conditional $L^p$-norms.

\begin{definition}
For any $\sigma$-algebra $\sigG$, we denote by $\E_{\sigG}=\E(\,\cdot\,|\sigG)$ and $\P_{\sigG} = \P(\,\cdot\,|\sigG)$
the conditional expectation and probability with respect to $\sigG$.
Moreover, we define the conditional $L^p$-norms by
\begin{align*}
  \|X\|_{L^p(\sigG)} &\;\deq\; \pb{ \E_{\sigG}|X|^p }^{1/p}  \quad (p\in[1,\infty)) \,,
  \\
  \|X\|_{L^\infty(\sigG)} &\;\deq\; \sup\hb{t>0 : \P_\sigG (|X|>t) >0} \,.
\end{align*}
In particular, $\|X\|_{L^p(\sigG)}$ is a $\sigG$-measurable random variable, and
\begin{equation*}
  \E_{\sigG}|X| \;=\; \|X\|_{L^1(\sigG)} \;\leq\; \|X\|_{L^\infty(\sigG)}
  \,.
\end{equation*}
Moreover, for any $\cal F_d$-measurable random variable $X \equiv X(\theta, \u_1, \dots, \u_d)$ we have
\begin{equation*}
  \|X\|_{L^\infty(\sigG_\mu)} \;=\; \max_{\u_\mu} |X(\theta, \u_1, \dots, \u_{\mu-1}, \u_\mu, \u_{\mu+1}, \dots \u_d)|
  \,.
\end{equation*}
\end{definition}

The following result is an important consequence of Proposition~\ref{prop:switch} for the Green's function.
It relies on the fundamental random control parameter
\begin{equation} \label{e:Gammamudef}
  \Gammamu \;\equiv\; \Gammamu(z)
  \;\deq\; \norm{\Gamma(z)}_{L^\infty(\sigG_\mu)}
  \,,
\end{equation}
where we recall the definition of $\Gamma(z)$ from \eqref{e:Gammadef}.
Also, we remind the reader that, according to
Definition~\ref{def:parametrization2}, a random variable (such as the
index $x$ or $y$ in the following lemma) is always defined on the
augmented probability space $\wt \Omega$, but the Green's function is
$\sigF_d$-measurable and does therefore not depend on $\uc_1, \dots, \uc_d$.

\begin{lemma} \label{lem:Gswitch} 
Fix $\mu \in \qq{1,d}$.
\begin{enumerate}
\item
For any $i,j \in \qq{1,N}$ we have
  \begin{equation} \label{e:GcmuG}
    \Gcmu_{ij}
    \;=\; \Gmu_{ij} + O(d^{-1/2}\Gammamu\Gamma)
    \,.
  \end{equation}
In particular, $\Gammamu = \Gamma + O(d^{-1/2}\Gammamu\Gamma)$,
and therefore $\Gamma \ll \sqrt{d}$ implies $\Gammamu \leq 2\Gamma$.
\item
For random variables $x,y$  such that,
conditioned on $\sigG_\mu$ and $x$, the random variable $y$ is approximately uniform,
\begin{align}
  \label{e:Gsumbd}
  \Emu \abs{\Gmu_{xy}}^2
  \;=\; O(\Gammamu^4 \Phi^2) \,.
\end{align}
An analogous statement holds with the roles of $x$ and $y$ exchanged, and with $G$ replaced by $\tilde G$.
\end{enumerate}
\end{lemma}

Assuming that $\Gamma = O(1)$, Lemma~\ref{lem:Gswitch} (i) states that the Green's function
has a bounded differences property with respect to the $u_\mu$: it
only changes by the small amount $O(d^{-1/2}) =  O(\Phi)$ if a single $u_\mu$ is changed.
Lemma~\ref{lem:Gswitch} (ii) states that if one of its indices is random, then (conditioned on $\cal G_\mu$) the $L^2$-norm of the Green's
function is smaller (again by a factor $\Phi$) than its $L^\infty$-norm.

\begin{proof}
We start with (i).
The resolvent identity \eqref{e:resolv} implies
\begin{equation} \label{e:GcmuG-pf1}
  \Gcmu_{ij}
  \;=\;
  \Gmu_{ij} + (d-1)^{-1/2} \sum_{k,l} \Gmu_{ik} (\Amu-\Acmu)_{kl} \Gcmu_{lj} \,.
\end{equation}  
By Proposition~\ref{prop:switch} (iii)(1),
$(\Amu-\Acmu)_{kl}=0$ except for a bounded number of pairs $(k,l)$,
and the non-zero entries are bounded by an absolute constant.
From this, we immediately get \eqref{e:GcmuG}.

Next, we prove (ii).
As in \eqref{e:Omega_tilde}, we may further augment the probability
space to include another independent copy of $u_\mu$, which we denote
by $\hat u_\mu$. From now on we drop the superscripts $\mu$, and
denote by $\hat X$ the version of an $\cal F_d$-measurable random variable $X$ obtained by replacing $u_\mu$ with $\hat u_\mu$.
On this augmented probability space, we introduce the $\sigma$-algebra
$\hat {\sigG}_\mu \deq \sigma(\sigG_\mu, \hat u_\mu)$.
Then, since $G$ is $\sigF_d$-measurable (i.e.\ it does not depend on $\hat u_\mu$), we have
$\E_{\sigG_\mu} f(G) = \E_{\hat {\sigG}_\mu} f(G)$ for any function $f$.
From \eqref{e:GcmuG}, with $\umuc$ replaced by $\hat u_\mu$, we get
\begin{equation*}
  G_{xy}
  \;=\; \hat G_{xy} + O(d^{-1/2}\Gammamu^2) \,,
\end{equation*}
and therefore
\begin{equation*}
  |G_{xy}|^2
  \;\leq\;
  2|\hat G_{xy}|^2 + O(d^{-1}\Gammamu^4) \,.
\end{equation*}
Since, conditioned on $\hat {\sigG}_\mu$ and $x$,
the distribution of $y$ has total variation distance $O\pb{\frac{1}{\sqrt{dD}}} \leq O\pb{\frac{1}{D}}$ to the uniform distribution on $\qq{1,N}$,
and since $|\hat G_{xy}|^2 \leq \Gammamu^2 \leq \Gammamu^4$,
the Ward identity \eqref{e:Ward} implies 
\begin{equation*}
  \Emu |G_{xy}|^2  \;=\;  \E_{\hat {\sigG}_\mu} |G_{xy}|^2 
  \;\leq\; 2\frac{\im \hat G_{xx}}{N\eta} + O(D^{-1}\Gammamu^4) \,.
\end{equation*}
Finally, by \eqref{e:GcmuG},
$\im \hat G_{xx} \leq \Gammamu + O(D^{-1/2}\Gammamu^2)$,
and therefore
\begin{equation*}
  \Emu |G_{xy}|^2
  \;\leq\; \frac{2\Gammamu}{N\eta} + \frac{O(D^{-1/2}\Gammamu^2)}{N\eta} + O(D^{-1}\Gammamu^4)
  \;=\; O\left( \sqrt{\frac{\Gammamu}{N\eta}} +\frac{1}{N\eta} + \frac{\Gammamu^2}{\sqrt{D}}\right)^2
  \,,
\end{equation*}
which yields \eqref{e:Gsumbd}.
\end{proof}

\section{Concentration}
\label{sec:concentration}

In this section we establish concentration bounds for polynomials in
the entries of $G$, with respect to the conditional expectation $\E_{\sigF_0}$.

\begin{proposition} \label{prop:concentration}
  Let $z \in \C_+$ satisfy $N\eta\geq 1$ and let $\xi, \zeta > 0$.
  Suppose that $\Gamma = O(1)$ with probability at least $1-\ee^{-\zeta}$.
  Then for any $p = O(1)$ and  $i_1,j_1, \dots, i_p,j_p \in \qq{1,N}$ we have
  \begin{equation}
    \label{e:concentration}
    G_{i_1j_1}\cdots G_{i_pj_p} - \Eu \qb{G_{i_1j_1}\cdots G_{i_pj_p}} \;=\; O(\xi\Phi)
  \end{equation}
  with probability at least $1- \ee^{-(\xi\log\xi)\wedge \zeta+O(\log N)}$.
\end{proposition}

The rest of this section is devoted to the proof of Proposition \ref{prop:concentration}.
The main tool in its proof is the following general concentration result.

\begin{proposition} \label{prop:concentration-X}
Let $X$ be a complex-valued $\sigF_d$-measurable random variable,
and $Y_1, \dots, Y_d$ nonnegative random variables such that $Y_\mu$ is $\sigG_\mu$-measurable.
Let $N$ satisfy $d \leq N^{O(1)}$.
Suppose that for all $\mu \in \qq{1,d}$ we have
\begin{equation} \label{e:concentration-X-ass}
  |X - \Emu X| \;\leq\; Y_\mu\,,\qquad
  \Emu|X-\Emu X|^2 \;\leq \; d^{-1} Y_\mu^2\,.
\end{equation}
Suppose moreover that $Y_\mu = O(1)$ with probability at least $1-\ee^{-\zeta}$,
and that $Y_\mu \leq N^{O(1)}$ almost surely.
Then
\begin{equation}
X - \E_{\sigF_0} X \;=\; O(\xi) \,,
\end{equation}
with probability at least $1-e^{-(\xi \log \xi) \wedge \zeta + O(\log N)}$.
\end{proposition}

\subsection{Proof of Proposition~\ref{prop:concentration-X}}

To prove Proposition~\ref{prop:concentration-X}, we define the complex-valued martingale
\begin{equation}
  \label{e:mart1}
  X_\mu \;\deq\; \EFmu X
 \qquad (\mu \in \qq{0,d})\,.
\end{equation}
In particular, $X_d = X$ and $X_0 = \Eu X$.
By assumption, $Y_\mu$ is bounded with probability least $1-\ee^{-\zeta}$.
By the first inequality of \eqref{e:concentration-X-ass},
we therefore get $|X_{\mu+1}-X_{\mu}| = O(1)$ with probability at least $1-\ee^{-\zeta}$.
If this bound held not only with high probability but almost surely,
a standard application of Azuma's inequality would show that
$X_d-X_0$ is concentrated on the scale $\sqrt{d}$. This bound is \emph{not} sufficient
to prove Propositions~\ref{prop:concentration}--\ref{prop:concentration-X},
which provide a significantly improved bound.
Instead of Azuma's inequality, we use Prokhorov's $\arcsinh$ inequality,
of which a martingale version is stated in the following lemma,
taken from \cite[Proposition~3.1]{MR770640}.
Compared to Azuma's inequality,
it can take advantage of an improved bound on the \emph{conditional square function}.

\begin{lemma}[Martingale $\arcsinh$ inequality] \label{lem:arcsinh}
  Let $(\sigF_\mu)_{\mu=0}^d$ be a filtration of $\sigma$-algebras and
  $(X_\mu)_{\mu=0}^d$ be a complex-valued $(\sigF_\mu)$-martingale.
  Suppose that there are deterministic constants $M,s_0, s_1, \dots, s_{d - 1}>0$ such that 
  \begin{equation} \label{e:arcsinh-ass}
    \max_{0\leq \mu < d} |X_{\mu+1}-X_{\mu}| \;\leq\; M \,, \qquad
\E_{\sigF_\mu}|X_{\mu+1}-X_{\mu}|^2 \;\leq\; s_\mu \quad (s = 0,1, \dots, d - 1)\,.
  \end{equation}
  Then
  \begin{equation} \label{e:arcsinh}
    \P(|X_d-X_0| \geq \xi) \;\leq\; 4\exp\pa{-\frac{\xi}{2\sqrt{2}M} \arcsinh\pB{\frac{M\xi}{2\sqrt{2}S}}}
    \,,
  \end{equation}
  where   $S \deq \sum_{\mu = 0}^{d - 1} s_\mu$.
  \end{lemma}

\begin{proof}
Since
\begin{equation*}
  \P\pb{|X_d-X_0| \geq \xi}
  \;\leq\;
  \P\pa{|\re (X_d-X_0)| \geq \frac{\xi}{\sqrt{2}}}
  +   \P\pa{|\im (X_d-X_0)| \geq \frac{\xi}{\sqrt{2}}}
  \,,
\end{equation*}
it suffices to prove that any real-valued martingale $X$ satisfying \eqref{e:arcsinh-ass} obeys
\begin{equation} \label{e:arcsinh-real}
  \P(|X_d-X_0| \geq \xi) \;\leq\; 2\exp\pa{-\frac{\xi}{2M} \arcsinh\pB{\frac{M\xi}{2S}}}\,.
\end{equation}
Hence, from now on, we assume that $X$ is real-valued.

First,  for all $x \in \R$, $\ee^x \leq 1+ x+ x^2 (\sinh x)/x$ and $(\sinh x)/x \leq (\sinh y)/y$ if $|x|\leq y$.
Using that $(X_\mu)$ is a martingale, it follows that for any $\lambda>0$,
\begin{equation*}
  \E_{\cal F_\mu}  \ee^{\lambda (X_{\mu + 1}-X_\mu)} \;\leq\; 1 + \E_{\cal F_\mu} (X_{\mu+1}-X_\mu)^2 \frac{\lambda}{M} \sinh \lambda M
  \;\leq\; 1 + s_\mu \frac{\lambda}{M} \sinh \lambda M
  \,.
\end{equation*}
Iterating this bound, using $1+x\leq \ee^x$, it follows that
\begin{equation*}
  \E \ee^{\lambda (X_\mu-X_0)} \;\leq\; \exp\pa{\frac{\lambda}{M} \sinh(\lambda M) S} \,.
\end{equation*}
The estimate  \eqref{e:arcsinh-real} then follows by the exponential Chebyshev
inequality with the choice $\lambda \deq \frac{1}{M} \arcsinh \pb{\frac{M \xi}{2 S}}$,
and an application of the same estimate with $X$ replaced by $-X$.
\end{proof}

In order to exploit the fact that $Y_\mu = O(1)$ with high probability, we introduce a stopping time $\tau$.
Let $\gamma \geq 1$ be the implicit constant in the assumption of Proposition~\ref{prop:concentration-X}
such that $Y_\mu \leq \gamma$ holds with probability at least $1-\ee^{-\zeta}$.
We define
\begin{equation} \label{e:taudef}
  \tau \;\deq\; \min \hb{ \mu \in \qq{0,d-1} \col \|Y_{\mu+1}\|_{L^2(\sigF_\mu)} \geq 2\gamma}\,,
\end{equation}
and if the above set is empty we set $\tau \deq d$. By definition, $\tau$ is an $(\sigF_\mu)$-stopping time.
The following result shows that $\tau < d$ on an event of low probability.

\begin{lemma} \label{lem:taubd}
  Suppose that for all $\mu \in \qq{1,d}$ we have
  $\P\pb{Y_\mu \geq \gamma} \leq \ee^{-\zeta}$ and $Y_\mu \leq N^{O(1)}$ almost surely.
  Then
  \begin{equation*}
    \P\pb{\tau < d} \;\leq\; \ee^{-\zeta + O(\log N)}
    \,.
  \end{equation*}
\end{lemma}

\begin{proof}
For $\mu \in \qq{0,d-1}$ set $\indbad_{\mu} \deq \ind{Y_{\mu + 1} \geq \gamma}$.
Then $Y_{\mu+1} \leq \gamma + N^{O(1)} \indbad_\mu$,
and, by Minkowski's inequality,
\begin{equation*}
  \|Y_{\mu+1}\|_{L^2(\sigF_\mu)}
  \;\leq\; \gamma + N^{O(1)} (\EFmu \indbad_\mu)^{1/2}
  \,.
\end{equation*}
Using a union bound, Markov's inequality, $\log d = O(\log N)$,  $\gamma \geq 1$, and that $\E\EFmu \indbad_\mu = \E\indbad_\mu \leq \ee^{-\zeta}$ by assumption,
we therefore get
  \begin{multline*}
    \P(\tau < d)
    \;\leq\; \sum_{\mu=0}^{d-1} \P\pb{\|Y_{\mu+1}\|_{L^2(\sigF_\mu)} \geq 2\gamma}
    \;\leq\; \sum_{\mu=0}^{d-1} \P\pb{\EFmu N^{O(1)} \indbad_\mu  \geq \gamma^2}
    \\
    \leq\; d N^{O(1)} \E\indbad_\mu
    \;\leq\; \ee^{-\zeta + O(\log N)}
    \,,
  \end{multline*}
  which concludes the proof. 
\end{proof}

Since $\tau$ is an $(\sigF_\mu)$-stopping time, $X_\mu^\tau \deq X_{\mu \wedge \tau}$ is an $(\sigF_\mu)$-martingale.
Because of Lemma~\ref{lem:taubd} and using a union bound, it will be sufficient to
study $X^\tau_\mu$ instead of $X_\mu$.
The next result shows that $X_\mu^\tau$ satisfies the assumptions of Lemma~\ref{lem:arcsinh}.

\begin{lemma} \label{lem:conc-stop}
For $\mu \in \qq{0,d-1}$ we have
  \begin{align}
    \label{e:Burk-input-max}
    |X_{\mu+1}^\tau-X_{\mu}^\tau| &\;=\; O(1) \,,
    \\
    \label{e:Burk-input-cond}
    \EFmu |X^\tau_{\mu+1}-X^\tau_{\mu}|^2 &\;=\; O(d^{-1}) \,.
  \end{align}
\end{lemma}

\begin{proof}
Set $\indgood_{\mu} \deq \ind{\tau\geq\mu+1}$. Then $\indgood_{\mu}$ is $\sigF_\mu$-measurable and
\begin{equation*}
  X_{\mu+1}^\tau - X_{\mu}^\tau
  \;=\;
  \indgood_{\mu}
  (X_{\mu+1} - X_{\mu})
  \;=\;
  \indgood_{\mu}
  (\EFmup X - \EFmu X)
  \,.
\end{equation*}
Note that, by definition, $\indgoodmu=1$ implies that $\|Y_{\mu+1}\|_{L^2(\sigF_\mu)} \leq 2\gamma  = O(1)$,
and that, by independence,
\begin{equation} \label{e:XmuE}
  \indgood_{\mu} (X_{\mu+1}-X_{\mu})
  \;=\; \indgood_{\mu} \EFmup (X - \Emup X)
  \,.
\end{equation}

We now prove \eqref{e:Burk-input-max}.
By the first bound of \eqref{e:concentration-X-ass},
\begin{equation*}
  |X - \Emup X| \;\leq\; Y_{\mu+1} \,,
\end{equation*}
and therefore
\begin{equation*}
  \indgood_{\mu} |X_{\mu+1}-X_{\mu}|
  \;\leq\; \indgood_{\mu} \EFmup |X - \Emup X|
  \;\leq\; \indgood_{\mu} \EFmup Y_{\mu+1}
  \;\leq\; 2\gamma
  \,.
\end{equation*}
In the last inequality, we used that
$\indgoodmu\EFmup Y_{\mu+1} = \indgoodmu \EFmu Y_{\mu+1} \leq \indgoodmu \|Y_{\mu+1}\|_{L^2(\sigF_\mu)} \leq 2\gamma = O(1)$
since $Y_{\mu+1}$ is $\sigG_{\mu+1}$-measurable, by H\"older's inequality, and by the definition of $\indgoodmu$.
This completes the proof of \eqref{e:Burk-input-max}.

Next, we prove \eqref{e:Burk-input-cond} in a similar fashion. 
By \eqref{e:XmuE}, 
Jensen's inequality for the conditional expectation $\EFmup$,
and then using the second inequality of \eqref{e:concentration-X-ass},
we get
\begin{equation*}
  \EFmu|X^\tau_{\mu+1}-X^\tau_{\mu}|^2
  \;=\;
  \indgood_{\mu} \EFmu \abs{X - \Emup X}^2 
  \;\leq\;
  d^{-1} \indgood_{\mu} \EFmu Y_{\mu+1}^2
  \;\leq\;
  4 \gamma^2 d^{-1}
  \,,
\end{equation*}
as desired.
\end{proof}

\begin{proof}[Proof of Proposition~\ref{prop:concentration-X}]
  By Lemmas~\ref{lem:arcsinh}--\ref{lem:conc-stop}, and $\xi \arcsinh \xi = \xi \log 2\xi + O(1)$ for $\xi > 0$, we get
  \begin{equation*}
    \P\pb{|X - \Eu X| \geq C\xi}
    \;\leq\;
    \P\pb{|X^\tau_d-X^\tau_0| \geq C \xi } + P\p{\tau < d}
    \;\leq\; \ee^{-(\xi\log\xi)\wedge \zeta + O(\log N)}
    \,,
  \end{equation*}
  for a sufficiently large constant $C$.
\end{proof}

\subsection{Proof of Proposition~\ref{prop:concentration}}

Throughout the remainder of this section, we assume that $N\eta \geq 1$ and $D \geq 1$.
From Definitions~\ref{def:parametrization}--\ref{def:parametrization2} we recall
the $\sigma$-algebras $\sigG_\mu$ and $\sigF_\mu$,
as well as the version $\Xcmu$ of a
random variable $X$.
In particular, we can express the conditional variance 
of an $\sigF_d$-measurable complex-valued random variable $X$ as
\begin{equation}
  \label{e:varindep}
  \E_{\sigG_\mu} \abs{X-\E_{\sigG_\mu}X}^2
  \;=\; \frac12 \, \E_{\sigG_\mu} \abs{\Xmu-\Xcmu}^2
  \,.
\end{equation}

The following result is the main ingredient in the verification of the second bound of \eqref{e:concentration-X-ass}.
For its statement, we recall the definition of $\Gammamu$ from \eqref{e:Gammamudef}. 

\begin{lemma}\label{lem:EGdiff}
We have
\begin{equation} \label{e:EGdiff}
\Emu \abs{G_{ij} - \Emu  G_{ij}}^2
\;=\; O\pb{d^{-1}\Gammamu^6\Phi^2} 
\,.
\end{equation}
\end{lemma}

\begin{proof}
We abbreviate $\Gc \equiv \Gcmu$. 
Applying \eqref{e:varindep} to $G_{ij}$, we get
\begin{equation}  \label{e:EGdifff}
\Emu \abs{G_{ij} - \Emu  G_{ij}}^2
\;\leq\;
\Emumuc \abs{G_{ij} - \Gc_{ij}}^2
\,.    
\end{equation}%
Let $\chi$ be the indicator function of the event of $\cal G_\mu$-probability at least $1 - O\pb{\frac{1}{\sqrt{dD}}}$
from Proposition~\ref{prop:switch} (iii)(1), and set $\bar\chi=1-\chi$.
Then the right-hand side of \eqref{e:EGdifff} is bounded by
\begin{equation} \label{e:EGdiff12}
\Emumuc \abs{(G_{ij} - \Gc_{ij})\chi}^2
+ \norm{G_{ij} - \Gc_{ij}}_{L^\infty(\sigG_\mu)}^2 \Emumuc (\bar\chi)
\,.
\end{equation}

To estimate both terms, we use that, by the resolvent identity and Proposition~\ref{prop:switch} (iii)(1),
there are a bounded (and possibly random) number  $\ell$ of random variables $(x_1,y_1), \dots, (x_\ell,y_\ell)$ such that
\begin{equation} \label{e:EGdiffresolv}
  |G_{ij} - \Gc_{ij}|
  \;\leq\; (d-1)^{-1/2} \sum_{k,l=1}^N |G_{ik} (A-\Ac)_{kl} \Gc_{lj}|
  \;=\; (d-1)^{-1/2} \sum_{p=1}^\ell |G_{ix_p} \Gc_{y_pj}|
  \,.
\end{equation}

We focus first on the second term of \eqref{e:EGdiff12}.
By Proposition~\ref{prop:switch} (iii)(1)  and \eqref{e:EGdiffresolv},
\begin{equation*}
  \norm{G_{ij} - \Gc_{ij}}_{L^\infty(\sigG_\mu)} 
  \;=\;
  O(d^{-1/2} \Gammamu^2)
  \;=\;
  O(d^{-1/2} \Gammamu^3)
  \,.
\end{equation*}
By the definition of $\bar\chi$ and Proposition~\ref{prop:switch} (iii)(1),
$\Emumuc(\bar\chi) = O(\frac{1}{\sqrt{dD}})= O(\Phi^2)$. This implies that the second term
in \eqref{e:EGdiff12} is bounded by the right-hand side of \eqref{e:EGdiff}.

Next, we estimate the first term of \eqref{e:EGdiff12}.
By the definition of $\chi$ and Proposition~\ref{prop:switch} (iii)(1),
the number $\ell$ in \eqref{e:EGdiffresolv} is constant on the support of $\chi$,
and, conditioned on $\sigG_\mu$, for each $p \in \qq{1,\ell}$,
at least one of $x_p$ and $y_p$ is approximately uniform.
Therefore
\begin{equation} \label{e:EGdiffsq}
  \Emu \abs{(G_{ij} - \Gc_{ij})\chi}^2 
  \;\leq\;
  \frac{1}{d-1}
  \sum_{p,q=1}^\ell
  \Emumuc
    |G_{ix_p} \Gc_{y_p j} G_{ix_q} \Gc_{y_q j}|
    \,,
\end{equation}
where,
conditioned on $\sigG_\mu$, for each $(p,q)$,
at least two of $x_p,y_p, x_q,y_q$ are approximately uniform.
We estimate two of the four factors of $G$ or $\Gc$ by $\Gammamu$,  including those without an approximately uniform index, and
use the Cauchy-Schwarz inequality to decouple the remaining two factors of $G$ or $\Gc$,
each of which has at least one approximately uniform index.
Then using \eqref{e:Gsumbd} we find that each such term is bounded by $O(\Gammamu^6\Phi^2)$.
Since the sum in  \eqref{e:EGdiffsq} has a bounded number of terms, the claim follows.
\end{proof}

\begin{proof}[Proof of Proposition \ref{prop:concentration}]
We verify the assumptions of Proposition \ref{prop:concentration-X}.
Given $p = O(1)$, set $Y_{\mu} \deq C_p\Gammamu^{2+p}$ for a sufficiently large constant $C_p$.
By definition, $Y_\mu$ is $\sigG_\mu$-measurable.
Moreover, by assumption, $\Gamma = O(1)$ with probability at least $1-\ee^{-\zeta}$.
Hence, Lemma~\ref{lem:Gswitch} (i) implies that $\Gammamu \leq 2 \Gamma = O(1)$ with probability at least $1-\ee^{-\zeta}$,
so that $Y_\mu = O(1)$ with probability at least $1-\ee^{-\zeta}$.
Moreover, the trivial bound \eqref{e:Getabd} and $N\eta \geq 1$ imply
$Y_\mu = O(\eta^{-2-p}) = O(N^{2+p}) = N^{O(1)}$.
We conclude that $Y_\mu$ satisfies the conditions from the statement of Proposition~\ref{prop:concentration-X}.

We first complete the proof for $p=1$.
Let $X \deq \Phi^{-1} G_{ij}$.
Then, by \eqref{e:GcmuG}  and $\Phi^{-1} \leq d^{1/2}$,
  \begin{equation} \label{e:XEmuXinfbd}
    |X-\Emu X| \;\leq\;  \Phi^{-1} |G_{ij} - \Emu G_{ij}| \;=\; \Phi^{-1} O(d^{-1/2} \Gammamu^2) \;=\; O(\Gammamu^2) \;\leq\; Y_{\mu}\,,
  \end{equation}
assuming that the constant $C_p$ was chosen sufficiently large.
This establishes the first estimate of \eqref{e:concentration-X-ass}. The second estimate of \eqref{e:concentration-X-ass}
follows from Lemma~\ref{lem:EGdiff}.
Therefore Proposition~\ref{prop:concentration} follows from Proposition~\ref{prop:concentration-X}.

Next, we deal with the case of general $p$.
For $k \in \qq{1,p}$ abbreviate $Q_k \deq G_{i_k j_k}$ and consider $X \deq \Phi^{-1} Q_1 \cdots Q_p$.
By telescoping, $\Phi(X - \Emu X)$ is equal to
\begin{equation*}
  \sum_{k=1}^{p}
  \Bigl[
Q_1 \cdots Q_{k-1}
  (Q_k -\Emu Q_k)
  \Emu (Q_{k+1} \cdots   Q_p)
-
Q_1 \cdots Q_{k-1}
  \Emu\pb{ (Q_k -\Emu Q_k)
  (Q_{k+1} \cdots   Q_p)}
  \Bigr]
  \,,
\end{equation*}
and therefore
\begin{equation*}
  |X - \Emu X|
  \;\leq\;
  \Phi^{-1} \Gammamu^{p-1} \sum_{k=1}^{p}
  \qB{
    |Q_k -\Emu Q_k|
    + \Emu     |Q_k -\Emu Q_k|}
  \,.
\end{equation*}
Using \eqref{e:GcmuG},
we therefore conclude that $|X-\Emu X| \leq Y_\mu$ (after choosing $C_p$ large enough).
Moreover, since $(a_1+\cdots+a_{2p})^2 \leq (2p)^2(a_1^2 + \dots + a_{2p}^2)$, by the conditional Jensen inequality
and Lemma~\ref{lem:EGdiff}, we find
\begin{equation*}
  \Emu \abs{X - \Emu X}^2
  \;\leq\;
  O(p^2) \Phi^{-2} \Gammamu^{2p-2} \max_{i,j}
    \Emu \abs{G_{ij}-\Emu G_{ij}}^2
  \;\leq\; O\pbb{\frac{p^2}{d}} \Gammamu^{2p+4} \,,
\end{equation*}
which is bounded by $d^{-1} Y_\mu^2$ (after choosing $C_p$ large enough).
The claim now follows from Proposition~\ref{prop:concentration-X}.
\end{proof}

\section{Expectation}
\label{sec:expectation}

In this section we prove Proposition~\ref{prop:Lambda}. We use the spectral parameters
\begin{equation} \label{e:def_z0_z}
z_0 \;=\; E + \ii \eta_0 \,, \qquad  z \;=\; E + \ii \eta \,, \qquad \xi^2/N \;\ll\; \eta_0\;\leq\; \eta\;\leq\; N\,.
\end{equation}
Fix $z_0$ as in \eqref{e:def_z0_z}. To prove Proposition \ref{prop:Lambda}, we assume that $D \gg \xi^2$ and that
\begin{equation} \label{e:Gamma_star_assump}
\P(\Gamma^*(z_0) \geq \gamma) \;=\; \P \pb{\max \hb{\Gamma(E + \ii \eta) \col \eta \geq \eta_0} \geq \gamma} \;\leq\; \ee^{-\zeta}
\end{equation}
for some constant $\gamma=O(1)$.
Recall the function $F \equiv F_z$ from \eqref{e:Fdef} and $\Phi$ from \eqref{e:Phidef}.
To prove Proposition~\ref{prop:Lambda} it suffices to show that, with probability at least $1-\ee^{-(\xi\log\xi)\wedge \zeta+O(\log N)}$,
\begin{align}
  \label{e:Lambdad}
  \max_i |G_{ii} - m|
  &\;=\;
  O(F(\xi\Phi))
  \,,
  \\
  \label{e:Lambdao}
  \max_{i\neq j} |G_{ij}|
  &\;=\;
  O(\xi\Phi)
  \,,
\end{align}
for any $z$ satisfying \eqref{e:def_z0_z}.

The proof of \eqref{e:Lambdad}--\eqref{e:Lambdao} proceeds in the following steps:
\begin{enumerate}
\item
Estimate of $s - m$, where
$s$ is the Stieltjes transform \eqref{e:sdef} of the empirical spectral measure,
and $m$ the Stieltjes transform \eqref{e:mdef} of the semicircle law.
\item
Estimate of $G_{ii} - m$\,.
\item
Estimate of $G_{ij}$ for $i \neq j$.
\end{enumerate}
Step (i) represents most of the work. Throughout this section we make the assumption \eqref{e:Gamma_star_assump}.

\subsection{High probability a priori bounds}

For the proof of Proposition~\ref{prop:Lambda} we use the following convenient notion of high probability.

\begin{definition} \label{defn:highprob}
  Given a parameter $t>0$,
  an event $\goodev$ holds with $t$-\emph{high probability}, abbreviated $t$-HP,
  if $\P(\goodev^c) \leq \ee^{-t+O(\log N)}$.
\end{definition}

In the nontrivial case $t \gg \log N$, the notion of $t$-high probability is stronger than the standard notion of
high probability (and in fact implies what is occasionally called overwhelming probability).
By definition and a union bound,
an intersection of $N^{O(1)}$ many events that each hold with $t$-high probability holds with $t$-high probability.
Moreover, if $\goodev$ holds with $t$-HP then $\Eu \ind{\badev} \leq N^{-k}$
with $t$-HP for any constant $k > 0$.
Indeed, by Markov's inequality,
\begin{equation} \label{e:Euhighprob}
  \P\pb{\Eu\ind{\Xi^c} > 1/N^k} \;\leq\; N^k \E\Eu\ind{\badev} \;=\; e^{k\log N} \P(\badev) \;\leq\; \ee^{-t + O(\log N)}
  \,.
\end{equation}
From now on, these properties will be used tacitly.

Furthermore, from now on, the parameter $t$ in Definition~\ref{defn:highprob}
will always be
\begin{equation}
  t \;\deq\; (\xi\log\xi)\wedge \zeta
\end{equation}
with $\zeta$ and $\xi$ the parameters given in the assumption of Proposition~\ref{prop:Lambda}.
Then,  for any $z$ as in \eqref{e:def_z0_z}, we get from the assumption \eqref{e:Gamma_star_assump}
and Proposition~\ref{prop:concentration} that,
with $t$-HP, for all deterministic $i,j,k,l,m,n \in \qq{1,N}$,
\begin{equation} \label{e:Gjj-conc1}
  |G_{ij}| \;=\; O(1)\,, \quad
  \Eu G_{ij} \;=\; G_{ij} + O(\xi\Phi)\,,
\end{equation}
and
\begin{equation} \label{e:Gjj-conc2}
  \Eu(G_{ij}G_{kl}) \;=\; G_{ij}G_{kl} + O(\xi\Phi)\,, \quad
  \Eu(G_{ij}G_{kl}G_{mn}) \;=\; G_{ij}G_{kl}G_{mn} + O(\xi\Phi)
  \,.
\end{equation}
To prove Proposition~\ref{prop:Lambda}, we need to show that \eqref{e:Lambdad}--\eqref{e:Lambdao}
then also hold with $t$-HP.

\subsection{Derivation of self-consistent equation}

In this subsection we derive the self-consistent equation, \eqref{e:Gjj} below,
which will allow us to obtain estimates on the entries of $G$ and hence prove Proposition \ref{prop:Lambda}.
The following lemma is, in combination with the concentration bounds \eqref{e:Gjj-conc1}--\eqref{e:Gjj-conc2},
the main estimate in its derivation. For its statement, recall from Proposition \ref{prop:switch} that
$(\alpha_1, \dots, \alpha_d)$ is an enumeration of the neighbours of $1$.
For the following we introduce the abbreviation
\begin{equation} \label{e:Ei}
  \Ei F(i) \;\deq\; \frac{1}{N}\sum_i F(i) \,,
\end{equation}
so that, under $\Ei$, $i$ is regarded as a uniform random variable
that is independent of all other randomness. With this notation,
we may express the the Stieljes transform \eqref{e:sdef} of the
empirical spectral measure as $s = \E^{[i]} G_{ii}$.

\begin{lemma} \label{lem:Gdiff}
Fix $\mu \in \qq{1,d}$.
Given $z \in \C_+$ with $N\eta \geq 1$, suppose that $\Gamma = O(1)$ with $t$-HP.
Then for all fixed $j,k,l \in \qq{1,N}$, 
\begin{align}
\label{e:Gdiff1}
\Eu\pB{G_{\alpha_\mu j} - \Ei G_{i j} + (d-1)^{-1/2} s G_{1j}}
&\;=\;
O(d^{-1/2} \Phi)
\,,
\\
\label{e:Gdiff2}
\Eu\pB{G_{kl}\pB{G_{\alpha_\mu j} - \Ei G_{i j} + (d-1)^{-1/2} s G_{1j}}}
&\;=\;
O(d^{-1/2} \Phi)
\,,
\end{align}
with $t$-HP.
\end{lemma}

Recall from \eqref{e:Xc} that $\tilde X^\mu$ is the version
of a random variable $X \equiv X(\theta, u_1, \dots, u_d, \uc_1, \dots, \uc_d)$
obtained from $X$ by exchanging its arguments $u_\mu$ and $\uc_\mu$.
Throughout this section we make use of the indicator function
\begin{equation} \label{e:chimudef}
\chi_\mu \;\deq\; \psi_\mu \tilde \psi_\mu^\mu
\quad \text{where }
\psi_\mu \;\deq\; \ind{\alpha_\mu = a_\mu} \, \indb{\text{claim \eqref{e:HHEexp} from Proposition \ref{prop:switch} holds}}
\,.
\end{equation}
Note that $\chi_\mu = \tilde \chi^\mu_\mu$.
Moreover, by Proposition~\ref{prop:switch} (ii)(2) and (iii)(2) as well as a union bound,
we have $\E_{\sigF_0}(\chi_\mu) = 1-O\pb{\frac{1}{\sqrt{dD}}}$.

For brevity, given a fixed index $\mu \in \qq{1,d}$, we
often drop sub- and superscripts $\mu$, and write simply
\begin{equation} \label{e:muabbr}
  \alpha\;\equiv\;\alpha_\mu\,,\quad
  a\;\equiv\;a_\mu\,,\quad
  \ac\;\equiv\;\ac_\mu\,,\quad
  \chi\;\equiv\;\chimu\,,\quad
  \Ac \;\equiv\; \Acmu\,,\quad
  \Gc \;\equiv\; \Gcmu\,.
\end{equation}
(As in Proposition \ref{prop:switch}, we always abbreviate $\tilde a_\mu \equiv \tilde a_\mu^\mu$.)

The following lemma provides several elementary bounds on the Green's function.
It is the main computational tool in the proof of Lemma \ref{lem:Gdiff}.

\begin{lemma} \label{lem:Gmanip}
Given $z \in \C_+$ with $N\eta \geq 1$, suppose that $\Gamma = O(1)$ holds with $t$-HP.
Fix $\mu \in \qq{1,d}$, and use the abbreviations \eqref{e:muabbr}.
Then the following estimates hold with $t$-HP.

\begin{enumerate}
\item
For all $j \in \qq{1,N}$ we have
\begin{align}
  \label{e:Gac1}
  \Eu \Emuc (G_{\ac j})
  &\;=\;  \Eu \Ei (G_{ij}) 
  + O\pb{\tfrac{1}{\sqrt{dD}}} 
  \,,
  \\
  \label{e:GacacG11}
  \Eu\Emuc (G_{\ac\ac} G_{jj})
  &\;=\; \Eu\Ei (G_{ii} G_{jj})
  + O\pb{\tfrac{1}{\sqrt{dD}}} 
  \,.
\end{align}
\item
For all $i,j,k,l,m,n \in \qq{1,N}$ we have
\begin{align}
  \label{e:Gchi}
  \Eu \Emuc (\chi G_{ij})
  &\;=\;
  \Eu \Emuc (G_{ij})
  + O\pb{\tfrac{1}{\sqrt{dD}}} 
  \,,
  \\
  \label{e:GGchi}
  \Eu\Emuc(\chi G_{ij}G_{kl})
  &\;=\;
  \Eu\Emuc(G_{ij}G_{kl})
  + O\pb{\tfrac{1}{\sqrt{dD}}} 
  \,,
  \\
  \label{e:GGGchi}
  \Eu\Emuc(\chi G_{ij}G_{kl}G_{mn})
  &\;=\;
  \Eu\Emuc(G_{ij}G_{kl}G_{mn})
  + O\pb{\tfrac{1}{\sqrt{dD}}} 
  \,.
\end{align}
Analogous statements hold if some factors $G$ are replaced with $\tilde G$.
\item
For any $i,j,k,l,m,n \in \qq{1,N}$ we have
\begin{align}
  \label{e:GGuc}
  \Eu\Emuc (G_{ij} \Gc_{kl})
  &\;=\; \Eu\Emuc  (G_{ij} G_{kl})
  + O\pb{\tfrac{1}{\sqrt{D}}}
  \,,
  \\
  \label{e:GGGuc}
  \Eu\Emuc (G_{ij} \Gc_{kl} G_{mn})
  &\;=\; \Eu\Emuc  (G_{ij} G_{kl} G_{mn})
  + O\pb{\tfrac{1}{\sqrt{D}}}
  \,.
  \end{align}
  \item
  If
(a) conditioned on $\sigG_\mu$ and $\ac$, the random variable $x$ is approximately uniform,
or
(b) conditioned on $\sigG_\mu$, the random variable $y$ is approximately uniform, then
\begin{align}
  \label{e:GacxGy1}
  \Eu\Emuc(G_{\ac x}\Gc_{y1}) 
  &\;=\;
  O\p{\Phi}
  \,,
  \\
  \label{e:GGacxGy1}
  \Eu\Emuc(G_{ij} G_{\ac x}\Gc_{y1})
  &\;=\;
  O\p{\Phi}
  \,.
\end{align}
\end{enumerate}
\end{lemma}

\begin{proof}
Fix $\mu \in \qq{1,d}$, and, as in the statement of the lemma,
use the shorthand notation \eqref{e:muabbr}.
Denote by $\indgood$ the indicator function of the event 
$\Gammamu \leq 2\gamma$, and set $\indbad = 1-\indgood$.
By definition, $\phi$ is $\sigG_\mu$-measurable.
By \eqref{e:GcmuG}, $\h{\Gammamu \leq 2\gamma} \subset \h{\Gamma \leq \gamma}$,
so that, by assumption, $\indgood=1$ with $t$-HP, for a constant $\gamma=O(1)$.
In particular, as noted around \eqref{e:Euhighprob},
for any constant $k$,
the event $\hb{\eta^{-k} \Eu \indbad \leq \frac{1}{N} \leq \frac{1}{\sqrt{dD}}}$ holds with $t$-HP.

\smallskip\noindent
(i)
We show \eqref{e:Gac1}; the proof of \eqref{e:GacacG11} is analogous.
Since, conditioned on $\sigF_0$, $\tilde a$ and $G$ are independent,
and the total variation distance between the distribution of $\tilde a$
and the uniform distribution on $\qq{1,N}$ is $O\pb{\frac{1}{\sqrt{dD}}}$,
\begin{align}
  \Eu\Emuc (G_{\ac j})
  &\;=\; \Eu\Emuc (G_{\ac j}\phi) + O\p{\Eu(\eta^{-1} \indbad)}
  \nonumber\\
  &\;=\; \Eu\Ei (G_{ij}\phi) + O\pb{\tfrac{1}{\sqrt{dD}}}
  \;=\; \Eu\Ei (G_{ij}) + O\pb{\tfrac{1}{\sqrt{dD}}}
  \,,
\end{align}
with $t$-HP.

\smallskip \noindent
(ii)
We show \eqref{e:GGchi}; the proofs of \eqref{e:Gchi} and \eqref{e:GGGchi} are analogous.
Since $\chi \leq 1$,
\begin{equation}
  \Eu\Emuc((1-\chi) G_{ij}G_{kl})
  \;=\;\Eu\Emuc(\indgood(1-\chi) G_{ij}G_{kl})
   + O\p{\Eu(\eta^{-2}\indbad)}
  \,.
\end{equation}
The first term is bounded by $O\pb{\Eu\Emuc(1-\chi)} = O\pb{\frac{1}{\sqrt{dD}}}$ by the definition of $\chi$
and Proposition~\ref{prop:switch}.
The second term is also $O\pb{\frac{1}{\sqrt{dD}}}$ with $t$-HP,
as observed at the beginning of the proof.

\smallskip\noindent
(iii)
We show \eqref{e:GGuc}; the proof of \eqref{e:GGGuc} is analogous.
Since $\eta^{-2} \Eu\indbad = O(\frac{1}{\sqrt{D}})$ with  $t$-HP  
and $\indgood |G_{ij}| = O(1)$, we get from \eqref{e:GcmuG} that
\begin{align}
  \Eu\Emuc
  (G_{ij} G_{kl})
  \;=\;
  \Eu\Emuc
  (\indgood G_{ij} G_{kl})
  + O\pb{\tfrac{1}{\sqrt{D}}}
  &\;=\;
  \Eu\Emuc
  (\indgood G_{ij} \Gc_{kl})
  + O\pb{\tfrac{1}{\sqrt{D}}}
  \nonumber\\
  &\;=\;
  \Eu\Emuc
  (G_{ij} \Gc_{kl})
  + O\pb{\tfrac{1}{\sqrt{D}}}
  \,,
\end{align}
 with $t$-HP.

\smallskip\noindent
(iv) We show \eqref{e:GacxGy1}; the proof of \eqref{e:GGacxGy1} is analogous.
 Under assumption (a),
the Cauchy-Schwarz inequality and \eqref{e:Gsumbd} imply
\begin{align} \label{e:Gtildeaa11bd}
  |\Eu\Emuc(G_{\ac x}\Gc_{y1})|
  &\;\leq\;
  \Eu\indgood \Emuc|G_{\ac x}\Gc_{y1}|
  +
  \eta^{-2}\Eu\indbad
  \nonumber\\
  &\;\leq\;
  O(\Eu \indgood \Emu \Emuc|G_{\ac x}|^2)^{1/2}
  +
  \eta^{-2}\Eu\indbad
  \;=\;
  O\p{\Phi} 
  \,,
\end{align}
with $t$-HP, where we used $\indgood\Emu |\Gc_{y1}|^2 \leq \indgood\Gammamu^2 \leq 4\gamma^2 =O(1)$.
Similarly, under assumption (b),
\begin{equation}
  |\Eu\Emuc(G_{\ac x}\Gc_{y1})|
  \;\leq\;
  O(\Eu \indgood \Emu \Emuc|\Gc_{y1}|^2)^{1/2}
  +
  \eta^{-2}\Eu\indbad
  \;=\;
  O\p{\Phi}
  \,,
\end{equation}
with $t$-HP, where we again used \eqref{e:Euhighprob}.
This completes the proof.
\end{proof}

\begin{proof}[Proof of Lemma \ref{lem:Gdiff}]
The proofs of both estimates are analogous, and we only prove \eqref{e:Gdiff1}.
Throughout the proof, we use Lemma~\ref{lem:Gmanip} repeatedly, and estimate
$\frac{1}{\sqrt{dD}} \leq \frac{1}{\sqrt{d}}\Phi$.
Since $\mu \in \qq{1,d}$ is fixed, we also use the abbreviations \eqref{e:muabbr}
in the remainder of the proof,
and use the indicator function $\chi \equiv \chi_\mu$ defined in \eqref{e:chimudef}.
By definition, conditioned on $\sigF_0$, the random variables $u$ and $\uc$ are identically distributed,
so that $G_{aj}$ and $\Gc_{\ac j}$ are also identically distributed. (Recall the definition \eqref{e:Xc} and the convention \eqref{e:muabbr}.)
Thus, by \eqref{e:Gchi}, and since $\alpha=a$ on the support of $\chi$ (by definition \eqref{e:chimudef} of $\chi$), we obtain
\begin{equation}
  \Eu(G_{\alpha j})
  \;=\; \Eu\Emuc(\chi G_{\alpha j}) + O\p{d^{-1/2}\Phi}
  \;=\; \Eu\Emuc(\chi \Gc_{\ac j}) + O\p{d^{-1/2}\Phi}
  \,,
\end{equation}
with $t$-HP,
where in the last step we also used that $\chi = \tilde \chi^\mu$.
By \eqref{e:Gac1} and \eqref{e:Gchi}, with $t$-HP,
\begin{equation}
  \Eu\Ei(G_{ij})
  \;=\; \Eu\Emuc(\chi G_{\ac j})
  + O\p{d^{-1/2}\Phi}
  \,.
\end{equation}
This implies, with $t$-HP,
\begin{equation} \label{e:1zEG11}
\Eu(G_{\alpha j} - \Ei G_{i j})
\;=\; \Eu \Emuc \chi \p{\Gc_{\ac j} - G_{\ac j}}
+ O\p{d^{-1/2}\Phi}
\,.
\end{equation}

By the resolvent identity, 
$\Gc - G = (d-1)^{-1/2} G (A-\Ac)\Gc$,
and therefore by Proposition~\ref{prop:switch} (iii)(2), on the event $\{\chi=1\}$ we have
\begin{equation} \label{e:GDGi1}
  \Gc_{\ac j} - G_{\ac j} \;=\;
  (d-1)^{-1/2} \pb{- G_{\ac\ac}\Gc_{1j} + S }
  \,,
\end{equation}
where $S$ a sum of a bounded number of terms of the form $ \pm G_{\ac x}\Gc_{y j}$
with random variables $x$ and $y$
such that at least one of the following two conditions is satisfied:
conditioned on $\sigG_\mu$ and $\ac$, the random variable $x$ is approximately uniform, or,
conditioned on $\sigG_\mu$, the random variable $y$ is approximately uniform.
(For example, $S$ contains the term $G_{\ac a} \Gc_{1j}$ corresponding to $(x,y)=(a,1)$.
Conditioned on $\sigG_\mu$, the random variable $x=a$ is approximately uniform
and independent of $\ac$, so that $x$ is approximately uniform conditioned on $\sigG_\mu$ and $\ac$.)
Therefore, by \eqref{e:GGchi}, \eqref{e:GGuc}, and \eqref{e:GacxGy1}, we get
\begin{equation} \label{e:EX}
  |\Eu\Emuc (\chi S)|
  \;\leq\; \Eu\Emuc |S| + O(\Phi) 
  \;=\; O(\Phi) \,,
\end{equation}
with $t$-HP.
Similarly, by \eqref{e:GGchi}, \eqref{e:GGuc}, and \eqref{e:GacacG11}, we get
\begin{align} \label{e:GiiG11}
  \Eu\Emuc (\chi G_{\ac\ac} \Gc_{1j})
  &\;=\;
  \Eu\Emuc (G_{\ac\ac} \Gc_{1j}) + O(\Phi) 
  \nonumber\\
  &\;=\;
  \Eu\Emuc (G_{\ac\ac} G_{1j}) + O(\Phi) 
  \;=\;
  \Eu\Ei (G_{ii} G_{1j}) + O(\Phi)
  \,,
\end{align}
with $t$-HP.
From \eqref{e:1zEG11}--\eqref{e:GiiG11}, we conclude that
\begin{equation}
\Eu\pb{G_{\alpha j} - \Ei G_{i j}}
\;=\; (d-1)^{-1/2}\pb{
  - \Eu (\E^{[i]} G_{ii} G_{1j}) + O(\Phi) 
}\,,
\end{equation}
with $t$-HP.
Since $\E^{[i]} G_{ii} = s$, we obtain \eqref{e:Gdiff1}.
The proof of \eqref{e:Gdiff2} is analogous, using
\eqref{e:GGGchi} instead of \eqref{e:GGchi}, \eqref{e:GGGuc} instead of \eqref{e:GGuc},
and \eqref{e:GGacxGy1} instead of \eqref{e:GacxGy1}.
\end{proof}

The main idea of the proof of Lemma \ref{lem:Gdiff} is \eqref{e:1zEG11}:
the left-hand side is a difference of Green's functions with
\emph{different indices}, while the right-hand side is
(up to a small error) a difference of Green's functions with the
\emph{same indices} but the first Green's function is computed in
terms of a \emph{switched} graph.

We now have all of the ingredients to derive the self-consistent equation for the diagonal entries of $G$.

\begin{lemma} \label{lem:Gjj}
  Given $z \in \C_+$ with $N\eta \geq 1$,
  suppose that \eqref{e:Gjj-conc1}--\eqref{e:Gjj-conc2} hold with $t$-HP.
  Then, we have with $t$-HP, for all $j \in \qq{1,N}$,
  \begin{equation} \label{e:Gjj}
    1 + (s+z) G_{jj} \;=\; O((1+|z|)\xi\Phi)\,.
  \end{equation}
In particular, with $t$-HP,
  \begin{equation} \label{e:Gjjs}
    1+sz+s^2 \;=\; O((1+|z|)\xi\Phi) \,.
  \end{equation}
\end{lemma}

\begin{proof}
The event that \eqref{e:Gjj} holds is measurable with respect to $A$.
By invariance of the law of $A$ under permutation of vertices, and a union bound, it therefore suffices to establish \eqref{e:Gjj} for $j=1$ only.
Then \eqref{e:Gjjs} follows by averaging \eqref{e:Gjj} over $j$.

To show \eqref{e:Gjj} with $j=1$, we make use of the larger probability space $\tilde \Omega$ from Definitions~\ref{def:parametrization}--\ref{def:parametrization2}, where the vertex~$1$ is distinguished.
By \eqref{e:Gjj-conc1}--\eqref{e:Gjj-conc2}, it is sufficient to show that, with $t$-HP,
\begin{equation}
  \label{e:EG11}
  1 + z \Eu G_{11}
  \;=\;
  - \Eu \pb{s G_{11}}
  + O(\Phi)
  \,.
\end{equation}
To show \eqref{e:EG11}, we use that by $(H-z)G = I$ and \eqref{e:H}, with \eqref{e:Ei},
\begin{align}
1 + z G_{11}
\;=\; \sum_i H_{1i} G_{i1}
&\;=\;  (d - 1)^{-1/2}  \sum_i \sum_{\mu = 1}^d \pbb{\delta_{i \alpha_\mu} - \frac{1}{N}} G_{i 1}
\nonumber
\\
&\;=\;  (d - 1)^{-1/2}  \sum_{\mu = 1}^d \Ei \pb{G_{\alpha_\mu 1} - G_{i 1}}
\,.
\label{e:1zG11}
\end{align}
Taking the conditional expectation $\Eu$ on both sides of \eqref{e:1zG11}
and using Lemma~\ref{lem:Gdiff}, we get
\begin{align}
1+z\Eu G_{11}
&\;=\;
- \frac{d}{d - 1} \Eu (s G_{11})
+ O(\Phi) 
\end{align}
with $t$-HP.
This implies \eqref{e:EG11} and therefore completes the proof.
\end{proof}

Under the assumptions of Proposition~\ref{prop:Lambda},
the statement of Lemma~\ref{lem:Gjj} may be strengthened as follows.

\begin{lemma} \label{lem:Gjjstrengthened}
Let $z_0$ be as in \eqref{e:def_z0_z} and suppose that \eqref{e:Gamma_star_assump} holds.
Then with $t$-HP the estimates \eqref{e:Gjj}--\eqref{e:Gjjs}
hold simultaneously for all $z$ as in \eqref{e:def_z0_z}.
\end{lemma}

\begin{proof}
Set
\begin{equation}
  \eta_{l} \;\deq\; \eta_0+l/N^4\,, \quad l \in \qq{0, N^5} \,,
\end{equation}
and $z_l \deq E + \ii \eta_l$.
Since \eqref{e:Gjj-conc1}--\eqref{e:Gjj-conc2} hold uniformly with $t$-HP for any $\eta \geq \eta_0$,
by Lemma~\ref{lem:Gjj} and a union bound,
\eqref{e:Gjj}--\eqref{e:Gjjs} hold simultaneously at all $z_l$ with $l \in \qq{0,N^5}$,
with $t$-HP.
Since $(\eta_{l})_l$ is a $1/N^4$-net of $[\eta_0,\eta_0+N]$
and $G_{ij}$ is Lipschitz continuous with constant $1/\eta^2 \leq N^2$,
the claim follows.
\end{proof}

\subsection{Stability of the self-consistent equation}

In Lemma~\ref{lem:Gjjstrengthened} we showed that, with $t$-HP,
\begin{equation} \label{e:Gjjs2}
  s^2+sz+1 \;=\; O((1+|z|)\xi\Phi) \,.
\end{equation}
It may be easily checked that the Stieltjes transform of the semicircle law \eqref{e:mdef}
is the unique solution $m: \C_+ \to \C_+$ of the equation
\begin{equation} \label{e:m_eq}
  m^2 + mz + 1 \;=\; 0 \,.
\end{equation}
To show that $m$ and $s$ are close, we
use the stability of the equation \eqref{e:m_eq},
in the form provided by the following deterministic lemma.
The stability of the solutions of the equation \eqref{e:m_eq} is a standard tool in the proofs of local semicircle laws for Wigner matrices; 
see e.g.\ \cite{MR2481753}.
Our version given below has weaker assumptions than previously used stability
estimates; in particular, we do not (and cannot) assume an upper bound on the spectral parameter $E$.

\begin{lemma} \label{lem:smdiff}
Let $s : \C_+ \to \C_+$ be continuous, and set
\begin{equation} \label{e:sR}
  R \;\deq\; s^2+sz+1 \,.
\end{equation}
For $E\in \R$, $\eta_0>0$ and  $\eta_1 \geq 3 \vee \eta_0$, suppose that there is a
nonincreasing continuous function
$r \col [\eta_0,\eta_1] \to [0,1]$
such that
$|R(E+\ii\eta)| \leq (1+|E + \ii \eta|)r(\eta)$ for all $\eta \in [\eta_0,\eta_1]$.
Then for all $z = E+\ii\eta$ with $\eta \in [\eta_0,\eta_1]$ we have
\begin{equation} \label{e:smbd}
  |s-m| \;=\; O(F(r)) 
  \,,
\end{equation}
where $F$ was defined in \eqref{e:Fdef}.
\end{lemma}

\begin{proof}
  Denote by $m$ and $\hat m$ the two solutions of \eqref{e:m_eq} with positive
  and negative imaginary parts, respectively:
  \begin{equation} \label{e:mmhat}
    m \;=\; \frac{-z  + \sqrt{z^2 - 4}}{2}\,, \qquad
    \hat m \;=\; \frac{-z - \sqrt{z^2 - 4}}{2} \,,
  \end{equation}
  where the square root is chosen so that $\im m > 0$.
  Hence, $\im \sqrt{z^2-4} > \eta$ and consequently
  $\im \hat m < - \eta$; we shall use this bound below.
  Note that $m$ and $\hat m$ are continuous.
  Set $v \deq s-m$ and $\hat v \deq s -\hat m$.
  Since
  \begin{equation}
    s \;=\; \frac{-z  \pm \sqrt{z^2 - 4 + 4R}}{2} 
  \end{equation}
  and since for any complex square root $\sqrt{\, \cdot\,}$ and $w,\zeta \in \C$ we have
  \begin{equation}
   |\sqrt{w+\zeta}-\sqrt{w}| \wedge |\sqrt{w+\zeta}+\sqrt{w}| \;\leq\;
   \frac{|\zeta|}{\sqrt{|w|}} \wedge \sqrt{|\zeta|}
   \,,
  \end{equation}
  we deduce that
  \begin{equation} \label{e:vhatvM}
    |v| \wedge |\hat v| \;\leq\;
    \frac{2 |R|}{\sqrt{|z^2-4|}} \wedge \sqrt{|R|}
    \;\leq\; 3F(r)
    \,.
  \end{equation}
  In the last inequality, we used that $\abs{R} \leq (1+|z|)r$ and that, for any $r \in [0,1]$,
  \begin{equation} \label{e:3F}
    \frac{2 (1+|z|)r}{\sqrt{|z^2-4|}} \wedge \sqrt{(1+|z|)r}
    \;\leq\; 3F(r)
    \,.
  \end{equation}
  
  The proof is divided into three cases.
  First, consider the case $(1+|z|)r(\eta) \geq |z^2-4|/16$. 
  Then, using \eqref{e:3F} and the fact that $\abs{\hat v - v} = \sqrt{\abs{z^2 - 4}}$, 
  we get
  \begin{equation}
    |\hat v|
    \;\leq\; |v| + \sqrt{|z^2-4|}
    \;=\; |v| + O\pBB{\frac{2(1+|z|)r}{\sqrt{|z^2-4|}} \wedge \sqrt{(1+|z|)r}}
    \;=\; |v|+ O(F(r))
    \,,
  \end{equation}
  and hence $|v| \leq |v| \wedge |\hat v| + O(F(r)) = O(F(r))$ by \eqref{e:vhatvM}.

  Next, consider the case $\eta \geq 3$. Then, on the one hand, by \eqref{e:vhatvM}
  and the assumption $r \in [0,1]$, we have
  $|v| \wedge |\hat v|  \leq 3F(r) \leq 3$.
  On the other hand, since $\im s > 0$ and $\im \hat m< -\eta \leq  -3$,
  \begin{equation}
    |\hat v| \;\geq\; |\im s - \im \hat m| \;>\; |\im \hat m| \;>\;  3 \,,
  \end{equation}
  and together we conclude that $|\hat v| > |v|$, so that the claim follows from \eqref{e:vhatvM}.
  
  Finally, consider the case $(1+|z|) r(\eta) < |z^2-4|/16$ 
  and $\eta < 3$.
  Without loss of generality, we set $\eta_0 = \eta$.
  Since $r$ is nonincreasing and $|z^2-4|$ increasing in $\eta$,
  and since $(1+|z|) \leq 4(1+|z_0|)$ for $\eta \in [\eta_0,3]$,
  we then have $(1+|z|)r(\eta) < |z^2-4|/4$ for all $\eta \in [\eta_0, 3]$.
  Therefore
  \begin{equation} \label{e:vhatv1}
    |v-\hat v| \;=\; \sqrt{|z^2-4|}
    \;\geq\; 2 \pBB{ \frac{2  (1+|z|) r}{\sqrt{|z^2-4|}} \wedge \sqrt{ (1+|z|)r} }
    \quad \text{for all $\eta \in [\eta_0, 3]$} \,.
  \end{equation}
  On the other hand, by \eqref{e:vhatvM}, and since $|R| \leq (1+|z|)r$, 
  \begin{equation} \label{e:vhatv2}
    |v| \wedge |\hat v|
    \;\leq\; \frac{2(1+|z|)r}{\sqrt{|z^2-4|}} \wedge \sqrt{(1+|z|)r}
    \quad \text{for all $\eta \in [\eta_0,3]$} \,.
  \end{equation}
  By continuity and \eqref{e:vhatv1}--\eqref{e:vhatv2},
  it suffices to show $|v(z)| < |\hat v(z)|$ for some $\eta \in [\eta_0, 3]$,
  since then $|v(z)| < |\hat v(z)|$ for all $\eta \in [\eta_0, 3]$.
  Since we have already shown $\abs{v(z)} < \abs{\hat v(z)}$ for $\eta=3$,
  the proof is complete. 
\end{proof}

\subsection{Proof of Proposition~\ref{prop:Lambda}}

We now have all the ingredients we need to complete the proof of \eqref{e:Lambdad}--\eqref{e:Lambdao}
under the assumption \eqref{e:Gamma_star_assump}, and hence the proof of Proposition \ref{prop:Lambda}.

\begin{proof}[Proof of \eqref{e:Lambdad}]
Let $z_0 = E + \ii \eta_0 \in \f D$ be given, where $\f D$ was defined in \eqref{e:D}.
Set $\eta_1=N$.

Lemma~\ref{lem:Gjjstrengthened} shows that, with $t$-HP,
for all $\eta \in [\eta_0, \eta_1]$,
the function $s$ satisfies \eqref{e:sR} with
\begin{equation}
  |R(z)| \;\leq\; (1+|z|) r(\eta)\,,
\end{equation}
where $r(\eta) \deq C \xi \Phi (E+\ii \eta)$ and $C$ is some large enough absolute constant.
Hence, $r$ is decreasing in $\eta$ and,
since $\xi\Phi \ll 1$ by assumption, we have $r \in [0,1]$.
From Lemma~\ref{lem:smdiff}, it therefore follows that $|m-s| = O(F(C\xi\Phi)) = O(F(\xi\Phi))$
for all $\eta \in [\eta_0,\eta_1]$, with $t$-HP.

Having determined $s$, we now estimate $G_{jj} - m$.
By \eqref{e:Gjj} and since $s=m+O(F(\xi\Phi))$, we find
\begin{equation} \label{e:sGii1}
  1+(z+m)G_{jj} \;=\; O(F(\xi\Phi))G_{jj} + O((1+|z|)\xi\Phi)
\end{equation}
with $t$-HP.
From \eqref{e:mdef} it is easy to deduce that  $z+m = -1/m$ and $(1+|z|)|m| = O(1)$. Hence,
by \eqref{e:sGii1} and since $G_{jj} = O(1)$ with $t$-HP,
\begin{equation} \label{e:sGii2}
   m-G_{jj} \;=\; O(F(\xi\Phi))G_{jj} + O(\xi\Phi) \;=\; O(F(\xi\Phi))
\end{equation}
with $t$-HP, as claimed.
\end{proof}

The off-diagonal entries of the Green's function can be estimated using a similar argument.

\begin{proof}[Proof of \eqref{e:Lambdao}]
From $(HG)_{ij} = ((H-z)G)_{ij} + z G_{ij} = \delta_{ij} + zG_{ij}$, it follows that
\begin{equation}
  G_{12} (HG)_{11} - G_{11} (HG)_{12} 
  \;=\; G_{12} \,,
\end{equation}
and therefore
\begin{equation} \label{e:G12newid}
  G_{12}
  \;=\; G_{12} \sum_{i} H_{1i} G_{i1} - G_{11} \sum_{i} H_{1i}G_{i2}
  \,.
\end{equation}
As in \eqref{e:1zG11}, using \eqref{e:G12newid} and \eqref{e:H}, we find
\begin{equation} 
  G_{12}
  \;=\; - (d-1)^{-1/2} \sum_{\mu=1}^d \Ei \pB{ G_{11}\pB{G_{\alpha_\mu 2} -G_{i2}} - G_{12}\pB{G_{\alpha_\mu 1}-G_{i1}}}
  \,.
\end{equation}
Using \eqref{e:Gdiff2} we therefore get, with $t$-HP,
\begin{equation}
  \Eu G_{12}
  \;=\; - (d-1)^{-1} \sum_{\mu=1}^d \Eu \pB{ G_{11} s G_{12} - G_{12} s G_{11} } + O(\Phi)
  \;=\; O(\Phi) \,.
\end{equation}
By Proposition~\ref{prop:concentration},
therefore $G_{12} = \Eu G_{12} + O(\xi\Phi) = O(\xi\Phi)$ with $t$-HP.
Again, by symmetry and a union bound, the claim then holds uniformly with $12$ replaced by $ij$.
\end{proof}

Summarizing, we have proved that, assuming
\eqref{e:Gamma_star_assump}, the estimates \eqref{e:Lambdad} and
\eqref{e:Lambdao} hold with $t$-HP.  Hence the proof of
Proposition~\ref{prop:Lambda} (and consequently of
Theorem~\ref{thm:semicircle}) is complete.

This proof of Theorem \ref{thm:semicircle} relies on
Proposition~\ref{prop:switch}, which we proved for the matching model
in Section \ref{sec:MM}. In order to establish Theorem
\ref{thm:semicircle} for the uniform and permutation models, we still
have to prove Proposition \ref{prop:switch} for these models. This is
done in Sections \ref{sec:UM} and \ref{sec:PM}, which constitute the
rest of the paper.

\section{Uniform model}
\label{sec:UM}

In this section we prove Proposition~\ref{prop:switch} for the uniform model.
We identify a simple graph on the vertices $\qq{1,N}$ with its set of edges $E$,
where an edge $e \in E$ is a subset of $\qq{1,N}$ with two elements.
The adjacency matrix of a set of edges $E$ is by definition
\begin{equation}
  M(E) \;\deq\; \sum_{\{i,j\} \in E} \Delta_{ij} \,,
\end{equation}
where $\Delta_{ij}$ was defined in \eqref{e:Eijdef}. Note that $M(\cdot)$ is one-to-one,
i.e.\ the matrix $M(E)$ uniquely determines the set of edges $E$. 
For a subset $S \subset E$ of edges we denote by $[S] \deq \bigcup_{e \in S} e$ the set of vertices
incident to any edge in $S$. Moreover, for a subset $B \subset \qq{1,N}$ of vertices, we define $E\vert_B \deq \h{e \in E \col e \subset B}$
to be the subgraph of $E$ induced on $B$.

\begin{figure}[t]
\begin{center}
\input{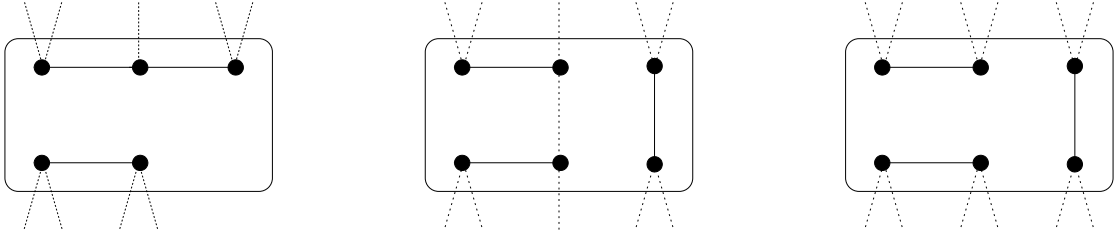}
\end{center}
\caption{
  Each of the three diagrams illustrates the subgraph incident to a set $S \subset E$ of three distinct edges ($\abs{S} = 3$).
  We draw the edges of $S$ with solid lines and the edges of $E \setminus S$ with dotted lines.
  In the left and centre diagrams we have $I(E,S) = 0$: in the left diagram $\abs{[S]} < 6$,
  while in the centre diagram $\abs{[S]} = 6$ but $E\vert_{[S]} \neq S$.   In the right diagram we have $I(E,S)=1$.
  \label{fig:S1}}
\end{figure}

\subsection{Switchings}
\label{sec:UM-switch}

For a subset $S \subset E$ with $\abs{S} = 3$ we define the indicator function
\begin{equation*}
I(E,S) \;\deq\; \ind{E\vert_{[S]} \text{ is $1$-regular}} \;=\; \ind{\abs{[S]} = 6, E\vert_{[S]} = S}\,.
\end{equation*}
The interpretation of $I(E,S) = 1$ is that $S$ is a switchable subset of $E$,
i.e.\ any double switching of the three edges $S$ results again in
a simple $d$-regular graph; see Figure \ref{fig:S1}. A switching of the edges $S$ may be identified
with a perfect matching of the vertices $[S]$.
There are eight perfect matchings $S'$ of $[S]$ such that $S \cap S' = \emptyset$.
We enumerate these matchings in an arbitrary way as $S'_s$ with $s \in \qq{1,8}$, and set
\begin{equation} \label{e:UM-TSsdef}
  T_{S,s}(E) \;\deq\; (E \setminus S) \cup S_s'
  \,,
\end{equation}
and say that $T_{S,s}(E)$ is a switching of $E$. (Compare this with Figure~\ref{fig:switch} (right)
in which one such perfect matching is illustrated.)
Note that there are $r\ur,a \ul a,b \ul b$ depending on $(S,s)$ such that
\begin{equation} \label{e:MTtau}
  M(T_{S,s}(E)) \;=\; \tau_{r\ur,a\ua,b\ub}(M(E)) \,,
\end{equation}
with the right-hand side defined by \eqref{e:Adoubleswitch}.
This correspondence will be made explicit later.
The definition \eqref{e:UM-TSsdef} implies, for $I(E,S) = 1$, that
$T_{S,s}(E)$ is a simple $d$-regular graph, and that
\begin{equation} \label{e:UM-Tincl}
  E \setminus S \;=\; T_{S,s}(E) \cap (E\setminus S) \,.
\end{equation}

Next, take two disjoint subsets $S_1, S_2 \subset E$ satisfying $I(E, S_1) = I(E,S_2) = 1$
and $[S_1] \cap [S_2] = \h{1}$. Thus, we require the sets $S_1$ and $S_2$ to be incident
to exactly one common vertex, which we set to be $1$; see Figure \ref{fig:S2}.
Then $S_2 \subset E\setminus S_1$ and $S_1 \subset E\setminus S_2$,
and \eqref{e:UM-Tincl} implies that the two compositions
$T_{S_1,s_1}(T_{S_2,s_2}(E))$ and $T_{S_2,s_2}(T_{S_1,s_1}(E))$
are well-defined and coincide,
\begin{equation} \label{e:UM-Tcomm}
  T_{S_1,s_1}(T_{S_2,s_2}(E))
  \;=\;
  T_{S_2,s_2}(T_{S_1,s_1}(E))
  \,.
\end{equation}

Let $S$ satisfy $I(E,S) = 1$ and $1 \in [S]$.
The map $T_{S,s}(E)$ switches the unique edge $\{1,i\} \in S$ incident to $1$ to a new edge $\{1,j\} \notin S$ with $j \in [S]$.
Our next goal is to extend this switching to a \emph{simultaneous} switching 
of all neighbours of $1$. As already seen in \eqref{e:UM-Tcomm}, simultaneous multiple switchings are not always possible,
and our construction will in fact only switch those neighbours of $1$ that can be switched without disrupting any other neighbours of $1$.
The remaining neighbours will be left unchanged.
Ultimately, this construction will be effective because the number of neighbours of $1$ that cannot be switched will be small with high probability. 

Let $(e_1(E), \dots, e_d(E))$ be an 
enumeration of the edges in $E$ incident to $1$, and denote by
\begin{equation} \label{e:def_calS}
\cal S_\mu(E) \;\deq\; \hB{S \subset E \col e_\mu(E) \in S \,, \abs{S} = 3 \,, 1 \notin e \text{ for } e \in S \setminus \{e_\mu(E)\}}
\end{equation}
the set of unordered triples of distinct edges in $E$ containing $e_\mu(E)$ and no other edge incident to $1$.
Conditioned on $E$, we define a random variable $(\f S, \f s)$, where $\f S = (S_1, \dots, S_d)$ and $\f s = (s_1, \dots, s_d)$,
uniformly distributed over $\cal S_1(E) \times \cdots \times \cal S_d(E) \times \qq{1,8}^d$.
In particular, conditioned on $E$, the random variables $(S_1, s_1), \dots, (S_d, s_d)$ are independent.

\begin{figure}[t]
\begin{center}
\input{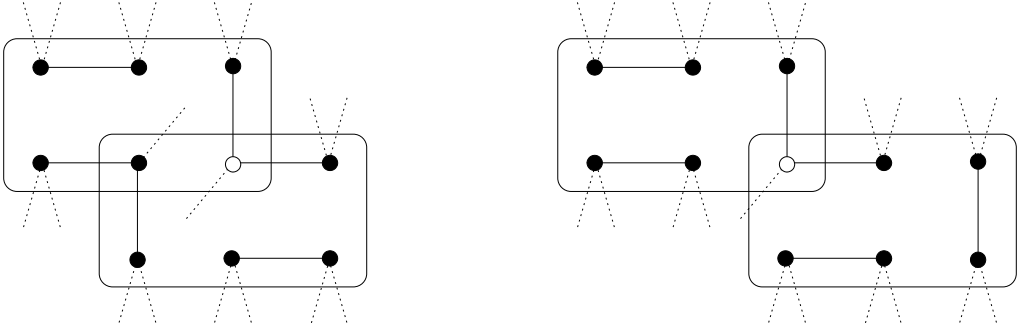}
\end{center}
\caption{
  Each of the two diagrams illustrates
  the subgraph incident to $S_1$ and $S_2$ for two sets of three edges $S_1, S_2 \subset E$ satisfying $I(E,S_1) = I(E, S_2) = 1$.
  The edges of $S_1 \cup S_2$ are drawn with solid lines and the edges of $E \setminus (S_1 \cup S_2)$ with dashed lines.
  The vertex $1$ is drawn using a white circle.
  In terms of the indicator functions $J_\mu$ defined in \eqref{e:UM-Jdef},
  in the left diagram we have $J_1(E,\f S) = J_2(E, \f S) =0$ since $[S_1] \cap [S_2] \neq \{1\}$,
  and in the right diagram $J_1(E,\f S) = J_2(E,\f S) = 1$.
  \label{fig:S2}}
\end{figure}

For $\mu\in \qq{1,d}$ we define the indicator functions
\begin{align}
\label{e:UM-Idef}
I_\mu \;\equiv\; I_\mu(E, \f S) &\;\deq\; I(E, S_\mu)\,, 
\\
\label{e:UM-Jdef}
J_\mu \;\equiv\; J_\mu(\f S) &\;\deq\; \indb{[S_\mu] \cap [S_\nu] = \{1\} \text{ for all $\nu \neq \mu$}}\,,
\end{align}
and the set
\begin{equation} \label{e:UM-Wdef}
W \;\equiv\; W(E,\f S) \;\deq\; \h{ \mu \in \qq{1,d} : I_{\mu}(E, \f S) J_{\mu}(\f S) = 1}\,.
\end{equation}
Their interpretation is as follows.
On the event $\{I_\mu = 1\}$, the edges $S_\mu$ are switchable in the sense that
any switching of them results in a simple $d$-regular graph.
The interpretation of $\{J_\mu = 1\}$ is that the edges of $S_\mu$
do not interfere with the edges of any other $S_\nu$, and hence any switching of them
will not influence or be influenced by the switching of another triple
of edges: on the event $\{J_\mu = 1\}$, \eqref{e:UM-Tcomm} implies that $T_{S_\mu,s_\mu}$ commutes
with $T_{S_\nu,s_\nu}$ for all $\nu \neq \mu$.
The set $W$ lists the neighbours of $1$ that can be switched simultaneously; see Figure \ref{fig:S2}.

Let $\mu_1, \dots,\mu_k$, where $k \leq d$, be an arbitrary enumeration of $W$, and set
\begin{equation}   \label{e:UM-Tdef}
T_{\f S, \f s}(E) \;\deq\; \pb{T_{S_{\mu_1},s_{\mu_1}} \circ \cdots \circ T_{S_{\mu_k},s_{\mu_k}}} (E)
\,.
\end{equation}
By \eqref{e:UM-Tcomm}, the right-hand side 
is well-defined and independent of the order of the applications of the $T_{S_\mu,s_\mu}$.
Equivalently, in terms of adjacency matrices, $T_{\f S, \f s}(E)$ is given by
\begin{equation} \label{e:UM-TM}
  M(T_{\f S, \f s}(E)) - M(E) \;=\; \sum_{\mu \in W} \pb{M(T_{S_\mu,s_\mu}(E)) - M(E)}
  \,,
\end{equation}
where we used that, by construction of $W$, any switchings $\mu \neq \nu$ with $\mu, \nu \in W$ 
do not interfere with each other, so that
$M(T_{S_\mu,s_\mu} T_{S_{\nu,s_\nu}}(E)) - M(T_{S_\nu,s_\nu}(E))= M(T_{S_\mu,s_\mu}(E)) - M(E)$.

The following result ensures that the simultaneous switching leaves
the uniform distribution on simple $d$-regular graphs invariant. For
this property to hold, it is crucial that, as in \eqref{e:def_calS},
we admit configurations $\f S$ that may have edges that cannot be
switched. The more naive approach of only
averaging over configurations $\f S$ in which
all edges can be switched simultaneously does not leave the
uniform measure invariant. The price of admitting configurations
$\f S$ that do not switch some neighbours of $1$ is mitigated by the fact
that such configurations are exceptional and occur with small probability,
i.e.\ conditioned on $E$, $I_\mu(E, \f S) = J_\mu(\f S) = 1$ with high
probability.

\begin{lemma} \label{lem:UM-rev}
If $E$ is a uniform random simple $d$-regular graph, and $(\f S, \f s)$ is uniform
over $\cal S_1(E) \times \cdots \times \cal S_d(E) \times \qq{1,8}^d$, then $T_{\f S, \f s} (E)$
is a uniform random simple $d$-regular graph.
\end{lemma}

\begin{proof}
It suffices to show reversibility of the transformation $E \mapsto T_{\f S, \f s}(E)$ with respect to the uniform measure,
i.e.\ that for any fixed simple $d$-regular graphs $E_1, E_2$ we have
\begin{equation} \label{transition_p}
\P(T_{\f S, \f s}(E) = E_2 | E = E_1) \;=\; \P(T_{\f S, \f s}(E) = E_1 | E = E_2)\,.
\end{equation}
Note that $(\f S, \f s)$ is uniformly distributed over
$\cal S_1(E_1) \times \cdots \times \cal S_d(E_1) \times \qq{1,8}^d$ on the
left-hand side and over $\cal S_1(E_2) \times \cdots \times \cal S_d(E_2) \times \qq{1,8}^d$ on the right-hand side.

First, given two (simple $d$-regular) graphs $E_1,E_2$, we say that $E_2$ is a switching of $E_1$ if there exist $(\f S,\f s)$ such that
$E_2=T_{\f S,\f s}(E_1)$,
and note that $E_1$ is a switching of $E_2$ if and only if $E_2$ is a switching of $E_1$.
If these conditions do not hold, then both sides of \eqref{transition_p} are zero.
We conclude that it suffices to show \eqref{transition_p} for the case that $E_2$ is a switching of $E_1$
(or, equivalently, $E_1$ is a switching of $E_2$).
In words, it suffices to show that the probability that $E_1$ is switched to $E_2$
is the same as the probability that $E_2$ is switched to $E_1$. To this end, we first construct a
bijection $\phi \col E_1 \to E_2$ between the edges of the two graphs,
and then show that the conditioned probability measure is invariant under this bijection.
The bijection is deterministic.

\begin{figure}[t]
\begin{center}
\input{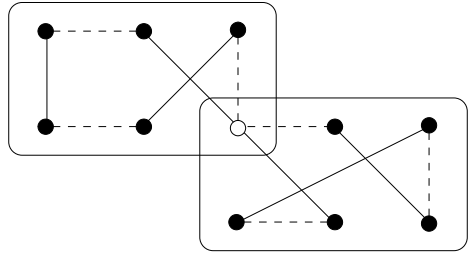}
\end{center}
\caption{
  The lines (solid and dashed) illustrate the edges of $E_\triangle$.
  The solid lines are the edges of $E_1$ and the dashed line the edges of $E_2$.
 The white circle represents the vertex $1$.
  Note that $E_\Delta|_{\h{2,\dots, N}}$ consists of $|W|=2$ disconnected subgraphs, each of
  which has $5$ vertices and $4$ edges. These subgraphs are given by
  the encircled regions with the vertex $1$ and its incident edges removed. The sets $B_\mu$
  are given by the encircled vertices including $1$.
  \label{fig:revproof1}}
\end{figure}

Define $E_\bigtriangleup \deq E_1 \bigtriangleup E_2$ and $E_\cap \deq E_1 \cap E_2$,
where $\bigtriangleup$ denotes the symmetric difference.
Define $W \deq \{\mu \col e_\mu (E_1) \in E_\bigtriangleup\}$.
The interpretation of $W$ is the index set of neighbours of $1$ in $E_1$ that were switched in going from $E_1$ to $E_2$.
This is the same set as the set from \eqref{e:UM-Wdef}. Note that $W$ is now deterministic: for $E_1$ and $E_2$ that are switchings of each other, the set $W$ is uniquely determined.
Since $E_2$ is a switching of $E_1$, and by the constraints in the indicator functions $I_\mu$ and $J_\mu$ in the definition of $T_{\f S, \f s}$,
we find that $E_\bigtriangleup \vert_{\{2, \dots, N\}}$ consists of $\abs{W}$ disconnected subgraphs,
each of which has $5$ vertices and $4$ edges.
Each such subgraph is adjacent to a unique edge $e_\mu(E_1)$ where $\mu \in W$.
For $\mu \in W$, we denote by $B_\mu$ the set of vertices consisting of $1$
and the vertices of the subgraph that is adjacent to $e_\mu(E_1)$.
By construction, $E_2 \vert_{B_\mu}$ is a switching of $E_1 \vert_{B_\mu}$ (and vice versa);
both are $1$-regular graphs on six vertices.
(The interpretation of $B_\mu$ is that the switching $T_{\f S, \f s}$ that maps $E_1$ to $E_2$ satisfies $S_\mu = E_1 \vert_{B_\mu}$.)
This construction is illustrated in Figure~\ref{fig:revproof1}.

Now we define the bijection $\phi \col E_1 \to E_2$.
For each $\mu \in W$, we choose $\phi$ to be a bijection from $E_1 \vert_{B_\mu}$ to $E_2 \vert_{B_\mu}$
such that $\phi(e_\mu(E_1))$ is incident to $1$.
(For each such $\mu$ there are two possible choices for this bijection; this choice is immaterial.)
Without loss of generality, we can choose the enumeration $e_\mu(E_2)$ of $E_2$ such that
$e_\mu(E_2) = \phi(e_\mu(E_1))$.
This defines a bijection $\phi \col E_1 \cap E_\bigtriangleup \to E_2 \cap E_\bigtriangleup$.
We extend it to a bijection $\phi \col E_1 \to E_2$ by setting $\phi(e) \deq e$ for $e \in E_\cap$.

With these preparations, we now show \eqref{transition_p}.
Given $E_1$ and $E_2$ that are switchings of each other,
we have constructed a set $W \subset \qq{1,d}$ and subsets $B_\mu$ for $\mu \in W$,
such that $E_1 \vert_{B_\mu}$ and $E_2 \vert_{B_\mu}$ are $1$-regular graphs obtained
from a unique switching from each other. Since  the $s_\mu$ are independent and since
for each $\mu \in W$ the random variable $s_\mu$ is uniform on $\qq{1,8}$,
we find that the left-hand side of \eqref{transition_p} is equal to
\begin{equation} \label{e:UM-rev-LHS}
8^{-\abs{W}} \, \P\pB{S_\mu = E_1 \vert_{B_\mu} \text{ for } \mu \in W\,,\, \hb{\mu \in \qq{1,d} \col I_\mu(E_1, \f S) J_\mu(\f S) = 1} = W}\,,
\end{equation}
where $\f S$ is uniform over $\cal S_1(E_1) \times \cdots \times \cal S_d(E_1)$.

By an identical argument, we find that the right-hand side of \eqref{transition_p} is equal to
\begin{equation} \label{e:UM-rev-LHS2}
8^{-\abs{W}} \, \P\pB{S_\mu = E_2 \vert_{B_\mu} \text{ for } \mu \in W\,,\, \hb{\mu \in \qq{1,d} \col I_\mu(E_2, \f S) J_\mu(\f S) = 1} = W}\,,
\end{equation}
where $\f S$ is uniform over $\cal S_1(E_2) \times \cdots \times \cal S_d(E_2)$. Note that, by construction, $W$ and $B_\mu$ for $\mu \in W$ are the same in both \eqref{e:UM-rev-LHS} and \eqref{e:UM-rev-LHS2}.

What remains is to show that \eqref{e:UM-rev-LHS} and \eqref{e:UM-rev-LHS2} are equal.
In order to prove this, we abbreviate $\phi(\f S) \deq (\phi(S_1), \dots, \phi(S_d))$.
Then the definitions of $I_\mu$, $J_\mu$, and $\phi$ imply that 
for all $\f S \in \cal S_1(E_1) \times \cdots \times \cal S_d(E_1)$ we have
\begin{equation}
  I_\mu(E_1, \f S) \;=\; I_\mu(\phi(E_1), \phi(\f S))\,, \quad
  J_\mu(\f S) = J_\mu(\phi(\f S)) \,,
\end{equation}
and hence \eqref{e:UM-rev-LHS} is equal to
\begin{equation} \label{e:UM-rev-RHS}
8^{-\abs{W}} \, \P\pB{\phi(S_\mu) = \phi(E_1) \vert_{B_\mu} \text{ for } \mu \in W\,,\, \hb{\mu \in \qq{1,d} \col I_\mu(\phi(E_1), \phi(\f S)) J_\mu(\phi(\f S)) = 1} = W}\,.
\end{equation}
Since $\phi(E_1) = E_2$ and
since $\phi$ is a bijection from $\cal S_1(E_1) \times \cdots \cal S_d(E_1)$ 
to $\cal S_1(\phi(E_1)) \times \cdots \times \cal S_d(\phi(E_1))$,
and therefore $\phi(\f S)$ is uniform over $\cal S_1(E_2) \times \cdots \times \cal S_d(E_2)$,
we conclude that \eqref{e:UM-rev-RHS} is equal to \eqref{e:UM-rev-LHS2}. This concludes the proof.
\end{proof}

\subsection{Estimate of exceptional configurations}
\label{sec:UM-exceptprob}

In preparation of the proof of Proposition~\ref{prop:switch} for the uniform model,
we now define 
the probability space $\Omega$ from Definition \ref{def:parametrization}.
The space $\Theta$ is the set of simple $d$-regular graphs (identified with their sets of edges $E$), and
\begin{equation*}
U_\mu \;\deq\; \hb{\text{sets of three distinct edges of the complete graph on $N$ vertices}} \times \qq{1,8}\,.
\end{equation*}
Hence, the probability space $\Omega = \Theta \times U_1 \times \cdots \times U_d$ from Definition \ref{def:parametrization}
consists of elements
\begin{equation*}
(\theta, u_1, \dots, u_d) \;=\;
\pb{E, (S_1,s_1), \dots, (S_d, s_d)} \,,
\end{equation*}
where we denote elements of $\Theta$ by $\theta = E$ and elements of $U_\mu$ by $u_\mu = (S_\mu, s_\mu)$.
Next, we define the (non-uniform) probability measure on the set $\Omega$.
To that end, we first endow the set of simple $d$-regular graphs $\Theta$ with the uniform probability measure.
For each $E \in \Theta$,
we fix an arbitrary enumeration $e_1(E), \dots, e_d(E) \in E$ of the edges of $E$ incident to $1$.
Then, conditioned on $E \in \Theta$, we take $u_1, \dots, u_d$ to be independent,
with $S_\mu$ uniformly distributed over $\cal S_\mu(E)$ defined in \eqref{e:def_calS},
and $s_\mu$ uniformly distributed over $\qq{1,8}$.
(In other words, the probability measure is uniform on $(\theta, u_1, \dots, u_d) \in \Omega$ such that
each $S_\mu$ contains $e_\mu(E)$ and no other edge incident to~$1$.)
Having defined the probability space $\Omega$, we augment it to $\tilde \Omega$ according to Definition \ref{def:parametrization2}.

From now on, we always condition on $E$ and write $e_\mu \equiv e_\mu(E)$. 
We denote by $p_\mu , q_\mu$ the two edges in $S_\mu$
not incident to $1$ (ordered in an arbitrary fashion), so that $S_\mu = \{e_\mu, p_\mu, q_\mu\}$.
The next lemma provides some general properties of the random sets $S_\mu$.

\begin{lemma} \label{lem:UM-condprob}
The following holds for any fixed $\mu \in \qq{1,d}$.
\begin{enumerate}
\item
There are at most five $\nu \neq \mu$ such that
\begin{equation} \label{e:unif-Fmu-maxbd}
  [S_\mu]\cap [S_\nu] \neq \h{1}
  \qquad \text{and} \qquad
  [S_\nu] \cap [S_\kappa] = \h{1} \text{ for all $\kappa \neq \mu,\nu$}
  \,.
\end{equation}
\item
For any symmetric function $F$ we have
\begin{equation} \label{e:UM-Sapproxunif}
  \E_{\sigG_\mu} F(p_\mu,q_\mu) \;=\;
  \frac{1}{(Nd/2)^2} \sum_{p,q\in E} F(p,q)
  + O\pbb{\frac{1}{N}} \|F\|_\infty
  \,.
\end{equation}
Similarly, for any function $F$ we have
\begin{equation} \label{e:UM-Sapproxunif1}
  \E_{\sigG_\mu,q_\mu} F(p_\mu) \;=\;
  \E_{\sigG_\mu,p_\mu} F(q_\mu) \;=\;
  \frac{1}{Nd/2} \sum_{p\in E} F(p)
  + O\pbb{\frac{1}{N}} \|F\|_\infty
  \,.
\end{equation}
\end{enumerate}
\end{lemma}

\begin{proof}
We begin with (i).
Let $B$ be the set of $\nu \neq \mu$ satisfying \eqref{e:unif-Fmu-maxbd}.
By definition, for $\nu \in B$, $[S_\nu] \cap [S_\kappa] = \h{1}$ for all $\kappa \neq \mu,\nu$.
Thus, each $p \in [S_\mu] \setminus \{1\}$ can be contained in at most one $S_\nu$ with $\nu \in B$.
The claim follows since $[S_\mu]\setminus\{1\}$ has at most five elements.

Next, we prove (ii).
Conditioned on $\sigG_\mu$, the two edges $p_\mu,q_\mu$ are by definition of $S_\mu$ chosen
to be distinct and uniformly distributed on the $Nd/2-d$ edges not incident to $1$.
Let $\partial 1 = \{e \in E: 1 \in e\}$
be the set of edges in $E$ incident to $1$.
Then
\begin{align} \label{e:UM-approxunif-pf}
  \Emu F(p_\mu,q_\mu)
  &\;=\;
  \frac{1}{2 \binom{Nd/2-d}{2}} \sum_{p,q \in E \setminus \partial 1: p \neq q} F(p,q)
  \nonumber\\
  &\;=\;
  \frac{1}{(Nd/2)^2} \sum_{p,q \in E} F(p,q)
  + O\pbb{\frac{1}{N}} \|F\|_\infty
  \,.
\end{align}
This shows \eqref{e:UM-Sapproxunif}; the proof of \eqref{e:UM-Sapproxunif1} is analogous.
\end{proof}

Next, we derive some basic estimates on the indicator functions
$I_\mu$ and $J_\mu$ and the random set $W$. Ideally, we would like to
ensure that with high probability $I_\mu J_\mu = 1$. While this event
does hold with high probability conditioned on $\cal F_0$, it does
\emph{not} hold with high probability conditioned on $\cal G_\mu$.
In fact, conditioned on $\cal G_\mu$, it may happen that $I_\mu J_\mu = 0$
almost surely. This happens if there exists a $\nu \neq \mu$
such that $[S_\nu] \cap e_\mu \neq \{1\}$. The latter event is clearly
independent of $S_\mu$. To remedy this issue, we introduce
 the $\sigG_\mu$-measurable
indicator function
\begin{equation*}
h_\mu \;\deq\; \ind{e_\mu \cap H_\mu = \{1\}} \,, \qquad H_\mu \;\deq\; \bigcup_{\nu \neq \mu} [S_{\nu}]
\end{equation*}
which indicates whether such a bad event
takes place. Then, instead of showing that conditioned on $\cal G_\mu$
we have $I_\mu J_\mu = 1$ with high probability, we show that
conditioned on $\cal G_\mu$ we have $I_\mu J_\mu = h_\mu$ with high probability.
This estimate will in fact be enough for our purposes,
by a simple analysis of the two cases $h_\mu = 1$ and $h_\mu = 0$. In
the former case, we are in the generic regime that allows us to
perform the switching of $e_\mu$ with high probability, and in the
latter case the switching of $e_\mu$ is trivial no matter the value of
$S_\mu$.

For the statement of the following lemma, we recall
the random set $W$ defined in \eqref{e:UM-Wdef}, and that $\tilde W^\mu$
is obtained from $W$ by replacing the argument $\umu$ with $\umuc$.
We also recall
that $X \triangle Y$ denotes the symmetric difference of the sets $X$ and $Y$.

\begin{lemma}  \label{lem:UM-condprob2}
The following holds for any fixed $\mu \in \qq{1,d}$.
\begin{enumerate}
\item
Conditioned on $\sigG_\mu$, with probability $1-O\pb{\frac{d}{N}}$, 
we have
\begin{equation} \label{e:unif-G-good}
  \ind{\mu \in W} \;=\; I_\mu J_\mu \;=\; h_\mu
  \,.
\end{equation}
Conditioned on $\sigF_0$, with probability $1-O\pb{\frac{d}{N}}$,
we have $\mu \in W$ (i.e.\ $h_\mu=1$).
\item
Almost surely, we have $|W\triangle \tilde W^\mu \setminus \{\mu\}| \leq 10$.
\item
Conditioned on $\sigG_\mu$, with probability $1-O(\frac{d}{N})$, we have $W\triangle \tilde W^\mu = \emptyset$.
\end{enumerate}

\end{lemma}

\begin{proof}
We first show the claim of (i) concerning conditioning on $\sigG_\mu$.
First, the definition of $J_\mu$ immediately implies that $I_\mu J_\mu = 0$ if $h_\mu = 0$.
Therefore and since $h_\mu$ is $\sigG_\mu$-measurable, it suffices to show
that, conditioned on $\sigG_\mu$ such that $h_\mu=1$,
we have $I_\mu J_\mu = 1$ with probability $1-O(\frac{d}{N})$.
Hence, for the following argument we condition on $\cal G_\mu$ and suppose that $h_\mu = 1$.
We estimate
\begin{equation} \label{e:unif-Gmu-good-pf}
  \P_{\cal G_\mu} (I_\mu J_\mu = 0) \;\leq\; \P_{\cal G_\mu} (I_\mu = 0) + \P_{\cal G_\mu}(J_\mu = 0)\,.
\end{equation}

The first term on the right-hand side of \eqref{e:unif-Gmu-good-pf} is equal to
\begin{align}
& \P_{\cal G_\mu} \pb{E \vert_{e_\mu \cup p_\mu \cup q_\mu} \text{ is not 1-regular}}
\notag \\
\label{e:1reg_split1}
&\;=\; \P_{\cal G_\mu} \pb{E \vert_{e_\mu \cup p_\mu} \text{ is not 1-regular}}
\\ \label{e:1reg_split2}
& \qquad + 
\P_{\cal G_\mu} \pb{E \vert_{e_\mu \cup p_\mu} \text{ is 1-regular} \,,\, E \vert_{e_\mu \cup p_\mu \cup q_\mu} \text{ is not 1-regular}}
\end{align}
We first estimate \eqref{e:1reg_split1}.
Since $p_\mu$ is uniformly distributed under the constraint $p_\mu \notin \partial 1$,
we find that \eqref{e:1reg_split1} is bounded by $O(d / (dN)) = O(1/N)$.
Similarly, given $p_\mu$ such that $E \vert_{e_\mu \cup p_\mu}$ is 1-regular,
$q_\mu$ is uniformly distributed under the constraint $q_\mu \notin \partial 1 \cup \{p_\mu\}$.
Moreover, if $E \vert_{e_\mu \cup p_\mu \cup q_\mu}$ is not 1-regular then a vertex of $q_\mu$
must coincide with or be a neighbour of a vertex in $e_\mu \cup p_\mu$.
From this we deduce that \eqref{e:1reg_split2} is bounded by $O(d / (dN)) = O(1/N)$.
We have therefore estimated the first term on the right-hand side of \eqref{e:unif-Gmu-good-pf} by $O(1/N)$.

To estimate the second term on the right-hand side of \eqref{e:unif-Gmu-good-pf}, since
\begin{equation}
\{J_\mu(\f S) = 0 \}
\;=\; \{[S_\mu] \cap H_\mu \neq \h{1} \}
\;\subset\; \{e_\mu \cap H_\mu \neq \h{1} \} \cup  \{ (p_\mu \cup q_\mu) \cap H_\mu \neq \emptyset \}
\end{equation}
and since $\{e_\mu \cap H_\mu = \h{1}\}$ holds by assumption, 
it suffices to estimate $\Pmu((p_\mu \cup q_\mu) \cap H_\mu \neq \emptyset)$.
Clearly, $H_\mu \setminus \{1\}$ has at most $5(d-1)$ vertices.
This implies that $p \cap H_\mu \neq \emptyset$ for at most $O(d^2)$ edges $p$.
Taking $F(p,q) \deq \ind{p\cap H_\mu \neq \emptyset} + \ind{q \cap H_\mu \neq \emptyset}$
in \eqref{e:UM-Sapproxunif}, we therefore get $(p_\mu\cup q_\mu) \cap H_\mu \neq \emptyset$ with probability $O(d/N)$.
This proves that conditioned on $\sigG_\mu$, with probability $1-O(\frac{d}{N})$, the event \eqref{e:unif-G-good} holds.

Next, we show that conditioned on $\sigF_0$, with probability $1-O(\frac{d}{N})$, we also have $h_\mu=1$.
By a union bound and since $e_\nu \cap e_\mu = \{1\}$ if $\nu \neq \mu$,
\begin{equation}
  \P_{\sigF_0}(h_\mu = 0)
  \;=\;
  \P_{\sigF_0}(e_\mu \cap H_\mu \neq \h{1})
  \;\leq\;
  \sum_{\nu=1}^d \P_{\sigF_0}(e_\mu \cap (p_\nu \cup q_\nu) \neq \emptyset)
  \;=\; O\pbb{\frac{d}{N}},
\end{equation}
as claimed. 
This completes the proof of (i).

Next, we show (ii).
We write $\tilde W \equiv \tilde W^\mu$ and $\tilde S_\mu \equiv \tilde S_\mu^\mu$.
We then need to show $|W \triangle \tilde W \setminus \{\mu\}| \leq 10$.
For this, we make the following observations.
\begin{enumerate}
\item[1.]
If $[S_\nu] \cap [S_\kappa] \neq \h{1}$ for some $\nu \neq \kappa$, with both distinct from $\mu$,
then $\nu \notin W$ regardless of $u_\mu$; therefore then also $\nu \notin \tilde W$ and
$\nu \notin W \triangle \tilde W$.
 
Hence, all $\nu \in (W \triangle \tilde W)\setminus \{\mu\}$ satisfy $[S_\nu] \cap [S_\kappa] = \h{1}$ 
for all $\kappa \notin \{\mu,\nu\}$.
\item[2.]
Under this last condition, i.e.\ $\nu \neq \mu$ satisfies $[S_\nu] \cap [S_\kappa] = \h{1}$ for all $\kappa \notin \{\mu,\nu\}$,
by definition $\nu \in W$ if and only if $[S_\nu] \cap [S_\mu] = \h{1}$ and $E|_{[S_\nu]}$ is $1$-regular,
and $\nu \in \tilde W$ if and only if $[S_\nu] \cap [\tilde S_\mu] = \h{1}$ and $E|_{[S_\nu]}$ is $1$-regular.

Hence, $\nu\in (W\triangle \tilde W) \setminus \{\mu\}$
requires $[S_\nu] \cap [S_\mu] \neq \h{1}$ or $[S_\nu] \cap [\tilde S_\mu] \neq \h{1}$.
\end{enumerate}
We conclude that all $\nu \in (W\triangle \tilde W) \setminus \{\mu\}$ obey \eqref{e:unif-Fmu-maxbd}
or \eqref{e:unif-Fmu-maxbd} with $S_\mu$ replaced by $\tilde S_\mu$.
Therefore, Lemma~\ref{lem:UM-condprob}(i) implies $|(W \triangle \tilde W)\setminus \{\mu\}| \leq 5+5$,
as claimed. 

Finally, we establish (iii).
For this, observe that $I_\nu$ is independent of $p_\mu,q_\mu$ and that
$J_\nu$ is independent of $p_\mu,q_\mu$ on the event $\{ (p_\mu \cup q_\mu) \cap H_\mu = \emptyset \}$.
In the proof of (i), we have already shown that the latter event has probability at least $1-O(\frac{d}{N})$
conditioned on $\sigG_\mu$. This concludes the proof.
\end{proof}

\subsection{Proof of Proposition~\ref{prop:switch} for the uniform model}

With the preparations provided in Sections~\ref{sec:UM-switch}--\ref{sec:UM-exceptprob},
we now verify the claims of Proposition~\ref{prop:switch} for the uniform model.

\begin{proof}[Proof of Proposition~\ref{prop:switch}: uniform model]
The parametrization obeying Definitions~\ref{def:parametrization}--\ref{def:parametrization2}
was defined at the beginning of Section~\ref{sec:UM-exceptprob}.
The random variables $a_1, \dots, a_d, \alpha_1, \dots, \alpha_d$, and $A$ are defined as follows.
By definition, $A$ is the graph with edge set $T_{\f S, \f s} (E)$, i.e.
\begin{equation}
  A \;\deq\; M(T_{\f S,\f s}(E)) \,.
\end{equation}
Moreover, $\alpha_\mu$ is by definition the unique vertex incident to $1$ in the subgraph
\begin{equation*}
\begin{cases}
T_{\f S, \f s}(E)|_{[S_\mu]} & \text{if } \mu \in W
\\
\{e_\mu(E)\} & \text{if } \mu \notin W \,,
\end{cases}
\end{equation*}
where we recall that $W$ was defined in \eqref{e:UM-Wdef}.
The definition of $\alpha_\mu$ is illustrated in Figure~\ref{fig:alphamu}.
Then by Lemma~\ref{lem:UM-rev},
$A$ is the adjacency matrix of the uniform model.
To see that $\alpha_1, \dots, \alpha_d$ are an enumeration of the neighbours of $1$,
it suffices to show that $\alpha_\mu \neq \alpha_\nu$ for $\mu \neq \nu$.
This follows from the following simple observations: if $\mu,\nu \notin W$ then $e_\mu \neq e_\nu$;
if $\mu,\nu \in W$ then by definition of $W$ we have $[S_\mu] \cap [S_\nu] = \{1\}$;
if $\mu \in W$ and $\nu \notin W$ then by definition of $W$ we have $e_\nu \cap [S_\mu] = \{1\}$.
This proves (i). 

\begin{figure}[t]
\begin{center}
\input{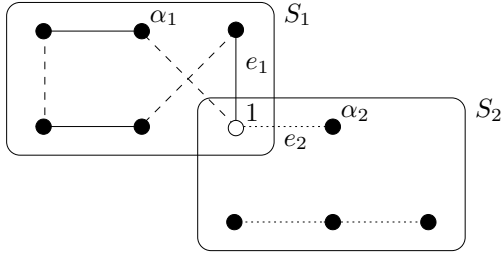}
\end{center}
\caption{
  Solid lines depict edges of $E \setminus T_{\f S,\f s}(E)$,
  dashed lines edges of $T_{\f S,\f s}(E) \setminus E$,
  and dotted lines edges of $E \cap T_{\f S,\f s}(E)$.
  For the graph $E$, the two encircled groups contain the three edges of $S_1$ and the three edges of $S_2$ respectively.
  In the switching $E \mapsto T_{\f S,\f s}(E)$, the solid edges $S_1$ are switched to the dashed configuration;
  hence, $1 \in W$ and $\alpha_1$ is the unique neighbour of $1$ in $[S_1]$ after the switching.
  On the other hand, $S_2$ is not switchable (so that $2 \notin W$) since $\abs{[S_2]}<6$, and $\alpha_2$
  is given by the original neighbour of $1$ in $e_2$.
  \label{fig:alphamu}}
\end{figure}

What remains is the definition of $a_\mu$ and the proof of (ii) and (iii).
To define $a_\mu$, we denote by $p_\mu$ and $q_\mu$ the two edges in $S_\mu$ not incident to 1,
ordered in an arbitrary fashion. (Note that, by definition of $\cal S_\mu$ from \eqref{e:def_calS},
$e_\mu$, $p_\mu$, and $q_\mu$ are distinct but not necessarily disjoint.)
We label the vertices of $p_\mu = \{p_\mu^1, p_\mu^2\}$ and $q_\mu = \{q_\mu^1,q_\mu^2\}$ in
an arbitrary fashion, and take the pair $(a_\mu, b_\mu)$ to be uniformly distributed
(parametrized by $s_\mu \in \qq{1,8}$) in the set
\begin{equation} \label{e:UM-abdef}
\hb{(p_\mu^1, q_\mu^1), (p_\mu^1, q_\mu^2), (p_\mu^2, q_\mu^1), (p_\mu^2, q_\mu^2), (q_\mu^1, p_\mu^1), (q_\mu^1, p_\mu^2), (q_\mu^2, p_\mu^1), (q_\mu^2, p_\mu^2)}\,.
\end{equation}
More precisely, we parametrize $(a_\mu, b_\mu) \equiv (a_\mu(S_\mu,s_\mu), b_\mu(S_\mu, s_\mu))$,
with $s_\mu \in \qq{1,8}$,
in such a way that, in the nontrivial case $I_\mu = 1$,
the switching $T_{S_\mu, s_\mu}(E)$ from \eqref{e:UM-TSsdef} is given
as in \eqref{e:MTtau} by
\begin{equation} \label{e:UM-Ttau}
  M(T_{S_\mu,s_\mu}(E)) \;=\; \tau_{r_\mu1, a_\mu\ua_\mu, b_\mu\ub_\mu}(M(E)) \,,
\end{equation}
where $\tau$ is defined in \eqref{e:Adoubleswitch},
and the vertices $r_\mu, \ul a_\mu, \ul b_\mu$ are defined by the conditions
$e_\mu = \{1, r_\mu\}$ and $\{p_\mu, q_\mu\} = \{\{a_\mu, \ul a_\mu\}, \{b_\mu, \ul b_\mu\}\}$.
Note that if $e_\mu, p_\mu, q_\mu$ are disjoint, $a_\mu$ is uniformly distributed on
$p_\mu \cup q_\mu$ and $b_\mu$ uniformly distributed on $p_\mu$ or $q_\mu$, whichever $a_\mu$ does not belong to.

We shall show (ii) and (iii) with the high-probability events 
given by those on which the conclusions of Lemma~\ref{lem:UM-condprob2} hold.
More precisely, the high-probability event in (iii)(1) is given by
\begin{equation} \label{e:unif-G-good-pf}
  \Xi_\mu \;\deq\;
  \h{I_\mu J_\mu = h_\mu} \cap \h{W \triangle \tilde W^\mu = \emptyset }
  \,
\end{equation}
and the high-probability event in (ii)(2) and (iii)(2) is given by
\begin{equation} \label{e:unif-F-good-pf}
  \Sigma_\mu \;\deq\;
  \h{I_\mu J_\mu = 1} \cap \h{W \triangle \tilde W^\mu = \emptyset }
  \;=\; \h{\mu \in W}  \cap \h{W \triangle \tilde W^\mu = \emptyset } \;=\; \Xi_\mu \cap \h{h_\mu = 1}
  \,.
\end{equation}
By Lemma~\ref{lem:UM-condprob2} (i) and (iii), recalling that $\cal F_0 \subset \cal G_\mu$
and $\frac{d}{N} = O\pb{\frac{1}{\sqrt{dD}}}$ by \eqref{e:D-UM},
\begin{equation} \label{Xi_Sigma_prob}
  \P_{\sigG_\mu}(\Xi_\mu) \;\geq\; 1-O\pbb{\frac{1}{\sqrt{dD}}}\,,\qquad
  \P_{\sigF_0}(\Sigma_\mu) \;\geq\; 1-O\pbb{\frac{1}{\sqrt{dD}}}\,.\qquad
\end{equation}

We now show (ii).
From the definition of $(a_\mu, b_\mu)$, we find that $a_\mu$ is chosen uniformly among the four (not necessarily distinct) vertices $p_\mu^1,p_\mu^2,q_\mu^1,q_\mu^2$. Therefore we get, for any function $f$ on $\qq{1, N}$ that
\begin{equation*}
\E_{\cal G_\mu} f(a_\mu) \;=\; \frac{1}{N} \sum_{i = 1}^N f(i) + O \pbb{\frac{1}{N}} \norm{f}_\infty\,,
\end{equation*}
where we used \eqref{e:UM-Sapproxunif} with $F(p_\mu,q_\mu) \deq \frac{1}{4} (f(p_\mu^1) + f(p_\mu^2) + f(q_\mu^1) + f(q_\mu^2))$ and the fact that each vertex is contained in exactly $d$ edges, so that $\frac{1}{Nd/2} \sum_{\{e^1, e^2\} \in E} \frac{1}{2} (f(e^1) + f(e^2)) = \frac{1}{N} \sum_{i = 1}^N f(i)$. This shows (ii)(1).
To verify (ii)(2), we use that on the event $\Sigma_\mu$ we have $\mu \in W$.
Moreover, from \eqref{e:UM-Ttau} we find that on $\Sigma_\mu$ we have $\{1,a_\mu\} \in T_{S_\mu, s_\mu}(E)$,
and consequently, using the definition of $W$ and \eqref{e:UM-Tdef}, $\{1, a_\mu\} \in T_{\f S, \f s}(E)$.
Since $a_\mu \in [S_\mu]$ the definition of $\alpha_\mu$ immediately implies that $a_\mu = \alpha_\mu$ on $\Sigma_\mu$.
Together with \eqref{Xi_Sigma_prob}, this concludes the proof of (ii)(2).

To show (iii), we often drop the sub- and superscripts $\mu$ and abbreviate $a \equiv \amu$, $\Ac \equiv \Acmu$, and so on.
For the first claim of (iii)(1), it suffices to show that $A-\Ac$ is always
a sum of at most $72$ terms $\pm \Delta_{xy}$.
By \eqref{e:UM-TM},
\begin{align}
A-\Ac
&\;=\;
\ind{\mu \in W}(M(T_{S_\mu,s_\mu}(E))-M(E)) - \ind{\mu \in \tilde W}(M(T_{\tilde S_\mu,\tilde s_\mu}(E))-M(E))
\nonumber\\
&\qquad
+ \sum_{\nu \in W \triangle \tilde W \setminus \{\mu\}} \pm (M(T_{S_\nu,s_\nu}(E))-M(E))
\,,
\end{align}
(where the sign $\pm$ in the last sum is $+$ if $\nu\in W$ and $-$ if $\nu \in \tilde W$).
In Lemma~\ref{lem:UM-condprob2} (ii), we proved that $|W \triangle \tilde W \setminus \{\mu\}| \leq 10$.
Therefore, since each term $M(T_{S,s}(E))-M(E)$ is the sum of six terms $\pm \Delta_{xy}$ by \eqref{e:Adoubleswitch},
we find that $A-\Ac$ is the sum of at most $12 \times 6 = 72$ terms $\pm \Delta_{xy}$, as desired.
This proves the first claim of (iii)(1).

Next, we verify the second claim of (iii)(1) and (iii)(2).
By \eqref{Xi_Sigma_prob} we may assume that the event $\Xi_\mu$ holds.
In particular, since $W \triangle \tilde W = \emptyset$,
\eqref{e:UM-TM} implies
\begin{equation} \label{e:AAcTTpf}
A-\Ac \;=\;
h \pB{ M(T_{S_\mu,s_\mu}(E))-M(T_{\tilde S_\mu,\tilde s_\mu}(E)) }
\,.
\end{equation}
In the case $h=0$, the right-hand side vanishes and
the second claim of (iii)(1) is trivial.
On the other hand, if $h =1$, by \eqref{e:UM-Ttau},
\begin{equation} \label{e:AAcTtaupf}
A-\Ac
\;=\;
M(T_{S_\mu,s_\mu}(E))-M(T_{\tilde S_\mu,\tilde s_\mu}(E))
\;=\;
\tau_{r1,a\ua,b\ub}(M(E))-\tau_{r1,\ac\uac,\bc\ubc}(M(E))
\,.
\end{equation}
As in the proof of (ii)(1), 
we find that conditioned on $\sigG_\mu$ each of the random variables $a,b,\ua,\ub$ is approximately uniform,
and using \eqref{e:UM-Sapproxunif1} instead of \eqref{e:UM-Sapproxunif}, that $b,\ub$ are each approximately uniform conditioned on $\sigG_\mu$ and $a,\ua$.
The same holds with $a,b,\ua,\ub$ replaced by $\ac,\bc,\uac,\ubc$. Thus, under the probability distribution conditioned on $\sigG_\mu$
and the event $\{h = 1\}$, we have
\begin{equation} \label{e:AAcDeltapf}
A-\Ac
\;=\;
\Delta_{1a}-\Delta_{1\ac} + X
\,,
\end{equation}
where $X$ is of the form \eqref{X_def} and has the property (a) or (b) from Remark~\ref{rk:switch3} (with all probabilities given by conditioning on $\cal G_\mu$).
Since for each term $\pm \Delta_{xy}$ in $X$ at least one of $x$ and $y$ is approximately uniform,
and since $a$ and $\ac$ are approximately uniform, we conclude that,
conditioned on $\cal G_\mu$ and the event $\Xi_\mu \cap \{h = 1\}$,
$A-\Ac$ is given by a sum of $10$ terms $\pm \Delta_{xy}$ such that at
least one of $x$ and $y$ is approximately uniform.
Together with the trivial identity $A - \tilde A = 0$ if $h = 0$ and \eqref{Xi_Sigma_prob},
this shows the second claim of (iii)(1).

Since $\Sigma_\mu = \Xi_\mu \cap \{h = 1\}$,
the identity \eqref{e:AAcDeltapf} also holds on the event $\Sigma_\mu$.
Recalling \eqref{Xi_Sigma_prob} and the form of $X$ from \eqref{X_def}, 
we obtain (iii)(2).
This concludes the proof.
\end{proof}

\section{Permutation model}
\label{sec:PM}

\newcommand{\ap}{a_+}
\newcommand{\am}{a_-}
\newcommand{\bp}{b_+}
\renewcommand{\bm}{b_-}
\newcommand{\apc}{\tilde a_+}
\newcommand{\amc}{\tilde a_-}
\newcommand{\bpc}{\tilde b_+}
\newcommand{\bmc}{\tilde b_-}

In this section we prove Proposition~\ref{prop:switch} for the permutation model.
First, we consider the $2$-regular random graph defined by a single uniform
permutation.
As before, the symmetric group of order $N$ is denoted by $S_N$,
and for any permutation $\sigma \in S_N$, the associated
\emph{symmetrized} permutation matrix is denoted by
\begin{equation}  \label{e:Asigma-recall}
  P(\sigma) \;\equiv\; P(\sigma^{-1}) \;\deq\; \sum_{i=1}^N \Delta_{i\sigma(i)}
  \,.
\end{equation}

\begin{figure}[t]
\begin{center}
\input{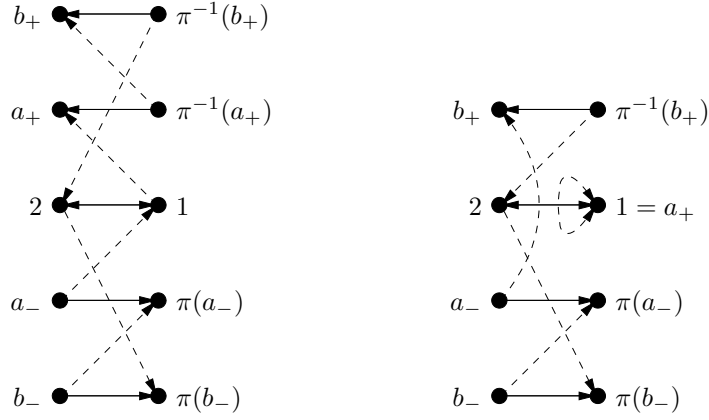}
\end{center}
\caption{The solid arrows depict the permutation $\pi$, and the dashed arrows the permutation $T_{\ap,\am,\bp,\bm}(\pi)$
  defined in \eqref{e:sigmamap-PM}.
  The left diagram depicts the generic case in which all of the elements of $\h{1,2,\ap,\am,\bp,\bm,\pi^{-1}(\ap),\pi^{-1}(\bp),\pi(\am),\pi(\bm)}$ are distinct.
  The right diagram depicts the case where $a_+=1$ but the remaining vertices are distinct; this leads to a loop at $1$.
  \label{fig:PM-switch}}
\end{figure}

Denote by $\gamma_{ij}=\gamma_{ji} \in S_N$ the transposition that exchanges $i$ and $j$.
For the remainder of this subsection, we identify $S_{N-2}$
as the subset of $S_N$ of permutations that exchange $1$ and $2$,
\begin{equation} \label{e:SNm2}
S_{N-2} \;\equiv\; \h{\pi \in S_N \col \pi(1) = 2, \pi(2) = 1}\,.
\end{equation}
For $\pi \in S_{N-2}$ and $\ap,\am,\bp,\bm \in \qq{1,N}$, we define $T_{\ap\am\bp\bm}(\pi) \in S_N$ by
\begin{equation} \label{e:sigmamap-PM}
  T_{\ap\am\bp\bm}(\pi) \;\deq\; \gamma_{2\bp} \gamma_{2\ap}\pi \gamma_{2\am} \gamma_{2\bm}
  \,.
\end{equation}
As illustrated in Figure~\ref{fig:PM-switch}, 
  in the case that $\h{1,2,\ap,\am,\bp,\bm,\pi^{-1}(\ap),\pi^{-1}(\bp),\pi(\am),\pi(\bm)}$ are distinct,
  the action of $T$ on $\pi$ amounts to two double switchings, as depicted in Figure~\ref{fig:switch}.

\begin{lemma} \label{lem:unif-PM}
If $(\pi, \ap,\am,\bp,\bm)$ is uniform 
on $S_{N-2} \times \qq{1,N} \times \qq{2,N}^3$,
then $T_{\ap\am\bp\bm}(\pi)$ is uniform on $S_N$.
Moreover, 
\begin{equation} \label{e:sigma1a}
  (T_{\ap\am\bp\bm}(\pi))(1) \;=\; \ap \,, \qquad (T_{\ap\am\bp\bm}(\pi))^{-1}(1) \;=\; \am \,,
\end{equation}
provided that $|\{1,2,\ap,\am,\bp,\bm\}| = 6$.
\end{lemma}

\begin{proof}[Proof of Lemma~\ref{lem:unif-PM}]
First, we show that the map
$\tilde\sigma: S_{N-2} \times \qq{1,N} \times \qq{2,N} \to S_N$ defined by
\begin{equation}
(\pi,\ap,\am) \;\longmapsto\; \tilde\sigma = \tilde\sigma(\pi,\ap,\am) \;\deq\; \gamma_{2\ap} \pi \gamma_{2\am}
\end{equation}
is a bijection. 
This follows from the following explicit inverse map $\tilde\sigma \mapsto (\pi,\ap,\am)$:
\begin{itemize}
\item[(i)]
if $\tilde\sigma(1) \neq 1$ then $\ap \deq \tilde\sigma(1)$, $\am \deq \tilde\sigma^{-1}(1)$,
and $\pi \deq \gamma_{2\ap} \sigma \gamma_{2\am}$; 
\item[(ii)]
if $\tilde\sigma(1) = 1$ then $\ap \deq 1$, $\am \deq \tilde\sigma^{-1}(2)$,
and $\pi \deq \tilde\sigma \gamma_{2\am}$.
\end{itemize}
Therefore, for all fixed $\bp,\bm \in \qq{1,N}$, also the map
\begin{equation}  \label{e:tildesigmapq}
  S_{N-2} \times \qq{1,N} \times \qq{2,N} \;\longrightarrow\; S_N, \quad (\pi,\ap,\am) \;\longmapsto\; \gamma_{2\bp} \gamma_{2\ap}\pi \gamma_{2\am}\gamma_{2\bm}
\end{equation}
is a bijection.
In particular, under \eqref{e:tildesigmapq}, the uniform distribution on $S_{N-2} \times \qq{1,N} \times \qq{2,N}$
is pushed forward to the uniform distribution on $S_N$.
Clearly, the distribution remains uniform after averaging over the independent random
variables $\bp,\bm \in \qq{2,N}$.
Finally, \eqref{e:sigma1a} is easily verified.
This completes the proof.
\end{proof}

The probability space for the permutation model (see Definition \ref{def:parametrization})
is realized as \eqref{e:Omega} endowed with the uniform probability measure,
where $\Theta \deq (S_{N-2})^{d/2}$ and
\begin{alignat}{2}
  U_\mu &\;\deq\;
  \qq{1,N} \times \qq{2,N} 
  &\qquad& (\mu \in \qq{1,d/2}) \,, \\
  U_\mu &\;\deq\;
  \qq{2,N} \times \qq{2,N}
  &\qquad& (\mu \in \qq{d/2 +1,d}) \,.
\end{alignat}
Elements of $\Theta$ and $U_\mu$ are written as $\theta = (\pi_1, \dots, \pi_{d/2}) \in \Theta$ and
$\umu = (a_{\mu},b_{\mu}) \in U_\mu$. For $\mu \in \qq{1,d/2}$ we define the random variable
\begin{equation} \label{e:sigmamu-PM}
  \sigma_\mu
  \;\deq\; T_{a_{\mu} a_{d-\mu} b_{\mu} b_{d-\mu}}(\pi_\mu)\,.
\end{equation}
By Lemma~\ref{lem:unif-PM}, $\sigma_1, \dots, \sigma_{d/2}$ are i.i.d.\ uniform permutations in $S_N$.
The adjacency matrix of the permutation model is given by
\begin{equation} \label{e:A-M-PM}
  A
  \;\deq\; \sum_{\mu=1}^{d/2} P(\sigma_\mu) 
  \,.
\end{equation}
It is convenient to augment the sequence $(\sigma_\mu)$ to be indexed by $\qq{1,d}$
by defining $\sigma_\mu \deq \sigma_{d-\mu}^{-1}$ for $\mu \in \qq{d/2 + 1, \mu}$.
Hence $P(\sigma_\mu)=P(\sigma_{d-\mu})$ for all $\mu \in \qq{1,d}$.
Also $\alpha_\mu \deq \sigma_\mu(1)$, where $\mu \in \qq{1,d}$, 
is an enumeration of the neighbours of $1$ in $A$.

\begin{proof}[Proof of Proposition~\ref{prop:switch}: permutation model]
We use the parametrization of the probability space defined above,
which satisfies the conditions of Definition \ref{def:parametrization},
and augment it according to Definition \ref{def:parametrization2}.
Then claim (i) follows immediately from Lemma~\ref{lem:unif-PM}.

To show (ii), we first recall
that $a_{\mu}$ is uniform on $\qq{1,N}$ if $\mu \in \qq{1, d/2}$
and that $a_{\mu}$ is uniform on $\qq{2,N}$ if $\mu \in \qq{d/2+1, d}$.
Either way, conditioned on $\cal G_\mu$, $a_\mu$ is approximately uniform on $\qq{1,N}$.
Moreover, for $\mu \in \qq{1,d}$,
conditioned on $\sigF_0$, with probability $1-O(\frac1N)$,
the event $\{|\{1,2,a_\mu,b_\mu,a_{d-\mu},b_{d-\mu}\}|=6\}$ holds.
By Lemma~\ref{lem:unif-PM}, on this event we have, for $\mu \in \qq{1,d/2}$, 
$\alpha_\mu =  \sigma_\mu(1)=a_{\mu}$ and $\alpha_{d-\mu} = \sigma_\mu^{-1}(1) = a_{d-\mu}$.
This concludes the proof of (ii).

What remains is the proof of (iii). We fix $\mu \in \qq{1,d}$, and drop the
index $\mu$ from the notation and write
$\Ac \equiv \Acmu$, $\sigma \equiv \sigma_\mu$ and $\sigma_- \equiv \sigma_{d-\mu}$, and so forth.
Then set
\begin{equation}
  \varrho \;\deq\;
  \gamma_{2b_-} \gamma_{2a_-} \pi^{-1} \,,
\end{equation}
so that $\sigma = \gamma_{2b} \gamma_{2a} \varrho^{-1}$.
Given $\varrho$, set $A_\varrho(a,b) \deq P(\gamma_{2b}\gamma_{2a} \varrho^{-1})$.
Then, for all $a,b,\ac, \bc$,
\begin{align}
\label{e:EHHexpa-pf}
A_\varrho(\ac,b)-A_\varrho(a,b)
&\;=\; \Delta_{\varrho \ac,b} - \Delta_{\varrho a, b} + \Delta_{\varrho a, \gamma_{2b}  \gamma_{2\ac} a} - \Delta_{\varrho \ac,\gamma_{2b}  \gamma_{2a} \ac} + \Delta_{\varrho 2, \gamma_{2b} \ac} - \Delta_{\varrho 2, \gamma_{2b} a}\,,
\\
\label{e:EHHexpp-pf}
A_\varrho(a,\bc)-A_\varrho(a,b)
&\;=\;  \Delta_{\tau \bc,2} - \Delta_{\tau b,2} + \Delta_{\tau b,  \gamma_{2\bc} b} - \Delta_{\tau \bc, \gamma_{2b}\bc} + \Delta_{\tau 2, \bc} - \Delta_{\tau 2, b} \,,
\end{align}
where $\tau \deq \varrho\gamma_{2a}$.
Indeed, to verify \eqref{e:EHHexpa-pf},
note that $A_\varrho(\ac,b)- A_\varrho(a,b)$ is given by
\begin{equation} 
  \sum_{i=1}^N \pB{\Delta_{i,\gamma_{2b}\gamma_{2\ac}\varrho^{-1}(i)} - \Delta_{i,\gamma_{2 b}\gamma_{2 a}\varrho^{-1}(i)}}
  \;=\; \sum_{i=1}^N \pB{\Delta_{\varrho(i), \gamma_{2b}\gamma_{2\ac}(i)} - \Delta_{\varrho(i),\gamma_{2b}\gamma_{2a}(i)}}
  \,,
\end{equation}
and that 
the differences under the last sum vanish unless $i\in\{2,a,\ac\}$,
and are given by \eqref{e:EHHexpa-pf}.
Similarly, for the difference $A_\varrho(a,\bc)-A_\varrho(a,b)$, we obtain
\begin{equation}
  \sum_{i=1}^N \pB{
    \Delta_{i,\gamma_{2\bc}\gamma_{2a}\varrho^{-1}(i)}
    -
    \Delta_{i,\gamma_{2b}\gamma_{2a}\varrho^{-1}(i)}
  }
  = \sum_{i=1}^N \pB{ \Delta_{\tau(i),\gamma_{2\bc}(i)} - \Delta_{\tau(i),\gamma_{2 b}(i)} } \,,
\end{equation}
and observe from this representation that only
$i \in \{2,b,\bc\}$ yields a nonzero contribution,
and that the corresponding terms are given by \eqref{e:EHHexpa-pf}--\eqref{e:EHHexpp-pf}.

To prove (iii), observe that only the term $P(\sigma_\mu)=P(\sigma_{d-\mu})$ in \eqref{e:A-M-PM}
contributes to $\Ac-A$ and that therefore $\Ac-A = A_\varrho(\ac,\bc)-A_\varrho(a,b)$.  
From \eqref{e:EHHexpa-pf}--\eqref{e:EHHexpp-pf}, it is straightforward
to verify (iii).
The first statement of (iii)(1) holds since \eqref{e:EHHexpa-pf}--\eqref{e:EHHexpp-pf}
contains at most 12 terms $\pm \Delta_{xy}$.
For the second statement of (iii)(1), note that,
since $\varrho$ is independent of $a,\ac,b,\bc$, and $\tau$ independent of $b,\bc$,
for each of these terms $\Delta_{xy}$ in \eqref{e:EHHexpa-pf}--\eqref{e:EHHexpp-pf}, 
at least one of $x,y$ is approximately uniform
since $a,\ac, b,\bc$ are approximately uniform.
For (iii)(2), observe that, conditioned on $\sigF_0$,
$\varrho_\mu(2)=1$ and $2,a,\ac,b,\bc$ are distinct with probability $1-O(\frac{1}{N})$.
On this event,
(iii)(2) can be verified directly from \eqref{e:EHHexpa-pf}--\eqref{e:EHHexpp-pf}.
For example,
$\Delta_{\varrho 2, \gamma_{2b} \ac} = \Delta_{1\ac}$ and $\Delta_{\varrho 2, \gamma_{2b} a} = \Delta_{1a}$ in \eqref{e:EHHexpa-pf}.
We skip the details for the other terms.
\end{proof}

\section{Isotropic local law and probabilistic local quantum unique ergodicity}
\label{sec:isotropic}

In this section we state and prove the \emph{isotropic local semicircle law} for $A$,
which controls the difference $G - m I$ in the sense of generalized matrix entries
$\scalar{\f a}{G \f b} - m \scalar{\f a}{\f b}$,
instead of the standard matrix entries from Theorem~\ref{thm:semicircle} obtained by taking
$\f a$ and $\f b$ to lie in the standard coordinate directions.
The arguments used in this section rely crucially on the \emph{exchangeability} of the random regular graph.
This is different from the remainder of the paper, in which we did not exploit exchangeability in an essential way.
An isotropic local law was first proved for Wigner matrices in \cite{MR3103909} and
subsequently extended to generalized Wigner matrices and sample covariance
matrices with uncorrelated population in \cite{MR3183577}. Recently, such control was also obtained for sample
covariance matrices with general population \cite{1410.3516}, in which case $G$ is approximated
by an anisotropic matrix that is not a multiple of the identity.

As applications of the isotropic local law,
we establish the isotropic delocalization of eigenvectors (Corollary~\ref{cor:iso_deloc})
and a local quantum unique ergodicity result (Corollary~\ref{cor:QUE}).
In the following we call an $\ell^2$-normalized vector a unit vector, and write $\f a \perp \f e$ if $\sum_{i} a_i e_i = 0$.

\begin{theorem}[Isotropic local law for random regular graphs] \label{thm:isotropic}
Under the assumptions of Theorem \ref{thm:semicircle}, for any
deterministic unit vectors $\f a, \f b \perp \f e$,
any $\zeta \gg 1$, and any $z \in \C_+$ satisfying $\eta \gg \xi^2/N$, we have
\begin{equation} \label{Gab_iso}
\scalar{\f a}{G(z) \f b} - m(z) \scalar{\f a}{\f b} \;=\; O \pb{F_z(\xi \Phi(z)) + \xi \zeta^4 \Phi(z)}
\end{equation}
with probability at least $1 - \ee^{-\xi \log \xi} - \ee^{-\zeta \sqrt{\log \zeta}}$.
\end{theorem}

Note that the isotropic law in the subspace spanned by $\f e$ is trivial since $G(z) \f e = -z^{-1} \f e$.
Theorem \ref{thm:isotropic} follows immediately from Theorem \ref{thm:semicircle} and the following general
result for exchangeable random matrices.
Recall that a random vector $(Y_i)_{i = 1}^N \in \C^N$
is called \emph{exchangeable} if for any permutation $\sigma \in S_N$ we have
\begin{equation}
(Y_i)_i \;\eqdist\; (Y_{\sigma(i)})_i\,.
\end{equation}
Similarly, a random matrix $(Y_{ij})_{i,j = 1}^N \in \C^{N \times N}$
is called \emph{exchangeable} if for any $\sigma \in S_N$ we have
\begin{equation}
(Y_{ij})_{i,j} \;\eqdist\; (Y_{\sigma(i)\sigma(j)})_{i,j}\,.
\end{equation}
In particular, the (normalized and centred) adjacency matrices of any of the models of random
$d$-regular graphs introduced in Section~\ref{sec:intro-models} are exchangeable.

\begin{theorem}[General isotropic local law] \label{thm:iso_gen}
Let $G$ be the Green's function \eqref{e:Gdef} of an
exchangeable random matrix $H$ at some $z \in \C_+$ .
Then, for any deterministic $\Psi_o,\Psi_d > 0$, $m \in \C$, unit vectors $\f a, \f b \perp \f e$, and $\zeta  \gg 1$, we have
\begin{equation}
\scalar{\f a}{G \f b} - m \scalar{\f a}{\f b} \;=\; O \pb{\Psi_d + \zeta^4 \Psi_o}
\end{equation}
with probability at least $\P(\max_i \abs{G_{ii} - m} \leq \Psi_d, \; \max_{i \neq j} \abs{G_{ij}} \leq \Psi_o) - \ee^{-\zeta  \sqrt{\log \zeta}}$.
\end{theorem}

The proof of Theorem~\ref{thm:iso_gen}
follows from the following moment bounds for exchangeable random matrices.
The estimate \eqref{e:aY} was previously established for $p=2,4$ in \cite{MR3129806,MR3227063}.

\begin{proposition} \label{prop:qY}
Let $a_1, \dots, a_N \in \C$ be deterministic with $\sum_{i = 1}^N a_i = 0$ and $\sum_{i=1}^N \abs{a_i}^2\leq1$.
\begin{enumerate}
\item 
Let $(Y_i)_{i = 1}^N$ be a exchangeable random vector.
Then for all $p \geq 1$ we have
\begin{equation} \label{e:aY}
\normBB{\sum_{i=1}^N a_i Y_i}_p \;=\; O\pbb{\frac{p^2}{\log p}} \, \|Y_1\|_{p}\,.
\end{equation}
\item 
Let $(Y_{ij})_{i,j = 1}^N$ be a exchangeable random matrix.
Then for all $p \geq 1$ we have
\begin{equation} \label{e:aY2}
\normBB{\sum_{i,j=1}^N \bar a_i a_j Y_{ij}}_p \;\leq\; \norm{Y_{11}}_p + O\pbb{\frac{p^2}{\log p}}^2 \|Y_{12}\|_p\,.
\end{equation}
\end{enumerate}
\end{proposition}

The proof of Proposition \ref{prop:qY} is given in Appendix \ref{sec:exchangeable}.

\begin{proof}[Proof of Theorem \ref{thm:iso_gen}]
By polarization and homogeneity, it suffices to consider the case where $\f a = \f b$ is a unit vector perpendicular to $\f e$.
Define $Y_{ij} \deq \phi\ind{i\neq j} G_{ij}$
with the indicator function $\phi \deq \ind{\max_{i \neq j} \abs{G_{ij}} \leq \Psi_o}$.
Then $(Y_{ij})$ is an exchangeable random matrix.
By Proposition~\ref{prop:qY}~(ii) with $p=\zeta \sqrt{\log \zeta}$ and $\zeta \gg 1$, we get using
using Markov's inequality 
\begin{equation*}
\P \qB{\abs{\scalar{\f a}{Y \f a}} > C\zeta^4 \Psi_o} \;\leq\; \ee^{-\zeta\sqrt{\log\zeta}}
\end{equation*}
for some constant $C$.
Since for any unit vector $\f a$ we have
$\phi|\scalar{\f a}{(G-m) \f a}| \leq \max_{i} |G_{ii}-m| + |\scalar{\f a}{Y \f a}|$,
the claim now follows by a union bound.
\end{proof}

The isotropic local law implies the \emph{isotropic delocalization} of the eigenvectors of $A$,
which also follows from Corollary~\ref{cor:eigenvectors} and Proposition~\ref{prop:qY}~(i),
similarly to the proof of Theorem \ref{thm:iso_gen}.

\begin{corollary}[Isotropic eigenvector delocalization] \label{cor:iso_deloc}
Under the assumptions of Theorem~\ref{thm:semicircle}, for any unit eigenvector $\f v$ of $A$ or $H$,
any deterministic unit vector $\f a \perp \f e$, and any $\zeta \gg 1$,
we have $\scalar{\f a}{\f v} = O(\xi \zeta^2 / \sqrt{N})$ with probability at least $1-\ee^{-\xi\log\xi} - \ee^{-\zeta  \sqrt{\log \zeta}}$.
\end{corollary}

Finally, we note that Corollary \ref{cor:eigenvectors} and Proposition \ref{prop:qY} (i) also imply the following local quantum unique ergodicity result, similarly to the proof of Theorem \ref{thm:iso_gen}.
\begin{corollary}[Probabilistic local quantum unique ergodicity]  \label{cor:QUE}
Let $a \col \qq{1,N} \to \R$ be a deterministic function satisfying $\sum_{i = 1}^N a_i = 0$. Under the assumptions of Theorem~\ref{thm:semicircle}, for any unit eigenvector $\f v = (v_i)_{i = 1}^N$ of $A$ or $H$ and for any $\zeta \gg 1$, we have
\begin{equation} \label{e:QUE}
\sum_{i = 1}^N a_i v_i^2 \;=\; O \pBB{\frac{(\xi \zeta)^2}{N} \pBB{\sum_{i = 1}^N a_i^2}^{1/2}}
\end{equation}
with probability at least $1-\ee^{-\xi\log\xi} - \ee^{-\zeta \sqrt{\log \zeta}}$.
\end{corollary}

Corollary \ref{cor:QUE} states that, on deterministic sets of at least $(\xi \zeta)^4$ vertices, 
all eigenvectors of the random graph $A$ are completely flat with high probability. In other words, with high probability, the random probability measure $i \mapsto v_i^2$
is close (when tested against deterministic test functions) to the uniform probability measure $i \mapsto 1/N$ on $\qq{1,N}$. 
For instance, let $I \subset \qq{1,N}$ be a deterministic subset of vertices. Setting $a_i \deq \ind{i \in I} - \abs{I} / N$
in Corollary \ref{cor:QUE}, we obtain
\begin{equation} \label{e:QUE_example}
\sum_{i \in I} v_i^2 \;=\; \sum_{i \in I} \frac{1}{N} + O \pbb{\frac{(\xi \zeta)^2 \sqrt{\abs{I}}}{N}}
\end{equation}
with probability at least $1-\ee^{-\xi\log\xi} - \ee^{-\zeta}$.
The main term on the right-hand side of \eqref{e:QUE_example} is much larger than the error term provided that $\abs{I} \gg (\xi \zeta)^4$.
Note that we can obtain \eqref{e:QUE} and \eqref{e:QUE_example} asymptotically almost surely with $(\xi \zeta)^2 = (\log N)^4$ by
choosing $\xi \log \xi = C (\log N)^2$ and $\zeta = C^{-1} \log \xi$ for some large enough constant $C > 0$,
so that \eqref{e:QUE_example} is a nontrivial statement for $\abs{I} \geq (\log N)^8$.

The celebrated quantum chaos conjecture 
states that the eigenvalue statistics of the
quantization of a chaotic classical system are governed by random matrix theory
\cite{PhysRevLett.52.1, MR1266075,MR2774090,MR916129}.
Random regular graphs are considered a good paradigm for probing quantum chaos; see \cite{MR3204183} for a review.
For generalized Wigner matrices, a probabilistic version of QUE,
as well as the Gaussian distribution of eigenvector components, was proved in \cite{BY2016}.
The first result of this kind, for a smaller class of Wigner matrices, was obtained in \cite{MR3034787,MR2930379}.
Moreover, using the local law proved in this paper, these results can be extended to the $d$-regular graph as well \cite{2016Huang}.

\begin{remark}[Erd\H{o}s-R\'enyi graphs]
We conclude this section by remarking that all of the results from this
section -- Theorems~\ref{thm:isotropic} and \ref{thm:iso_gen} and Corollaries~\ref{cor:iso_deloc} and \ref{cor:QUE} -- have
analogues for the adjacency matrix of the Erd\H{o}s-R\'enyi graph,
whose proofs follow in exactly the same way using Proposition~\ref{prop:qY} and \cite[Theorem 2.9]{MR3098073}.
We leave the details to the interested reader.
\end{remark}

\appendix

\section{Improved bound near the edges}
\label{app:Psi}

In this appendix, we sketch the changes required to improve \eqref{e:semicircle} 
such that $\xi\Phi$ is replaced by \eqref{e:Psibd} on the right-hand sides, namely by
\begin{equation}
  \xi\Psi + \pbb{\frac{\xi}{N\eta}}^{2/3} \qquad \text{where } \Psi \;\deq\; \sqrt{\frac{\im m}{N\eta}}+ \frac{1}{\sqrt{D}} \,.
\end{equation}

First, analogously to $\Gamma$ and $\Gammamu$, define
\begin{equation}
  \Upsilon \;\deq\; \max_i \im G_{ii}\,, \qquad \Upsilon_\mu \;\deq\; \|\Upsilon\|_{L^\infty(\sigG_\mu)}\,.
\end{equation}
Then it is easy to see that Lemma~\ref{lem:Gswitch} (ii) can be improved to replace $\Gammamu^4\Phi^2$ by $\Phi_\mu^2$ where
\begin{equation} \label{e:Phimudef}
  \Phi_\mu \;\deq\; \sqrt{\frac{\Upsilon_\mu}{N\eta}} + \frac{1}{N\eta} + \frac{\Gammamu^2}{\sqrt{D}}
  \,.
\end{equation}
Assuming that $\Gammamu=O(1)$ and that $\Upsilon_\mu = O(\delta)$, we have $\Phi_\mu = O(\Phi_\delta)$ where
\begin{equation}
  \Phi_{\delta} \;\deq\; \sqrt{\frac{\delta}{N\eta}} + \frac{1}{N\eta} + \frac{1}{\sqrt{D}} \,.
\end{equation}

Next, Proposition~\ref{prop:Lambda} can be improved as follows.

\begin{proposition}[Improved version of Proposition~\ref{prop:Lambda}]
\label{prop:Lambda-delta}
Suppose that $\xi>0$, $\zeta > 0$, and that $D \gg \xi^2$.
Let $\delta \geq N^{-C}$ be deterministic. 
If for $z \in \f D$ we have 
\begin{equation*} 
  \Gamma^*(z) \;=\; O(1) \,, \qquad \Upsilon \;\leq\; \delta 
\end{equation*}
with probability at least $1-\ee^{-\zeta}$, then
\begin{equation} \label{e:Lambdaind-delta}
  \max_i \abs{G_{ii}(z) - m(z)} \;=\; O(F_z(\xi \Phi_\delta(z)))\,, \qquad \max_{i \neq j} \abs{G_{ij}(z)} \;=\; O(\xi\Phi_\delta(z))
  \,,
\end{equation}
with probability at least $1-\ee^{-(\xi\log\xi)\wedge \zeta + O(\log N)}$.
\end{proposition}

\begin{proof}[Sketch of proof]
We first verify that in all estimates of Sections \ref{sec:concentration}--\ref{sec:expectation},
the parameter $\Phi$ arises from just two possible sources: from $D^{-1/2}$ or Lemma~\ref{lem:Gswitch} (ii).

In particular, by the improved version of Lemma~\ref{lem:Gswitch} discussed around \eqref{e:Phimudef},
Lemma~\ref{lem:EGdiff} can be improved so that $\Gammamu^6\Phi^2$ is replaced by $\Gammamu^2 \Phi_\mu^2$.
This implies an improved version of Proposition~\ref{prop:concentration} 
in which the assumption that $\Gamma = O(1)$ holds with probability at least $1-\ee^{-\zeta}$
is replaced by the assumption $\Gamma = O(1)$ and $\Upsilon = O(\delta)$ with the same probability,
and $\Phi$ is replaced by $\Phi_\delta$ in the conclusion.
The proof of the improved version of Proposition~\ref{prop:concentration} is then analogous to the proof of Proposition~\ref{prop:concentration}
given in Section~\ref{sec:concentration}. In particular, note that using $\delta \geq N^{-C}$,
it is easy to verify that there are $Y_\mu$ satisfying all required conditions and
$\Gamma^{2p}\Phi_\mu^2 \leq Y_\mu^2\Phi_\delta^2$.

Similarly, given the improved versions of Lemma~\ref{lem:Gswitch} and Proposition~\ref{prop:concentration},
the proof of the improved version of Proposition~\ref{prop:Lambda} is identical to that given in Section~\ref{sec:expectation}.
\end{proof}

Finally, given Proposition~\ref{prop:Lambda-delta},
the proof of Theorem~\ref{thm:semicircle} involves
a slightly more involved induction than that given in Section~\ref{sec:outline}, in which we propagate both estimates
\begin{equation} \label{e:imGii-ind}
\Gamma \;\leq\; 2\,, \qquad  \Upsilon \;\leq\; Q^2 \pBB{ \im m + \sqrt{\xi\Psi + \pbb{\frac{\xi}{N\eta}}^{2/3}} }\,,
\end{equation}
simultaneously, for some sufficiently large constant $Q$.
The hypothesis \eqref{e:imGii-ind} is trivial for $\eta \geq 1$.
By an explicit spectral decomposition (or by an argument analogous to the proof of Lemma~\ref{lem:Gbd}), it is easy to verify that
$\eta \im G_{ii}$ is increasing in $\eta$, and hence \eqref{e:imGii-ind} for some $z=E+\ii\eta$
implies
\begin{equation} \label{e:imGii-ind2}
  \Upsilon(z') \;\leq\; \delta(z') \;\deq\;
  4Q^2 \pBB{ \im m(z') + \sqrt{\xi\Psi(z') + \pbb{\frac{\xi}{N\eta'}}^{2/3}} }\,,
\end{equation}
for $z'=E+\ii\eta'$ with $\eta'=\eta/2$.
Hence, we may apply Proposition~\ref{prop:Lambda-delta}.
It is easy to verify that
\begin{equation}
  \xi \Phi_\delta \;\leq\; O(Q^2) \pBB{\xi\Psi + \pbb{\frac{\xi}{N\eta}}^{2/3}} \,.
\end{equation}
Therefore, using $F_z(r) \leq \sqrt{r}$, we get from Proposition~\ref{prop:Lambda-delta} that
\begin{equation}
  \Upsilon(z') \;\leq\; \im m(z') + O(\sqrt{\xi\Phi_\delta})
  \;\leq\; O(Q) \pBB{ \im m + \sqrt{\xi\Psi + \pbb{\frac{\xi}{N\eta}}^{2/3}} }\,,
\end{equation}
and this propagates the induction hypothesis \eqref{e:imGii-ind} since $O(Q) \leq Q^2$ for a sufficiently
large $Q$.

\section{Moment bounds for exchangeable random matrices: proof of Proposition \ref{prop:qY}} \label{sec:exchangeable}

In this appendix we prove Proposition \ref{prop:qY}.
To avoid extraneous notational complications arising from complex conjugates,
we suppose that all quantities are real-valued. We abbreviate $\qq{n} \deq \qq{1,n}$,
and denote by $\fra P_n$ the set of partitions of $\qq{n}$.
By H\"older's inequality, it suffices to consider the case $p \in 2\N$.

We begin with (i).
Abbreviate $X \deq \sum_i a_{i} Y_{i}$.
For $\f i \in \qq{N}^p$ define $P(\f i) \in \fra P_p$ as the
partition generated by the equivalence relation $k \sim l$ if and only if $i_k = i_l$.
Then we get
\begin{equation} \label{EXp0}
\E X^p \;=\;  \sum_{\f i \in \qq{N}^p} \prod_{k = 1}^p a_{i_k} \, \E \prod_{k = 1}^p Y_{i_k} \;=\;
\sum_{\Pi \in \fra P_p} K(\Pi) \sum_{\f i \in \qq{N}^p} \ind{P(\f i) = \Pi} \prod_{k = 1}^p a_{i_k}\,,
\end{equation}
where $K(\Pi) \deq \E \prod_{k = 1}^p Y_{i_k}$ for any $\f i$ satisfying $P(\f i) = \Pi$.
That $K(\Pi)$ is well defined, i.e.\ independent of the choice of $\f i$, follows from the exchangeability of $(Y_i)$.
For future use we also note that by H\"older's inequality we have $|K(\Pi)| \leq \|Y_1\|_p^p$.
We use the notation $\pi \in \Pi$ for the blocks of $\Pi$.
Next, we rewrite the sum over $\f i \in \qq{N}^p$ as a sum over $\f r = (r_\pi)_{\pi \in \Pi} \in \qq{N}^\Pi$ to get
\begin{equation} \label{EXp}
\E X^p \;=\; \sum_{\Pi \in \fra P_p} K(\Pi) \sum_{\f r \in \qq{N}^\Pi}^* \prod_{\pi \in \Pi} a_{r_\pi}^{\abs{\pi}}\,,
\end{equation}
where the star on top of the of the sum indicates summation over distinct indices $\f r$, i.e.
\begin{equation} \label{repulsion}
\sum_{\f r}^* \;=\; \sum_{\f r} \prod_{e \in \cal E(\Pi)} (1 - I_e(\f r))\,,
\end{equation}
where $\cal E(\Pi) \deq \{\{\pi,\pi'\} \col \pi,\pi' \in \Pi, \pi \neq \pi'\}$ is the set of edges of
the complete graph on the vertex set $\Pi$, and $I_{\{\pi,\pi'\}}(\f r) \deq \ind{r_\pi = r_{\pi'}}$. 

We need to estimate the right-hand side of \eqref{EXp} by exploiting the condition $\sum_i a_i = 0$.
Obtaining the bound of order $(p^2/\log p)^p$ requires some care in handling the combinatorics.
We shall multiply out the product in \eqref{repulsion}, which has to be done with moderation to avoid overexpanding,
since the resulting sum is highly oscillatory.
The naive expansion $\prod_{e \in \cal E(\Pi)} (1 - I_e) = \sum_{E \subset \cal E(\Pi)} \prod_{e \in E} (-I_e)$ is too rough.
Instead, we only expand a subset of the edges $\cal E(\Pi)$, and leave some edges $e$ unexpanded,
meaning that the associated factors $(1 - I_e)$ remain.

The partial expansion of the product $\prod_{e \in \cal E(\Pi)} (1 -
I_e)$ is best formulated using edge-coloured graphs. We consider graphs on $\Pi$ whose edges are coloured
black or white, and denote by $B$ the set
of black edges and by $W$ the set of white edges.
Thus, an edge-coloured graph is a pair $(B,W) \subset \cal E(\Pi)^2$ satisfying $B \cap W = \emptyset$.
For any edge-coloured graph $(B,W)$ we define
\begin{equation} \label{bw_edges}
J_{B,W} \;\deq\; \prod_{e \in B} (- I_e) \prod_{e \in W} (1 - I_e)\,.
\end{equation}
Hence,
each black edge $e \in B$ encodes the indicator function $-I_e$ and
each white edge $e \in W$ the indicator function $1 - I_e$.
Note that for $e \in W$ we have the trivial identity
\begin{equation} \label{expansion_iter}
J_{B,W} \;=\; J_{B, W \setminus \{e\}} + J_{B \cup \{e\}, W \setminus \{e\}}\,.
\end{equation}
We shall define a process on the set of edge-coloured graphs that operates on each white edge, either
leaving it as it is or generating two new graphs using \eqref{expansion_iter},
one graph where this white edge is removed, and another graph where the white edge is replaced by a black one.
To that end, we choose a total order on $\cal E(\Pi)$ and denote by
$e-$ and $e+$ the immediate predecessor and successor of $e$.
We denote by $e_{\text{min}}$ and $e_{\text{max}}$ the smallest and
largest edges of $\cal E(\Pi)$,
and introduce the formal additional edge $0$ to be the
immediate predecessor of $e_{\text{min}}$.

For each $e \in \cal E(\Pi)$ we shall
define a set of $\cal G_\Pi(e)$ of edge-coloured graphs $(B,W)$ such that
$W$ contains all edges greater than $e$.
The sets $\cal G_\Pi(e)$ are defined recursively as follows. First, we set $\cal G_\Pi(0) \deq \{(\emptyset, \cal E(\Pi))\}$. Thus, $\cal G_\Pi(0)$ consists of the complete graph with all edges coloured white.
Then for $e_{\text{min}} \leq e \leq e_{\text{max}}$ the set $\cal G_\Pi(e)$ is obtained from $\cal G_\Pi(e-)$ by
\begin{equation} \label{def_calG}
\cal G_\Pi(e) \;\deq\; \bigcup_{(B,W) \in \cal G_\Pi(e-)} \cal U(B,W,e)\,,
\end{equation}
where $\cal U(B,W,e)$ is a set of one or two edge-coloured graphs obtained from $(B,W)$, using one of the two formulas
\begin{align}
\label{U2}
\cal U(B,W,e) \;&=\; \{(B, W \setminus \{e\}), (B \cup \{e\}, W \setminus \{e\})\}\,,
\\
\label{U1}
\cal U(B,W,e) \;&=\; \{(B,W)\}\,; 
\end{align}
which choice among \eqref{U2} and \eqref{U1} to make will be determined in \eqref{choice_U} below.
The choice \eqref{U2} amounts to multiplying out $1 - I_e$
and \eqref{U1} to not multiplying out $1 - I_e$.
Note that, by construction, we always have $e \subset W$ on the right-hand side of \eqref{def_calG}.
Moreover, by \eqref{expansion_iter}, no matter which choice we make between \eqref{U2} and \eqref{U1},
we always have the identity
\begin{equation*}
\sum_{(B,W) \in \cal G_\Pi(e-)} J_{B,W} \;=\; \sum_{(B,W) \in \cal G_\Pi(e)} J_{B,W}
\end{equation*}
for all $e \in \cal E(\Pi)$, and hence by induction
\begin{equation} \label{expansion_result}
\prod_{e \in \cal E(\Pi)} (1 - I_e) \;=\; \sum_{(B,W) \in \cal G_\Pi(e_{\text{max}})} J_{B,W}\,.
\end{equation}
Note that always choosing \eqref{U2} leads to the identity
$\prod_{e \in \cal E(\Pi)} (1 - I_e) = \sum_{B \subset \cal E(\Pi)} \prod_{e \in B} (-I_e)$,
which, as explained above, is too rough;
conversely, always using \eqref{U1} leads to the trivial identity
$\prod_{e \in \cal E(\Pi)} (1 - I_e) = \prod_{e \in \cal E(\Pi)} (1 - I_e)$.

In order to define which choice of $\cal U$ in \eqref{U2}--\eqref{U1} we make,
we also colour the vertices $\Pi$ black or white.
A vertex is black if it is a block of size one and white if it is a block of size greater than one.
This defines a splitting of the vertices $\Pi= \Pi_1 \sqcup \Pi_2$
into black vertices $\Pi_1$ are white vertices $\Pi_2$, and also induces
a splitting of the edges $\cal E(\Pi) = \cal E_1(\Pi) \sqcup \cal E_{12}(\Pi) \sqcup \cal E_2(\Pi)$,
where $\cal E_i(\Pi)$ is the set of edges connecting two vertices of $\Pi_i$ for $i = 1,2$,
and $\cal E_{12}(\Pi)$ is the set of edges connecting two vertices of different colours.
We choose the total order on $\cal E(\Pi)$ so that $\cal E_1(\Pi) < \cal E_{12}(\Pi) < \cal E_{2}(\Pi)$.
With this order, we define our choice of $\cal U$:
\begin{equation} \label{choice_U}
\text{use }
\begin{cases}
\eqref{U2} & \text{if $e$ is incident to a black vertex that is not incident to a black edge,}
\\
\eqref{U1} & \text{otherwise.}
\end{cases}
\end{equation}
See Figure~\ref{fig:colouredgraphs} for an illustration of the resulting process on coloured graphs.

Let $(B,W) \in \cal G_\Pi(e_{\text{max}})$. 
The following properties can be checked in a straightforward manner by induction:
(a) $W$ is uniquely determined by $B$ (given the colouring of the vertices and the total order on $\cal E(\Pi)$);
(b) $B$ is a forest (i.e.\ a disjoint union of trees);
(c) a black vertex can only be incident to a white edge if it is also incident to a black edge;
(d) two white vertices cannot be connected by a black edge;
(e) a black and a white vertex can only be connected by a black edge if the
black vertex is not incident to any other black edge.

\begin{figure}[t]
\begin{center}
\includegraphics{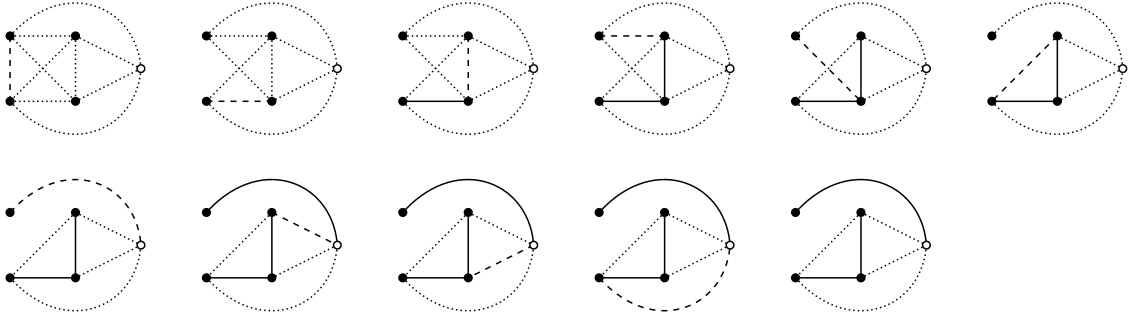}
\end{center}
\caption{The process from $\cal G_\Pi(0)$ to $\cal G_{\Pi}(e_{\mathrm{max}})$.
  Black and white vertices are depicted using black and white dots, respectively.
  Black and white edges are depicted using solid and dotted lines, respectively.
  Since the complete graph on 5 vertices has 10 edges, there are 10 steps, i.e.\ maps from a coloured graph to another.
  We start from the complete graph whose edges are all white.
  We highlight the white edge $e$ indexing the subsequent step by drawing it using dashed lines.
  In each step using \eqref{U2}, we choose to draw one of the two possible resulting graphs.
  The steps 6,8,9,10 use \eqref{U1}, and the other steps use \eqref{U2}.
  \label{fig:colouredgraphs}}
\end{figure}

Now going back to \eqref{EXp}, we find using \eqref{expansion_result} that
\begin{equation}
\E X^p \;=\; \sum_{\Pi \in \fra P_p} K(\Pi) \sum_{(B,W) \in \cal G_\Pi(e_{\text{max}})} \sum_{\f r \in \qq{N}^\Pi}J_{B,W}(\f r) \prod_{\pi \in \Pi} a_{r_\pi}^{\abs{\pi}}
\end{equation}
By property (c), if there is a black vertex that is not incident to a black edge, it is also not incident to a white edge,
and $\sum_i a_i = 0$ therefore implies
$\sum_{\f r \in \qq{N}^\Pi}J_{B,W}(\f r) \prod_{\pi \in \Pi} a_{r_\pi}^{\abs{\pi}} = 0$.
Therefore the sum over $(B,W)$ can be restricted
to graphs in which every black vertex is incident to at least one black edge.
For such graphs, $\sum_i a_i^2 \leq 1$ and $\abs{J_{B,W}} \leq 1$ imply the
bound $\absb{\sum_{\f r \in \qq{N}^\Pi}J_{B,W}(\f r) \prod_{\pi \in \Pi} a_{r_\pi}^{\abs{\pi}}} \leq 1$.
Using $|K(\Pi)| \leq \norm{Y_{1}}_{p}^{p}$ we conclude
\begin{equation} \label{EXp_est1}
\E X^p \;\leq\; \norm{Y_{1}}_{p}^{p} \sum_{\Pi \in \fra P_p} \abs{\cal G_\Pi(e_{\text{max}})}\,.
\end{equation}

It remains to estimate the sum on the right-hand side of \eqref{EXp_est1}. For fixed $\Pi$,
by (a) above it suffices to estimate the number of $B$ satisfying the remaining conditions (b)--(e).
From now on, all graph-theoretic notions always pertain to $B$, i.e.\ we discard all white edges.
Let $\Phi$ denote the set of black vertices not adjacent (by a black edge) to a white vertex:
\begin{equation*}
\Phi \;\deq\; \h{\pi \in \Pi_1 \col \text{$\pi$ is not adjacent to a vertex of $\Pi_2$}} \,.
\end{equation*}
We shall estimate the number of graphs $B$ on $\Pi$ associated with any fixed $\Phi$.
By (e), each vertex $\pi \in \Pi_1 \setminus \Phi$ has degree at most one.
With (d), this gives the upper bound $\abs{\Pi_1 \setminus \Phi}^{\abs{\Pi_2}}$ on the possible choices of
$B$ in $\Pi \setminus \Phi$. 
Moreover, by (b), $B$ is a forest on $\Phi$. 
By Cayley's formula $n^{n-2} \leq n^n$ for the number of trees on $n$ vertices
and the bound $\abs{\fra P_n} \leq (n / \log n)^n$ on the number of partitions of a set,
we find that there are at most $(\abs{\Phi}^2 / \log \abs{\Phi})^{\abs{\Phi}}$ forests on $\Phi$.
In summary, we conclude that the number of graphs $B$ associated with $\Pi$ and $\Phi \subset \Pi_1$
is bounded by $(\abs{\Phi}^2 / \log \abs{\Phi})^{\abs{\Phi}} (\abs{\Pi_1}- \abs{\Phi})^{\abs{\Pi_2}}$.

Abbreviating $k = \abs{\Pi_1}$ and $l = \abs{\Phi}$, we therefore obtain
\begin{equation*}
\sum_{\Pi \in \fra P_p} \abs{\cal G_\Pi(e_{\text{max}})}
\;\leq\; \sum_{k = 0}^p \sum_{l = 0}^k \binom{p}{k} \pbb{\frac{p - k}{\log (p - k)}}^{p - k} \binom{k}{l} \pbb{\frac{l^2}{\log l}}^l (k - l)^{p - k}
\;\leq\; \pbb{\frac{C p^2}{\log p}}^p
\end{equation*}
for some universal constant $C > 0$,
where the factor $\binom{p}{k}$ accounts for the choice of $\Pi_1$,
the factor $((p - k) / \log (p - k))^{p - k}$ for the choice of $\Pi_2$,
the factor $\binom{k}{l}$ for the choice of $\Phi$,
and the factor $(l^2 / \log l)^l (k - l)^{p - k}$ for the choice of $B$ as explained above.
Here in the last inequality we used
\begin{equation*}
\pbb{\frac{p - k}{\log (p - k)}}^{p - k}
(k - l)^{p - k} \;\leq\; \pbb{\frac{p^2}{\log p}}^{p-k}\,,\qquad
\pbb{\frac{l^2}{\log l}}^l \;\leq\; \pbb{\frac{p^2}{\log p}}^k
\end{equation*}
for $0 \leq l \leq k \leq p$.
This concludes the proof of (i).

Next, we prove (ii). 
By splitting $Y$ into its diagonal and off-diagonal entries and using Minkowski's inequality,
it suffices to prove \eqref{e:aY2} under the assumption $Y_{ii} = 0$ for all $i$.
Similarly to the proof of (i), we write
\begin{equation*}
\E \pBB{\sum_{i,j} a_i a_j Y_{ij}}^p \;=\; \sum_{\f i \in \qq{N}^{2p}} \prod_{k = 1}^{2p} a_{i_k} \, \E \prod_{k = 1}^p Y_{i_{2k -1} i_{2k}}
\;=\; \sum_{\Pi \in \fra P_{2p}} \tilde K(\Pi) \sum_{\f i \in \qq{N}^{2p}} \ind{P(\f i) = \Pi} \prod_{k = 1}^{2p} a_{i_k}\,,
\end{equation*}
where $\tilde K(\Pi) \deq \E \prod_{k = 1}^p Y_{i_{2k -1} i_{2k}}$ for any $\f i$ satisfying $P(\f i) = \Pi$
(recall the definition of $P(\f i) \in \fra P_{2p}$ above \eqref{EXp0}).
As in the proof of (i), $\tilde K(\Pi)$ is well-defined by exchangeability of $(Y_{ij})$.
By H\"older's inequality and exchangeability, we have the bound $\abs{\tilde K(\Pi)} \leq \norm{Y_{12}}_p^p$.
Now the proof of (i) following \eqref{EXp0} may be taken over verbatim, by replacing $p$ with $2p$. This concludes the proof of (ii). The proof of Proposition \ref{prop:qY} is therefore complete.

\section*{Acknowledgements}

AK was partly supported by Swiss National Science Foundation grant 144662.
HTY was partly supported by NSF grants DMS-1307444 and DMS-1606305.
HTY and RB were partly supported by a Simons Investigator Award.
The authors gratefully acknowledge the hospitality and support of the
Institute for Advanced Study in Princeton, and the National Center for Theoretical Sciences
and the National Taiwan University in Taipei, where part of this research was carried out.
The authors' stay at the IAS was supported by NSF grant DMS-1128155.

\bibliography{all}
\bibliographystyle{plain}

\end{document}